\documentclass[11pt]{imsart}
\RequirePackage{amsthm,amsmath,amsfonts,amssymb}
\RequirePackage[colorlinks,citecolor=blue,urlcolor=blue]{hyperref}
\RequirePackage[numbers]{natbib}
\RequirePackage{graphicx}

\startlocaldefs
\theoremstyle{plain}

\newtheorem{theorem}{Theorem}[section]
\newtheorem{lemma}[theorem]{Lemma}
\newtheorem{problem}{Problem}[section]
\newtheorem{proposition}{Proposition}[section]
\newtheorem{corollary}{Corollary}[section]
\theoremstyle{plain} 
\newtheorem{definition}[theorem]{Definition}
\newtheorem{example}{Example}[section]

\newtheorem{assumption}{Assumption}[section]
\newtheorem{remark}{Remark}[section]

\usepackage{amsmath}
\usepackage{amssymb}
\usepackage{amsthm}
\usepackage{dsfont}
\usepackage{algorithm}
\usepackage{algpseudocode}

\usepackage{booktabs}
\usepackage{graphicx}
\usepackage{multirow}
\usepackage{natbib}
\usepackage{siunitx}
\usepackage{caption}
\usepackage{subcaption}
\usepackage{algorithm}
\usepackage{algpseudocode}
\numberwithin{equation}{section}

\graphicspath{{../figures/}}

\newcommand\comment[1]{{\color{red}\fbox{ commented out text }}}
\renewcommand{\appendixname}{Supplements}

\def\E{\mathbb E}

\def\P{\mathbb P}

\def\R{\mathbb R}

\def\M{\mathbb M}
\def\N{\mathbb N}

\def\wh{\widehat} 
\def\wt{\widetilde}

\def\wh{\widehat }

\def\SX{{S_X}}

\def\To{\longrightarrow}
\allowdisplaybreaks

\endlocaldefs

\begin{document}

\begin{frontmatter}

\title{On the optimal prediction of extreme events}

\runtitle{Optimal prediction of extreme events}

\begin{aug}
\author[A]{\fnms{Benjamin}~\snm{Bobbia}\thanks{[\textbf{Corresponding author.}]}\ead[label=e1]{benjamin.bobbia@isae-supaero.fr}}\orcid{0009-0007-2163-4298},
\and
\author[B]{\fnms{Stilian}~\snm{Stoev}\ead[label=e2]{sstoev@umich.edu}\orcid{0000-0003-2108-5421}}
\address[A]{Département d'Ingénierie des Systêmes complexes\\ Fédération ONERA ISAE-SUPAERO ENAC, université de Toulouse\\ Toulouse, France\printead[presep={ ,\ }]{e1}}

\address[B]{Department of Statistics\\
University of Michigan\\ Ann Arbor, USA\printead[presep={,\ }]{e2}}
\end{aug}

\begin{abstract}  
The prediction of the extremely large values of a response variable $Y$ in terms of a vector of covariates 
$X=(X_i)_{i=1}^d$ is a fundamental problem arising in many scientific and engineering domains. The scarcity of data in the extremes 
makes the {\em optimal 
solution} of this problem of particular importance. 
%
The optimal predictors of such events can be explicitly characterized in just a few cases and it is of fundamental 
practical and theoretical interest to develop optimal estimators over large classes of models and predictors.
In this work, the focus is on the case where $(Y,X)$ have a multivariate regularly varying 
distribution and one seeks an optimal predictor expressed as a positive homogeneous function $h(X)$ of the covariates. 
The asymptotic prediction precision in this setting coincides with the {\em tail-dependence} coefficient $\lambda(Y,h(X))$ 
and it can be expressed as an integral functional of the associated angular measure of $(Y,X)$.  Thus, finding  
asymptotically optimal homogeneous predictors amounts to solving a variational problem.  We obtain a general solution 
to this problem, which is expressed in terms of a {\em non-extreme} conditional quantile of a tilted distribution derived 
from the angular measure.  This leads to a general inference methodology for the optimal predictors in the 
peaks-over-threshold framework from extreme value theory.  We establish the universal consistency for these estimators over large classes 
of angular measures. A general-purpose implementation of the resulting inference procedure is shown to work remarkably well 
against optimal oracle estimators, as well as in the challenging problem of extreme solar flare prediction.

\end{abstract}

\begin{keyword}[class=MSC]
\kwd[Primary ]{62G32}
\kwd{60G70}
\kwd[; secondary ]{49J50}
\end{keyword}

\begin{keyword}
\kwd{optimal prediction}
\kwd{regular variation}
\kwd{homogeneous functions}
\kwd{imbalanced classification}
\kwd{universal consistency}
\end{keyword}

\end{frontmatter}


\section{Introduction}

 In many scientific fields one has to predict rare events of the type $\{Y>y_0\}$, 
where an unobserved response random variable $Y$ exceeds a large threshold $y_0$.  Assuming that one observes
a vector of covariates $X = (X_i)_{i=1}^d$, such binary predictors can always be written in the form 
$\{h(X)> h_0\}$, for some measurable function $h$.  The optimal predictors, in a certain natural sense,
can be explicitly characterized via density ratios, as in Neyman-Pearson's Lemma  (see Section \ref{sec:NP}).  
Nevertheless, this is the beginning rather than the end of the story.  Except in a handful of cases where closed-form 
solutions exist (Examples \ref{ex:linear-paper}--\ref{ex:bivariate-max-stable}), the optimal predictors need to be 
estimated.  While the estimation of density ratios has received some attention \citep[][]{sugiyama2012dens} it remains
a methodologically and practically challenging problem \citep[][]{nguyen2010estimating,tong2013plugin}.  In this paper, we are 
interested in the regime of extremes, where the events $\{Y>y_0\}$ are rare or {\em extremely rare}.  An equivalent perspective 
to this problem is from the point of view of classification, where there is an {\em extreme imbalance} between the two classes 
of labels $\{Y\le y_0\}$ and $\{Y>y_0\}$ in the training data.  Imbalanced classification problems arise naturally and are 
ubiquitous in science, medicine, engineering, etc \citep[][]{bobra2015sola,irvin2019chexpert,valavi2022predictive}. 
Imbalanced classification has received considerable theoretical attention 
\citep[see, e.g.][]{rigollet2011neyman,tong2013plugin,tong2020parametrics}.
Nevertheless, the theoretical study of the optimal {\em extreme event} prediction or, equivalently, binary classification in 
the {\em extremely imbalanced} regime has received relatively less attention.  To the best of our knowledge, 
\cite[][]{JSsS2018} is the first paper that has addressed this problem by arguing that the traditional empirical risk minimization 
approaches do not enjoy provable guarantees in the regime of extremes. They proposed a principled 
solution based on extreme value theory by considering risk minimization after conditioning on asymptotically extreme events,
which has led to an active area of research \citep[see, e.g.][and the references therein]{clemencon2023concentration,aghbalou2024sharp,clemencon2025regression,decarvalho2025cascading,clemencon2026weak}. 
The recent work of \cite{Legrand03072025} provides a framework for the statistical analysis of the extreme event 
classifiers in the setting of multivariate as well as hidden regular variation, which is related to our context; \cite{decarvalho2025cascading} introduce a probability of cascade formulation that quantifies probabilities of
extremes as a function of covariates and proposed a Kolmogorov-Arnold-Network inference framework; and last but not least, 
the very recent work in \cite{leroux2026oodquantile} develops a general conditional empirical risk minimization framework, 
in the regime of extremes, which can be used to address the extreme event prediction problem. For a comprehensive account 
of the state-of-the-art in the field, see \cite{clemencon2026weak} and the references therein.\\

In this paper, however, our main focus on understanding and characterizing the {\em optimal extreme event prediction} 
rather than on the study of specific imbalanced classifiers.  While an agnostic (conditional) 
empirical risk minimization approach \citep[as in][]{JSsS2018} to optimal extreme event prediction is possible 
and powerful, we focus here on the {\em exact} characterization of optimal predictors in certain general settings when 
$Y$ and $X$ are jointly regularly varying. These characterizations lead to novel inference methodologies, which 
have the potential to be more efficient since they aim to estimate optimal predictors. Our implementations of the resulting
estimators indeed show near-optimal performance comparable to the available oracles in a number of settings.  This is the 
main methodological distinction of our work in comparison with the existing approaches that rely on powerful but general 
statistical learning ideas.  \\

An important starting point of our investigation is the framing of the notion of optimality. Colloquially, 
prediction of the unobservable event $\{Y>y_0\}$ via the observed covariates amounts to raising an 
{\em alarm} when an event $\{h(X)> h_0\}$ occurs.  Let the thresholds $y_0:=F_Y^\leftarrow(p)$ and $h_0:=F_{h(X)}^\leftarrow(q)$ 
be generalized quantiles of the distributions of $Y$ and $h(X)$, respectively, for some $p,q\in (0,1)$.  Then, under mild 
regularity conditions, the {\em alarm rate} equals $1-q = \P[h(X)> F_{h(X)}^\leftarrow(q)]$ and the {\em event rate} is 
$1-p = \P[Y>F_Y^\leftarrow(p)]$. Given certain alarm and event rates, the {\em optimal  predictors} are the ones that maximize 
the {\em precision:}
$$
\P[ Y> F_Y^\leftarrow(p)\, |\, h(X)> F_{h(X)}^\leftarrow(p)].
$$
It is often natural to seek {\em calibrated predictors} where $p=q$, i.e., the alarm and event rates coincide.
Then, in the regime of extremes, as $p\uparrow 1$, the {\em asymptotic precision} is precisely the tail-dependence 
coefficient 
\begin{equation}\label{e:lambda-intro}
\lambda(Y,h(X)) = \lim_{p\uparrow 1} \P[ Y> F_Y^\leftarrow(p)\, |\, h(X)> F_{h(X)}^\leftarrow(p)],
\end{equation}
which is a well-known quantity in extreme value theory \citep[see e.g.][]{joe2015dependence}. 

In the general setting when $Y$ and $X$ are {\em jointly regularly varying}, and the predictors $h$ are continuous and
positive homogeneous functions, the tail dependence coefficient $\lambda(Y,h(X))$ is well-defined, and it can be expressed
as an integral functional with respect to the angular measure (cf Propositions \ref{p:rv-via-h} and \ref{p:lambda-via-U-Theta}).
Specifically,
\begin{equation}\label{e:lambda-expression-intro}
\lambda(Y,h(X)) = \frac{1}{\E[U]} \E [ U \wedge (1-U) h(\Theta) ],
\end{equation}
where the random variable $U\in [0,1]$ and the random vector $\Theta = (\Theta_i)_{i=1}^d$ characterize the joint extremal
behavior of $Y$ and $X$ as follows:
$$
\frac{(Y,X)}{\tau(Y,X)} \, | \, \{\tau(Y,X)>t\} \stackrel{d}{\to} (U,(1-U)\Theta),\ \ \mbox{ as }t\to\infty.
$$
Here $\tau(y,x):= y_+ + \tau_X(x)$ is a positive $1$-homogeneous gauge function codifying the notion 
of multivariate regular variation (Definition \ref{d:tau-rv}) and $\Theta$ is a random vector taking values in the generalized
unit sphere $S_{\tau_X} =\{x \in \R^d\, :\, \tau_X(x) = 1\}$.

In view of \eqref{e:lambda-intro}, we seek homogeneous predictors that maximize  
the tail-dependence coefficient $\lambda(Y,h(X))$ and refer to them as {\em asymptotically optimal}. 
Given the joint distribution of $(U,\Theta)$, Relation \eqref{e:lambda-expression-intro} shows that finding 
asymptotically optimal homogeneous predictors amounts to solving a variational problem.  Namely, finding a homogeneous 
function $h$ in an infinite-dimensional class of functions, that maximizes the integral functional in 
\eqref{e:lambda-expression-intro}, under the {\em asymptotic calibration constraints}:
$$
\P[ Y>t ] \sim \P[h(X)>t],\ \ \mbox{ as }t\to\infty.
$$
Here and throughout the paper by $a_t\sim b_t$ we mean asymptotic equivalence in the sense that $a_t/b_t \to 1$ as $t\to\infty$.

The first main contribution of the paper is a solution of the aforementioned variational problem (Theorem \ref{thm:final}).  Then,
in Section \ref{sec:Breiman-models}, we demonstrate how this solution yields asymptotically optimal homogeneous predictors in 
the context of the general class of Breiman type models.  It turns out that homogeneous optimal predictors can be expressed in terms of 
the homogenous function
$$
g^{\rm (opt)}(x):= \tau_X(x) \cdot \frac{q_\alpha(x/\tau_X(x))}{ 1- q_\alpha(x/\tau_X(x))},
$$
where $\theta \mapsto q_\alpha(\theta)$ is the conditional $\alpha$-level quantile of the tilted conditional
probability distribution $b(\theta)^{-1} (1-u) p_{U|\Theta}(du|\theta)$, where $p_{U|\Theta}(du |\theta)$ is 
the regular conditional distribution of $U$ given $\Theta$, and $b(\theta)$ is a normalization constant.

All homogeneous functions of the above form are solutions to the optimal variational problems with constraints
governed by the parameter $\alpha\in (0,1)$ (Theorem \ref{thm:final}).  Thus, to find an optimal and calibrated predictor, in practice,
one merely needs to estimate the quantile functions $q_\alpha$ {\em and} find $\alpha$ so that 
$\P[g^{\rm (opt)}(X)>t] \sim \P[Y>t],\ t\to\infty$.
Inference theory for the optimal homogeneous predictors in the peaks-over-threshold framework developed in 
Section \ref{sec:inference}.  In this setting, by using a contiguity argument, we established the universal consistency of these
optimal homogeneous predictors in Theorem \ref{thm:univ_consitency} under mild regularity conditions.  This result may be
viewed as a counterpart of the Stone universal consistency theory in the context of extremes \citep[][]{devroye1996probabilistic}.

In Section \ref{sec:methodology}, we discuss the implementation of the estimators and illustrate their performance 
in comparison with optimal oracle predictors, both in the context of discrete and continuous angular measures.  A general-purpose,
practical, and nearly tuning free implementation of the optimal homogeneous predictor methodology was built by using powerful 
{\em quantile regression forests} \citep[][]{meinshausen2006quantile,beirlant:goegebeur:teugels:segers:2004}, which were
used to estimate the quantile functions $q_\alpha$.   The resulting predictors nearly match the performance of the oracle predictors in 
a variety of simulated models (Figures \ref{fig:boxplot_discrete} and \ref{fig:boxplot_spec_continuous}).   Furthermore, their 
application to the open challenge of predicting extreme X-class solar flare events achieves state-of-the-art results 
without sophisticated feature engineering and tuning (see Section \ref{sec:solar-flares}). 
R code and R Shiny Apps illustrating the methodology are freely available 
on \cite{optXpred:2026,optXpred_shiny:2026}.

The paper contains a variety of additional contributions of independent interest. First, in Proposition \ref{p:archimedean} we provide a
characterization of the optimal (non-asymptotic) predictors for Archimedean copula models; A general contiguity argument 
justifying the peaks-over-threshold based inference was developed (Section \ref{sec:PoT-and-Contiguity}); A flexible family of
Pareto-Breiman models was studied where regular variation holds in a stronger, total-variation sense (Proposition \ref{p:total-Breiman}).
The case of Pareto-Dirichlet models has a particularly elegant solution (Proposition \ref{p:Pareto-Dirichlet}), while a large 
class of spectrally discrete models offer insights to further phenomena and applications (Section \ref{sec:numerical-examples}).  
A curious phenomenon of {\em perfect} asymptotic precision arises for certain spectrally discrete models, 
when all underlying latent factors influencing $Y$ are involved in forming the covariates $X$ (see Section \ref{sec:explicit_examples}). \\

{\em The rest of the paper is structured as follows.} In Section \ref{sec:Neyman-Pearson}, we begin with the formulation of
the optimal prediction framework and present a general characterization results of optimal predictors in terms of density ratios
(Theorem \ref{thm:general-opt}). This naturally leads to several examples where closed-form optimal (non-asymptotic) 
predictors can be derived (i.e., Sections \ref{sec:NP} and \ref{sec:archimedean-statements}). 
In Section \ref{sec:extremal-precision}, we conclude the preliminaries by relating the notion of asymptotic optimality 
and tail dependence.

In Section \ref{sec:opt_ext_pred}, we present the formulation and solution to the variational problem on optimal homogeneous prediction,
where the main result is Theorem \ref{thm:final}.  Section \ref{sec:inference} presents an inference theory for these optimal 
predictors in the context of peaks-over-threshold, where we carefully treat the distribution shift problem via contiguity argument.
The main result is Theorem \ref{thm:univ_consitency}, which shows that the so-constructed estimators achieve the asymptotically optimal
precision.  This can be seem as a counterpart to the celebrated universal consistency results of Stone.  Section \ref{sec:methodology}
briefly discusses methodological aspects and illustrates the non-asymptotic, practical performance of the optimal homogeneous predictors 
with simulations.  Further technical details and some lengthy proofs are given in the Supplement.

\section{Preliminaries}\label{sec:Neyman-Pearson}

In this section, we introduce the notion of optimal event prediction considered in this paper.  
A general characterization result and several examples of closed-form optimal predictors are discussed, which may be
of independent interest and can serve as benchmarks.  The {\em extremal} counterpart to optimal event prediction is 
then formulated and related to asymptotic tail dependence.  This sets the stage for the problem formulation 
discussed in Section \ref{sec:problem}.

\subsection{Optimal event prediction: A Neyman-Pearson perspective}\label{sec:NP}

Consider a vector of covariates $X= (X_i)_{i=1}^d$ taking values in $\R^d$ and an 
unobserved positive response random variable $Y$.  We would like to optimally predict the 
event $\{Y > y_0\}$ in terms of $X$.  All such predictors are of the form $\{h(X)>h_0\}$,
for some Borel measurable function $h$ and a suitable calibration threshold $h_0\in \R$.  

\begin{definition} {\em (i)} We say that the predictor $\{h(X)>h_0\}$ for $\{Y>y_0\}$, is 
calibrated at level $q$, or $q$-calibrated, if
$$
\P[ h(X)>h_0] = 1-q,\ \ \mbox{ for some }q\in (0,1).
$$
The predictor is said to be balanced if it is calibrated at the same level as the 
target event, i.e., $\P[ Y >y_0] = \P[h(X)> h_0].$

{\em (ii)} A $q$-calibrated predictor 
$\{h(X)>h_0\}$ of $\{Y>y_0\}$ is said to be optimal, if
$$
 \P[Y>y_0 | h(X)>h_0] \ge \P[ Y > y_0 | g(X) >g_0 ],
$$
for all other $q$-calibrated predictors $\{g(X)> g_0\}$. 

The conditional probability $\P[Y>y_0| h(X)>h_0]$ will be referred to as the {\em precision} of the event predictor. 
Thus, the optimal $q$-calibrated predictors have maximal precision among all other $q$-calibrated predictors.
\end{definition}

The following result is a restatement of Theorem 2.1 in \cite{verma:stoev:chen:2024}.  It 
provides a general characterization of the optimal calibrated predictors.

\begin{theorem}\label{thm:general-opt}  Let $\nu(dx)$ be a $\sigma$-finite Borel measure on 
$\R^d$. Assume that $X|Y> y_0$ and $X|Y\le y_0$ have densities $f_0(x)$ and $f_1(x)$,
respectively, with respect to $\nu(dx)$.  Consider the density ratio:
\begin{equation}\label{e:thm:general-opt}
 r(x):= \frac{f_0(x)}{f_1(x)},
\end{equation}
where $r(x)$ is interpreted as $\infty$ if $f_1(x)=0$. Let $F_{r(X)}(x)
= \P[ r(X)\le x]$ be the distribution function of $r(X)$ and 
$F_{r(X)}^\leftarrow(q):=\inf\{ x\, :\, F_{r(X)}(x) \ge q\}$. For all $q\in F_{r(X)}(\R)$,
$\{r(X) > F_{r(X)}^{\leftarrow}(q)\}$ is an optimal $q$-calibrated predictor of $\{Y>y_0\}$. 
\end{theorem}

Theorem \ref{thm:general-opt} leads to closed form optimal predictors in a handful but important
cases.

\begin{example}[Additive heteroskedastic models]  \label{ex:linear-paper} 
 Let $X=(X_i)_{i=1}^d$ and $\epsilon$ be independent, where $\epsilon$ has a density with respect to the Lebesgue measure 
 on $\R$. Define
 \begin{equation}\label{e:ex:linear-paper}
 Y = g(X) + \sigma(X) \epsilon,
 \end{equation}
 for some deterministic measurable functions $g$ and $\sigma$ such that
 $\sigma(X)>0$ almost surely.  Then, by Theorem 2.2 in \cite{verma:stoev:chen:2024},
 for all $y_0$ and $\tau$, 
 $$
 \{ h_{y_0} (X) >\tau \} := \Big\{ \frac{g(X) - y_0}{\sigma(X)}>\tau\Big\} 
 $$
 is a $q$-calibrated optimal predictor of $\{Y > y_0\}$, where $1-q = \P[ h(X)>\tau]$.  
 \end{example}

The above example leads to explicit characterizations of the optimal event-predictors
as well as their precisions in the general class of linear models \citep[][]{verma:stoev:chen:2024}.
In particular, since Gaussian vectors can always be cast in the form of \eqref{e:ex:linear-paper}, 
where $g$ is a linear function and $\sigma$ is non-random, Example \ref{ex:linear-paper} immediately
implies  the following.

 \begin{example}[Gaussian copula] \label{ex:Gaussian-paper}
 Let $(Y,X_1,\cdots,X_d)$ follow a Gaussian copula model with 
 continuous marginal distribution functions $F_Y$ and $F_{X_i},\ i=1,\cdots,d$. Then, 
 for all  $p,q\in (0,1)$, we have 
 that $\{h(X) > F_{h(X)}^{\leftarrow}(q)\}$ is a $q$-calibrated optimal predictor for 
 $\{Y>F_Y^\leftarrow(p)\}$, where
 $
  h(X) = \sum_{i=1}^d a_i \Phi^{-1}(F_{X_i}(X_i)),
 $
 and $\Phi$ stands for the standard normal cdf \citep[Corollary 2.1 in][]{verma:stoev:chen:2024}.
 \end{example}

\begin{example}[bivariate max-stable models]\label{ex:bivariate-max-stable} 
Let now $d=1$ and $(Y,X)$ follow a bivariate extreme value distribution having a 
joint density with respect to the Lebesgue measure in $\R^2$.  Using Theorem 
\ref{thm:general-opt}, in Proposition \ref{p:max-stable} below, it is shown that 
$\{X>F_X^\leftarrow(q)\}$ is a $q-$calibrated 
optimal predictor of the event $\{Y>F_Y^\leftarrow(p)\}$, for all $p,q\in (0,1)$. 
\end{example}

The above example is rather natural since max-stable laws are positively dependent
and hence the extremes of $Y$ go with the extremes of $X$.  See also Example 
\ref{ex:bivariate-max-stable-contd}, below.

\subsection{Further examples: Archimedean copula} \label{sec:archimedean-statements}

The closed-form expressions of optimal predictors as in the previous examples are rare. 
In this section, we present such an expression for the broad class of Archimedean copula 
with completely monotone generators. It can be of independent interest 
and shall serve as a benchmark for our optimal prediction inference methodology.

Suppose that the random vector
 $
  (Y,X)=(Y,X_1,\cdots,X_d)
 $ 
 has continuous marginal cdf's $F_Y(t) = \P[Y\le t]$ and $F_{X_i}(t):= \P[X_i\le t],\ t\in \R,\ i=1,\cdots,d$ and the copula
 $$
  C(y,x):= \P[F_Y(Y)\le y, F_{X_i}(X_i)\le x_i,\ i=1,\cdots,d],\ \ y\in (0,1), x=(x_i)_{i=1}^d\in (0,1)^d,
 $$
 The copula $C$ is said to be Archimedean if
 \begin{equation}\label{e:archimedean}
   C(y,x) = \psi\Big( \psi^{-1}(y)+ \psi^{-1}(x_1)+\cdots+\psi^{-1}(x_d) \Big),
 \end{equation}
 where  $\psi:(0,\infty) \to (0,1)$ is a decreasing function referred to as the {\em generator} of the copula.  
 
 It is known that $C$ is a valid copula provided $\psi$ is a  $(d+1)$-monotone function \citep[see, e.g.][]{mcneil:neslehova:2009}. 
 %
 %
 Here, we focus on the important particular case of Archimedean copula, where the generator $\psi$ is completely 
 monotone and hence $(d+1)$-monotone for all $d$.  Such functions have the representation
 \begin{equation}\label{e:Bernstein}
  \psi(u) = \int_0^\infty e^{- ux } \mu(dx),\ \ u>0
 \end{equation}
 for some unique Borel measure $\mu$ on $[0,\infty)$ \citep[see e.g., Theorem 1.4 in ][]{schilling2012bern}.
 Since we deal with copulas, the measure $\mu$ will not charge $\{0\}$. 

 \begin{proposition}\label{p:archimedean} 
 Let $(Y,X_1,\cdots,X_d)$ be a random vector in $\R^{1+d}$ with continuous marginal cdfs $F_Y$ and $F_{X_i},\ i=1,\cdots,d$, respectively,
 and the Archimedean copula $C$ as in \eqref{e:archimedean}, where the generator $\psi$ is completely monotone, i.e., as in \eqref{e:Bernstein}.  
 Then, for all $\tau\in (0,1)$ and $y_0\in \R$, the events
 $$
  \Big\{ \psi \Big(  \psi^{-1}(F_{X_1}(X_1))+\cdots + \psi^{-1}(F_{X_d}(X_d))  \Big) > \tau \Big\}
 $$
 are optimal predictors of the events $\{Y>y_0\}$ in terms of $X$.     
 \end{proposition}

 The proof is given in Section \ref{sec:archimedean}, below.
 This result furnishes a complete characterization of optimal event prediction 
 in an interesting family of copula models.  We will illustrate it further with an important
 Archimedean copula arising in extreme value theory.

\begin{example}[Gumbel copula]\label{ex:gumbel} Consider the function
$$
\psi(u) = e^{-u^{1/\beta}} = \int_{0}^\infty e^{-ux}f_{1/\beta}(x)dx,\ \ \beta > 1
$$
where $f_{1/\beta}(x)$ is the probability density of a $(1/\beta)-$stable subordinator.  Notice
that $\psi$ is completely monotone since it is the Laplace transform of a probability distribution on $[0,\infty)$
for $\beta \in (1,\infty)$. (The case $\beta =1$ is also trivially covered since the corresponding 
copula is the independent one.)  The copula $C$ in \eqref{e:archimedean} then becomes the so-called 
{\em Gumbel copula}:
$$
C(v,u) = \exp\Big\{ -\Big(\log(1/v)^\beta + \log(1/u_1)^\beta + \cdots + \log(1/u_d)^\beta  \Big)^{1/\beta} \Big\},
$$
for $v\in (0,1)$ and $u = (u_i)_{i=1}^d \in (0,1)^d$. By attaching
standard $1$-Fr\'echet marginals to the resulting Archimedean copula, we 
obtain the so-called multivariate logistic max-stable distribution model for
the random vector $(Y,X_1,\cdots,X_d)$:
\begin{equation}\label{e:logistic-old}
     \P( Y\le y, X\le x) = \exp\Big\{ -\Big( \frac{1}{y^\beta} + \frac{1}{x_1^\beta} +\cdots + \frac{1}{x_d^\beta}  \Big)^{1/\beta} \Big\},
\end{equation}
$x_i>0,\ y>0$. Proposition \ref{p:archimedean} implies that 
for all $p,q\in (0,1)$, an optimal predictor of $\{Y>F_Y^{-1}(p)\}$ is of the form:
$\{h(X) > F_{h(X)}^{-1}(q)\}$, where 
$
 h(X) =(\sum_{i=1}^d 1/X_i^{\beta})^{-1/\beta}.
$
Its precision can be obtained numerically. In the following section, we will consider the notion of 
asymptotically optimal precision and obtain a formula for it for these copula.
\end{example}

\subsection{Tail dependence and optimal extreme event prediction} \label{sec:extremal-precision}

In this paper, we are primarily interested in predicting {\em extreme events} of the form 
$$
\{Y> F_Y^\leftarrow(p)\},
$$
as the level $p$ becomes extreme, i.e., as $p\uparrow 1$.  
For simplicity, assume that $p\in F_Y(\R) = \{F_Y(y),\ y\in \R\}$
and focus on the natural case of  {\em balanced predictors}, where one has
\begin{equation}\label{e:balanced-predictors}
 1-p = \P[ Y> F_Y^\leftarrow(p) ] = \P [ h(X) > F_{h(X)}^\leftarrow(p) ].
\end{equation}

In this case, the conditional probability of a true positive alarm, referred to as the
{\em precision} of the predictor, is:
\begin{equation}\label{e:lambda-p}
\lambda_p(Y,h(X))  :=
\P[Y>F_Y^\leftarrow(p)\, |\, h(X)>F_{h(X)}^\leftarrow (p)].
\end{equation}
Observe that, as the level $p$ becomes extreme, i.e., $p\uparrow 1$, the above
precision recovers the well-known notion of {\em tail-dependence coefficient}:
\begin{equation}\label{e:lambda-tdep-def}
\lambda(Y, h(X)) := \lim_{p\uparrow 1} \P[Y>F_Y^\leftarrow(p)\, |\, h(X)>F_{h(X)}^\leftarrow (p)],
\end{equation}
whenever the limit exists \citep[see, e.g.,][]{joe2015dependence,jansen:neblung:stoev:2023}. 
The random variables $Y$ and $h(X)$ will be referred to as 
{\em tail-independent} if $\lambda(Y, h(X))=0$; otherwise we will sometimes call them tail-dependent.

Relation \eqref{e:lambda-tdep-def} motivates the following notion of
extremal optimality \citep[Definition 2.3 in][]{verma:stoev:chen:2024}.

\begin{definition}\label{d:opt-pred}
Let ${\cal C}_p(X)$ denote the class of all Borel functions $g$
such that $\P[ g(X) > F_{g(X)}^\leftarrow(p)] = 1-p$.  That is, ${\cal C}_p(X)$ is 
the set of functions $g$ that can serve as $p$-calibrated predictors.\\

{\em (i)} The {\em optimal extremal precision} for predicting $Y$ in terms of $X$ is defined as:
\begin{equation}\label{e:opt-ext-prec}
\lambda^{{\rm (opt)}} (Y,X):= \limsup_{p\uparrow 1}\sup_{g\in {\cal C}_p(X)} \Big\{
\lambda_p(Y,g(X))
\Big\}. 
\end{equation}

{\em (ii)} A random variable $h(X)$ is said to be an optimal extremal predictor
for $Y$, if $h\in {\cal C}_p(X)$ and its asymptotic precision attains the optimal extremal precision in  \eqref{e:opt-ext-prec}.  That is,
if 
$$
\lambda^{\rm (opt)} (Y,X) = \limsup_{p\uparrow 1} \lambda_p(Y,h(X)) 
$$
In particular, if the tail-dependence $\lambda(Y,h(X))$ coefficient 
exists, we also obtain $\lambda^{{\rm (opt)}}(Y,X) = \lambda(Y, h(X)).$
\end{definition}

Relation \eqref{e:opt-ext-prec} indicates that $\lambda^{{\rm (opt)}}(Y,X)$ is 
the best asymptotic precision that can be achieved among all balanced 
calibrated predictors of the extreme events $\{Y> F_Y^\leftarrow(p)\}$.  
Theorem \ref{thm:general-opt} implies that 
the density ratios in \eqref{e:thm:general-opt} yield  
fixed-$p$ as well as asymptotically optimal predictors. This is the gold-standard in
extreme event prediction, which can in fact be achieved in a number of interesting models
(see the Examples \ref{ex:linear-paper}-\ref{ex:gumbel} above and their follow-ups in 
\ref{ex:Gaussian-contd}--\ref{ex:logistic-contd} below).  

In general, however, the density ratios are hard to compute and estimate 
\citep[][]{sugiyama2012dens}. Thus, it is challenging to characterize the optimal extremal 
predictors and their extremal precisions in full generality.  Therefore, it is of interest 
to limit the notion of extremal optimality to a class of predictors.  


\begin{definition} Consider a class ${\cal G}$ of measurable functions $g:\R^d\to\R$. 
We shall say that $h$ is a ${\cal G}$-optimal extremal predictor for $Y$ if the 
tail-dependence coefficient $\lambda(Y,h(X))$ exists and 
$$
\lambda(Y,h(X)) = \lambda_{\cal G}^{{\rm (opt)}}(Y,X)
:=\sup_{g\in {\cal G}} \Big( \limsup_{p\uparrow 1} \lambda_p(Y,g(X))\Big), 
$$
where $\lambda_p$ is as in \eqref{e:lambda-p}.
\end{definition}


 We conclude this section by following-up on Examples \ref{ex:linear-paper} -- \ref{ex:gumbel},
 where the optimal extremal predictors and their precisions can be derived.  
 The next section focuses on ${\cal G}$-optimal prediction for the class of 
 homogeneous functions under the framework of joint regular variation for $Y$ and $X$.


 \begin{example}\label{ex:Gaussian-contd}
 When $(Y,X) = (Y,X_1,\cdots,X_d)$ follow a Gaussian copula such that $Y$ is not a deterministic function of 
 $X$, then the covariates in $X$ have vanishing precision in predicting the extremes of $Y$, i.e., 
 $\lambda^{\rm (opt)}(Y,X) = 0$. Indeed, upon applying increasing component-wise deterministic 
 transformations to $Y$ and $X_i,\ i=1,\cdots,d$, we may assume that $(Y,X)$ is Gaussian with standard normal 
 marginals.  Hence, for some $a=(a_i)_{i=1}^d\in \R^d$, we have $Y =  a^\top X +  Z$, where 
 $Z\sim {\cal N}(0,\sigma^2),\ (\sigma>0)$ is independent 
 from $X$.  By a classical result of Sibuya \citep[Theorem 3][]{sibuya:1960}, it follows that
 the jointly Gaussian variables $a^\top X$ and $Y$ are asymptotically independent,
 and hence by Example \ref{ex:Gaussian-paper},  
 $
 \lambda^{\rm (opt)}(Y,X) = \lambda(Y, a^\top X) = 0.
 $
 \end{example}

\begin{example}\label{ex:bivariate-max-stable-contd} Let $(Y,X)$ be a bivariate 
max-stable vector having a Lebesgue density.  Then, by 
Example \ref{ex:bivariate-max-stable},
$
\lambda^{\rm (opt)}(Y,X) = \lambda(Y,X),
$
i.e., the optimal extremal precision is precisely the tail-dependence coefficient.
\end{example}

 \begin{example}[Multivariate logistic max-stable law]\label{ex:logistic-contd} Recall Example \ref{ex:gumbel},
 where the joint distribution of $(Y,X)$ is multivariate max-stable with standard $1$-Fr\'echet marginals
 and the Gumbel copula in \eqref{e:logistic-old}.  We know that $h(x):= (\sum_{i=1}^d x_i^{-\beta} )^{1/\beta}$ provides
 an optimal prediction function for this model.  As shown in Corollary \ref{c:logistic}, 
 the optimal extremal precision is given by:
 \begin{equation}\label{e:logistic-fla}
 \lambda^{\rm (opt)}(Y,X) = \lambda(Y,h(X)) = \E\Big[ \min\Big\{ \Gamma_1^{-1/\beta}/c_{1,\beta},\ \Gamma_d^{-1/\beta}/c_{d,\beta} \Big\} \Big], 
 \end{equation}
 where $\Gamma_1\sim {\rm Gamma}(1,1)$ and $\Gamma_d \sim {\rm Gamma}(d,1)$ are 
 independent Gamma-distributed random variables and where $c_{d,\beta} = 
 \Gamma(d-1/\beta)/\Gamma(d)$.
 \end{example}

While elegant, the above examples are the only a handful of cases we know of, where the optimal (extremal) predictors
and their (asymptotic) precisions can be obtained in closed form.  This motivates the main pursuit of the
present paper on characterizing optimal extremal predictors over a smaller but still very 
rich class of homogeneous predictors.

\section{A characterization of the optimal homogeneous predictors}
\label{sec:opt_ext_pred}

In this section, we present the core contribution of this paper. We start with the main framework that 
relies on an assumption of joint regular variation
between the response and the predictors.  In this context, we cast the main problem of 
finding asymptotically optimal homogeneous predictors as a calculus of variations problem 
relative to the angular measure. The general solution of this problem is presented in 
Theorem \ref{thm:final}. The section concludes with a rejoinder, where the solutions to the 
general variational problem are related to optimal homogeneous predictors in the context of the 
generalized Breiman models.

\subsection{Joint regular variation and problem formulation} \label{sec:joint-RV} \label{sec:problem}

 Let $Z$ be a random vector taking values in $\R^k$ and let 
$\tau :\R^k\to \R_+ :=[0,\infty)$ be a non-negative continuous 
$1$-homogeneous function, 
i.e., $\tau(c\cdot x) = c\tau(x),\ \forall c\ge 0,\ x\in \R^k$.

\begin{definition}[$\tau$-regular variation]\label{d:tau-rv} We say that $Z$ is 
$\tau$-regularly varying if, there exists a 
positive sequence $\{a_n\}$, and constants $c_Z>0$ and $\alpha>0$, such that
for all $t>0$,
\begin{equation}\label{e:d:tau-rv}
n\P[ \tau(Z) > t a_n ] \To c_Z t^{-\alpha},\ (n\to\infty)
\ \mbox{ and }\ 
\frac{Z}{\tau(Z)} \, \vert\, \tau(Z)>r \stackrel{d}{\To} \Theta_Z,\ \ (r\to\infty)
\end{equation}
for some random variable $\Theta_Z$ taking values in 
$$
S_\tau:= \{ x\in\R^k\, :\, \tau(x) = 1\}.
$$
In this case, we shall write $Z\in {\rm RV}_\alpha(\R^k,\{a_n\},\ c_Z, \tau, \sigma)$,
where $\sigma$ is the probability distribution of $\Theta_Z$.
\end{definition}

For more details and an equivalent definition, see Section \ref{sec:RV} below. In particular, the first convergence in
Relation \eqref{e:d:tau-rv} holds if and only if $x\mapsto \P[\tau(Z) > x]$ is a regularly varying function at infinity
with exponent $-\alpha$.  Therefore, \eqref{e:d:tau-rv} is equivalent to
\begin{equation}\label{e:d:tau-rv-b(t)}
b(t)\P[ \tau(Z) > t ] \To c_Z t^{-\alpha},\ (n\to\infty)
\ \mbox{ and }\ 
\frac{Z}{\tau(Z)} \, \vert\, \tau(Z)>r \stackrel{d}{\To} \Theta_Z,\ \ (r\to\infty)
\end{equation}
where $b(t)\sim L(t)t^\alpha$, as $t\to\infty$, for some slowly varying function $L(\cdot)$ such that
$b(a_n)\sim n$, as $n\to\infty$. Thus, for an $X$ as in Definition \ref{d:tau-rv}, 
the function $\{b(t)\}$ in \eqref{e:d:tau-rv-b(t)} is unique modulo asymptotic equivalence
\citep[see, e.g. Propositions 1.1.8 and 1.2.2 in][]{kulik2020heav}.
When we need \eqref{e:d:tau-rv-b(t)}, we shall equivalently  write $X\in {\rm RV}_\alpha(\R^k,\{a_n\},b(\cdot),\ c_Z,\tau,\sigma)$.

\begin{remark}
The notion of $\tau$-regular variation is equivalent to the notion of regular variation on the 
cone $\R^k \setminus\{\tau =0 \}$ \citep[see e.g.][]{scheffler:stoev:2017,dombry:ribatet:2015}.
It only focuses on the extremal behavior of $Z$ over the open cone 
$\{\tau>0\}$, where {\em extremeness} is measured in terms of the positive homogeneous 
loss function $\tau$. For a complete and general treatment, see the recent monograph of
\cite{basrak:milincevic:molchanov:2025}.  The ability to choose different  loss-functions $\tau$ 
allows one to be more flexible and study the dependence in the extremes relative
to the relevant loss somewhat in the spirit of {\em hidden regular variation}
\citep[][]{resnick:2024}. 
\end{remark}

The following is a key result that will allow us to formulate the optimal extreme 
event prediction problem in terms of a calculus of variations problem for the angular 
measure. See also Proposition 2.5 in \cite{jansen:neblung:stoev:2023} and Theorem 2.1 
in \cite{Dyszewski2020}, for similar results.  

\begin{proposition}\label{p:rv-via-h} Let $Z \in RV_\alpha(\R^k,\{a_n\},b(\cdot),c_Z,\tau,\sigma)$. Let 
also ${\cal H}_+(\R^k,\tau)$ denote the class of all non-negative
continuous $1$-homogeneous functions $h:\R^d\to \R_+$  such that $\{h>0\} \subset \{\tau>0\}$. 
For all $h\in {\cal H}_+(\R^k,\tau)$ such that
\begin{equation}\label{e:p:rv-via-h-assumption}
\|h\|_{\infty,S_\tau} := \sup_{\theta\in S_\tau} |h(\theta)| <\infty,
\end{equation}
we have
\begin{equation}\label{e:p:rv-via-h}
\lim_{n\to\infty} n\P[ h(Z) > a_n ]  = \lim_{t\to\infty} b(t) \P[h(Z)>t] = c_Z \sigma(h),
\end{equation}
where 
$$
\sigma(h) := \E[h(\Theta_Z)^\alpha] = \int_{S_\tau} h^\alpha(\theta) \sigma(d\theta).
$$
\end{proposition}

\noindent For the sake of completeness, the proof of Proposition \ref{p:rv-via-h} is given in Section \ref{sec:proofs:sec:joint-RV}.

\medskip
In the rest of the section, we will specialize the above regular variation 
concept to the setting considered in this paper by considering a specific $\tau$.
Namely, let
$$
Z = (Y,X),
$$
where $Y$ is a positive response variable and $X = (X_i)_{i=1}^d$ is a $d$-dimensional vector of 
covariates. We will assume that $Y$ and $X$ are {\em jointly regularly varying}, which will be 
made precise next.  For simplicity and without loss of generality, we will assume the exponent of regular variation equals one, i.e.: 
$$\alpha=1.$$ 

\begin{assumption}[joint regular variation] \label{a:X-Y-joint} Let $\tau_X:\R^d \to \R_+$ be a 
continuous $1$-homogeneous function, $\tau_Y:\R \to \R_+$ be the function
$\tau_Y(y):= y_+:= \max\{0,y\}$, and let
$$
\tau(y,x):= \tau_Y(y) + \tau_X(x)  = y_++\tau_X(x).
$$
We shall say that $Y$ and $X$ are jointly regularly varying if
$$
Z = (Y,X) \in RV_1(\R_+\times \R^{d}, \{a_n\}, 1, \tau, \sigma).
$$ 
and
\begin{equation}\label{e:a:X-Y-joint}
\sigma(\{1\}\times\{\tau_X=0\}) < 1 \ \ \mbox{ and } \ \ \sigma(\{0\} \times \{\tau_X = 1\}) <1.
\end{equation}
\end{assumption}

\begin{remark}\label{r:joint-RV} The joint regular variation assumption for $Y$ and $X$ amounts essentially to 
assuming that the vector $Z=(Y,X)$ is $\tau$-regularly varying.  The additional conditions in  \eqref{e:a:X-Y-joint}
mean that both $Y$ and $X$ are {\em marginally} regularly varying with the same normalizing 
sequence $\{a_n\}$ in the sense that, for all $t>0$,
$$
n\P[ Y > a_n t]\to_{n\to\infty} c_Y t^{-1}\ \ \mbox{ and }\ \ X \in RV_1(\R^d, \{a_n\},c_X,\tau_X,\sigma_X).
$$
for some $c_Y>0,\ c_X>0$ and a probability measure $\sigma_X$ supported on 
$S_{\tau_X}=\{ x\in\R^d\, :\, \tau_X(x) = 1\}$.  
For more details, see Section \ref{sec:RV} and the references therein.
\end{remark}

\begin{remark} In many applications the radial function $\tau$ is simply a norm in $\R^k$.  In this case, $S_\tau=\{\tau=1\}$ is compact, 
and the condition \eqref{e:p:rv-via-h-assumption} is fulfilled for all continuous and homogeneous $h$. Moreover those assumptions still are fulfilled for radial functions with weighted or thresholded directions which is commonly use in portfolio optimization. 
\end{remark}
\begin{remark} Note that \eqref{e:p:rv-via-h} holds for Breiman models with any non-negative and integrable $h$. Namely if $Z = \xi W$, with $b(t) \P[\xi>tx] \to x^{-1}$, for $x>0$, then for all $W$ such that $\E[h(W)^{1+\epsilon}] <\infty$,
\[
b(t) \P[ h(Z)>t ] =b(t) \P[ \xi h(W) >t ] \to \E[h(W)],\ \ t\to\infty.
\]
See also Section \ref{sec:Breiman-models}.
\end{remark}

\medskip
Suppose now that $X$ and $Y$ are jointly regularly varying in the sense of 
Assumption \ref{a:X-Y-joint} and let
\begin{equation}\label{e:tilde-U-and-tilde-Theta}
\wt U := \frac{Y}{\tau(Y,X)},\ \ \mbox{ and }\ \ \wt \Theta:= \frac{X}{\tau_X(X)}.
\end{equation}
Observe that with the chosen polar coordinate parameterization $\tau(Y,X) = Y+\tau_X(X)$,
we have: 
\begin{equation}\label{e:Z/tau(Z)-via-tilde-U-and-tilde-Theta}
\frac{(Y,X)}{\tau(Y,X)} = \Big( \wt U, (1-\wt U) \wt \Theta\Big).
\end{equation}
Hence, in view of Definition \ref{d:tau-rv}, we obtain that for all $t>0$, as $n\to\infty$ and $r\to\infty$:
$
n\P[ \tau(Y,X) > ta_n] \To t^{-1},
$
and
\begin{equation}\label{e:U-THETA}  
(\wt U, (1-\wt U) \wt\Theta)\, | \, \tau(Y,X)>r \stackrel{d}\To \Theta_{Y,X}=:(U,(1-U)\Theta).
\end{equation}

In the sequel, we shall use the $(U,(1-U)\Theta)$ parameterization of the angular vector 
$\Theta_{Y,X}$ to formulate the optimal homogeneous prediction problem.  Observe that if 
$U<1$, then $\Theta=W/(1-U)$ can be identified from 
$(U,W) := \Theta_{Y,X}.$  Whenever $U=1$, however,
$\Theta$ is undefined, though unimportant since $(1-U)\Theta = 0$. In this case, by convention,
we shall define $\Theta := \theta_*$, for some arbitrary fixed $\theta_*\in S_{\tau_X}$
such that $\sigma_X(\{\theta_*\}) = 0$, where $\sigma_X$ is the angular measure of $X$ (cf Remark \ref{r:joint-RV}).

\begin{definition} \label{def:asymp-calibration} 
Let $(Y,X)$ be jointly $\tau$-regularly varying in the sense of 
Assumption \ref{a:X-Y-joint} and let ${\cal G}({\tau_X})$ 
be a class of continuous, non-negative $1$-homogeneous functions
fulfilling the following three conditions: \\

{(i)} ($\tau$-support-dominance) For all $g\in \mathcal{G}(\tau_X)$, $\{g>0\} \subset \{\tau_X>0\}$.\\

{(ii)} (boundedness) For all $g\in \mathcal{G}(\tau_X)$,
$
\|g\|_{\infty,\SX}:= \sup_{\theta\in \SX} |g(\theta)| <\infty.
$\\

{(iii)} (asymptotic calibration) For all $g\in \mathcal{G}(\tau_X)$,
\begin{equation}\label{e:g-calibration-lemma}
\P[ g(X) > t ] \sim \P[ Y>t], \ \ \mbox{ as }t\to\infty.
\end{equation}
  The predictors $g(X)$ that satisfy Relation \eqref{e:g-calibration-lemma} will be referred to as asymptotically calibrated extremal predictors for $Y$.
\end{definition}

Note that (ii) is a technical condition, which is always 
satisfied if $\SX$ is compact.

\begin{lemma}\label{lem:nonameidea} Let $(Y,X)$ and ${\cal G}(\tau_X)$ be as in Definition \ref{def:asymp-calibration}.
Then, for all $g\in {\cal G}(\tau_X)$ the bivariate tail dependence coefficient 
$\lambda(Y,g(X))$ defined in \eqref{e:lambda-tdep-def} exists and
$$
\lambda(g(X),Y) = \lim_{t\uparrow \infty} \P[ Y > t |  g(X) > t ] = \lim_{t\uparrow \infty} \P[ g(X) > t | Y > t ].
$$
\end{lemma}

The proof of this lemma is given in Section \ref{sec:proofs:sec:joint-RV}.  The next proposition leads us to the 
formulation of the optimal homogeneous prediction problem as a calculus of variations problem.

\begin{proposition}\label{p:lambda-via-U-Theta} Let $(Y,X)$ be jointly $\tau$-regularly varying
in the sense of Assumption \ref{a:X-Y-joint}.  Let ${\cal G}(\tau_X)$ 
be the class of predictors in Definition \ref{def:asymp-calibration}. Then, for all 
$g\in {\cal G}(\tau_X)$, and $(U,\Theta)$ as in \eqref{e:U-THETA}, 
we have
\begin{equation}\label{e:constraint-via-U-Theta}
\E[U] = \E[(1-U)g(\Theta)] =c>0
\end{equation}
and
\begin{equation}\label{e:lambda-via-U-Theta}
\lambda(Y,g(X)) = \frac{1}{c} \E[ U \wedge (1-U) g(\Theta)]
\end{equation}
\end{proposition}

\medskip
\noindent
The claim follows from Proposition \ref{p:rv-via-h} and Lemma \ref{lem:nonameidea}, but its proof is given
Section \ref{sec:proofs:sec:joint-RV}, for the sake of completeness.\\


Observe that every non-negative homogeneous function $h:\R^d\to\R_+$, support-dominated 
by $\tau_X$, is uniquely determined by its values on the unit sphere $S_{\tau_X}$, since for all
$\tau_X(x)>0$, we have $h(x) = \tau_X(x) \cdot h(\theta)$, where $\theta:= x/\tau_X(x)$. 
Thus, Proposition \ref{p:lambda-via-U-Theta} above entails that the optimal
extremal predictors via {\em homogeneous functions} can be obtained by the 
functions $g:\SX \to \R_+$ that minimize \eqref{e:lambda-via-U-Theta} under the 
constraints \eqref{e:constraint-via-U-Theta}.  This leads us to the following optimization problem.   

\begin{problem}[maximal precision] \label{p:opt-lambda} Consider the admissible set of 
measurable predictors, 
$$
D(c) := \Big\{ g: \SX \to \R_+\,\ :\, \E[ (1-U)g(\Theta)] = c\Big\},
$$
where $c>0$ and the expectations are relative to the probability measure $\sigma$ in Assumption \ref{a:X-Y-joint}. 

Introduce {\em precision functional} 
\begin{equation}\label{e:Lambda-functional}
\Lambda(g):=\frac{1}{\E[U]} \E[ U \wedge (1-U)g(\Theta) ],
\end{equation}
defined for all non-negative Borel measurable functions $g:\SX \to \R_+$.

The problem is to:
$$
\mathop{{\rm maximize}}\, \Lambda(g) ,\ \ \mbox{ subject to }\ 
 g \in D(\sigma_Y),\ \ \mbox{ where }\sigma_Y:= \E[U].
$$
\end{problem} 

\noindent Using the identity $|a-b| = a+b -2(a\wedge b)$, valid  for all $a,b\in \R$, 
Problem \ref{p:opt-lambda} can be equivalently formulated as a minimization problem:

\begin{problem}[minimal $L^1$-loss] \label{p:opt-lambda-minimization}  For a measurable $g : S_{\tau_X} \to \R_+$, define the functional
\begin{equation}\label{e:I(g)-via-U}
I(g):= \E[| U  -  (1-U)g(\Theta)| ].
\end{equation}
The problem is to 
\begin{equation}\label{e:L1-proj}
{\rm minimize}\, I(g)  ,\ \ \mbox{ subject to }\ \ 
 g \in D(\sigma_Y).
 \end{equation}
\end{problem}

In the following sub-section we provide a general solution to 
Problem \ref{p:opt-lambda-minimization} (equivalently,  \ref{p:opt-lambda})
for all possible spectral measures arising under Assumption \ref{a:X-Y-joint}.

\begin{remark}[The domain of optimization]
The natural domain for Problems \ref{p:opt-lambda} and \ref{p:opt-lambda-minimization} 
is the convex set $D(\sigma_Y)$, which is closed in a certain $L^1$-sense (cf Section \ref{sec:solution}). 
As shown in Theorem \ref{thm:final} below, over this domain, these two problems always have a solution.  
Naturally, the question arises as to what is the connection between a {\em discontinuous} solution 
$g^{\rm (opt)}$ and optimal extreme event prediction.  This is discussed in the following remarks as well as 
in Section \ref{sec:Breiman-models}. 
\end{remark}

\begin{remark}[Optimal prediction]\label{r:circling-back-to-the-optimal-h}   
If $g^{(\rm opt)}$ is a solution of Problems \ref{p:opt-lambda} and \ref{p:opt-lambda-minimization}, then 
one can define an extremal predictor function 
$$
h^{(\rm opt)}(x):= \tau_X(x) g^{\rm (opt)}(x/\tau_X(x)) 1_{\{\tau_X(x)>0\}}.
$$
If $\theta \mapsto g^{\rm (opt)}(\theta)$ happens to be continuous and bounded on $\SX$,
then so is $h^{\rm (opt)}(\cdot)$, and in view of Proposition \ref{p:lambda-via-U-Theta}, we obtain 
\begin{equation}\label{e:Lambda=lambda}
 \lambda(Y,h^{\rm (opt)}(X)) = \Lambda(g^{\rm (opt)}(\Theta)).
 \end{equation}
That is, $h^{(\rm opt)}$ is an optimal extremal predictor over the class ${\cal H}_+$.   
\end{remark}

\begin{remark}[On the correspondence between $\Lambda$ and the tail-dependence coefficient] \label{r:subtlety}
 The solution $\theta\mapsto g^{\rm (opt)}(\theta)$ may sometimes be {\em discontinuous}. In such cases, the functional 
 $\Lambda(g)$ remains well-defined, while the tail-dependence functional for $Y$ and $h^{\rm (opt)}(X)$, may not exist for 
 certain pathological jointly regularly varying $(Y,X)$ models.  In Section \ref{sec:Breiman-models}, however, we shall 
 demonstrate that over the broad class of Breiman-type models, which encompasses all possible angular distributions, the 
 general (and possibly discontinuous) solutions to Problems \ref{p:opt-lambda} and \ref{p:opt-lambda-minimization} 
 {\em always yield} asymptotically optimal predictors.  More precisely, for such models the tail dependence coefficient 
 between $Y$ and $h^{\rm (opt)}(X)$ {\em always exists} and \eqref{e:Lambda=lambda} holds (cf Corollary \ref{c:Pareto-Breiman} 
 and Remark \ref{r:subtlety-response}, below).
\end{remark}

\subsection{The general solution to the variational problems} \label{sec:solution}

In this section, we will focus on the abstract optimization Problem \ref{p:opt-lambda-minimization} in the
general case of non-negative measurable functions $g$, where the pair
$(U,\Theta)$ takes values in $[0,1]\times \SX$ and has probability distribution $p_{U,\Theta}$. 
To exclude trivialities, we will assume that 
\begin{equation}\label{e:condition-on-U}
    \P[U=0]<1\ \ \mbox{ and }\ \ \P[U=1]<1,
\end{equation}
which corresponds precisely to \eqref{e:a:X-Y-joint}
in the joint regular variation condition on $Y$ and $X$. For the purposes of this section, however, the variables 
$Y$ and $X$, will play no role.

The joint distribution $p_{U,\Theta}(dud\theta)$ can always be written as
\begin{equation}\label{e:p_U,Theta}
p_{U,\Theta}(dud\theta) = p(du|\theta) p_\Theta(d\theta),
\end{equation}
where $p(du|\theta)$ is a {\em regular conditional probability} i.e., 
$p(du|\theta)$ is a probability transition kernel \citep[see e.g. Theorem 10.2.2 in][]{dudley:1989}.\\

{\em We proceed with some convenient notation.}  Let
\begin{equation}\label{e:b(theta)}
b(\theta):= \int_{[0,1]} (1-u)p(du |\theta)
\end{equation}
and observe that $p_\Theta( \{0<b(\theta)<1\} ) = 1$, by  \eqref{e:condition-on-U}. 

The functional $I(\cdot)$ in \eqref{e:I(g)} can be written as:
\begin{equation}\label{e:I(g)}
I(g) := \int_{\SX} \Big\{ \int_{[0,1)} | u - (1-u)g(\theta)| p(du|\theta) \Big\} p_{\Theta}(d\theta). 
\end{equation}
In view of \eqref{e:b(theta)}, by an application of the triangle inequality and Fubini's theorem, we see that $I(h)<\infty$,
for all measurable $h$, such that 
\begin{equation}\label{e:C(g)}
C(h) :=  \E[ (1-U) |h(\Theta)|] = \int_{\SX} |h(\theta)| b(\theta) p_\Theta(d\theta)<\infty.
\end{equation}
We will denote this class of functions by $L^1(b\cdot p_\Theta)$ and the non-negative such functions by $L_+^1(b\cdot p_\Theta)$.

The functional $C(\cdot)$ in \eqref{e:C(g)} will be referred to as the {\em constraint functional}.  This terminology is
motivated by Problem \ref{p:opt-lambda-minimization}, which amounts to minimizing $I(g)$ over all 
$g\in L_+^1(b\cdot p_\Theta)$ such that $C(g) = \E[U] = \sigma_Y$. Note also that the constraint functional $C(g)$ is 
precisely the {\em weighted} $L^1$-norm of $g$ in $L^1(b\cdot p_\Theta)$, so that for $C(g)<\infty$ is equivalent to
$g\in L^1(b\cdot p_\Theta)$.\\

Consider now the family of cumulative probability distribution functions 
 indexed by $\theta\in \SX$:
\begin{equation}\label{e:F-theta}
  F_\theta(t ):= \frac{1}{b(\theta)} \int_{ \{0\le u\le t\}}  (1-u)p(du |\theta),\ \  t\in [0,1].
 \end{equation}
 Observe that the $F_\theta$'s are the distribution functions of probability measures concentrated
 on $[0,1)$.  Notably, these measures, also denoted by $F_\theta$, have no atoms at $1$, i.e., 
 $F_\theta(1) = F_\theta(t-) = 1$ and
 \begin{equation}\label{e:no-atom-at-1}
  F_\theta(\{1\}) = F_\theta(1) -F_\theta(1-) =0,\ \ \mbox{ for all }\theta\in \SX.     
 \end{equation}
The following lemma establishes a precise formula for the G\^ateaux directional derivative of the functional $I$, 
which will be key.  Its proof is given in Section \ref{sec:proofs:main}.

 \begin{lemma}[G\^ateaux derivative] \label{l:gateaux-general} Let $I(\cdot)$ be as in \eqref{e:I(g)}, the $b(\theta)$'s and $F_\theta(t)$'s be as in 
 \eqref{e:b(theta)} and \eqref{e:F-theta}, respectively.  Then, for all $g\ge 0$ and $g,\, h\in L^1(b\cdot p_\Theta)$, we have
 \begin{align}\label{e:lower-bound-via-F}
&\lim_{\delta\downarrow 0} \frac{I(g+\delta h)-I(g)}{\delta} \nonumber\\
&\quad \quad= 2\int_{\SX}   b(\theta)  \Big\{ F_\theta(q(\theta)) h(\theta) +  F_\theta(\{q(\theta)\}) h(\theta)_-\Big\}  p_\Theta(d\theta) 
 - \int_\SX b(\theta) h(\theta) p_\Theta(d\theta),
 \end{align}
 where $F_\theta(\{t\}) = F_\theta(t) - F_\theta(t-)$ and
$$
 q(\theta):= \frac{g(\theta)}{1+g(\theta)}\ \ \  \mbox{ and }\ \ \ h(\theta)_- = \max\{0,-h(\theta)\}.
$$
 \end{lemma}

 \medskip
 {\em We proceed with some further notation and auxiliary results needed to formulate and prove the main result.} 
 It turns out one can obtain an essentially explicit solution to  Problem \ref{p:opt-lambda-minimization} in terms of 
 the quantiles of the weighted conditional probability distribution 
 $F_\theta$ in \eqref{e:F-theta}. Let
 \begin{equation}\label{e:q-alpha}
  q_\alpha(\theta):= \inf\Big\{ t \ge 0\, :\, F_\theta(t )  \ge \alpha \Big\},\ \ \alpha\in [0,1]
 \end{equation}
 be the left-continuous generalized inverse of $F_\theta(\cdot)$, with the small modification that above 
 infimum is restricted to $t\in [0,\infty)$ since the distributions $F_\theta$ are supported on $[0,1)$.  This implies 
 that $q_0(\theta) = 0$. We shall also need the {\em right-continuous} generalized inverse
 \begin{equation}\label{e:q-alpha-plus}
  q_{\alpha + }(\theta) = \inf\Big\{ t \ge 0\, :\, F_\theta(t )  > \alpha \Big\},
 \end{equation}
 where by convention we set $q_{1+}(\theta) = q_1(\theta) =\sup\{ t \in \R\, :\, F_\theta(t) <1\}\le 1$.  Note that
 $$
  q_{\alpha+}(\theta) = \lim_{\beta>\alpha,\ \beta\downarrow \alpha} q_{\beta}(\theta),\ \ \mbox{ for all $0\le \alpha <1$.}
 $$ 
 The fact that the distributions $F_\theta$'s have no atoms at $1$ (cf \eqref{e:no-atom-at-1}), 
 implies the important observation that
\begin{equation}\label{e:q-alpha-+<1}
   q_\alpha(\theta) \le q_{\alpha+}(\theta) < 1,\ \ \mbox{ for all $0\le \alpha<1$ and $\theta\in \SX$.}
\end{equation}

 The following general result will be useful.  Its proof is given in Section \ref{sec:proofs:main}.
 
 \begin{lemma}\label{l:CDF-inverses} Let $F$ be the CDF of a probability distribution on $[0,1]$ and let
 $\alpha \mapsto q_\alpha$ and $\alpha \mapsto q_{\alpha+}$ be its left- and right-continuous generalized inverses, defined as in 
 \eqref{e:q-alpha} and \eqref{e:q-alpha-plus}, respectively.  The following statements hold:

 {\em (i)} For all $0\le \alpha \le 1$, we have
 \begin{equation}\label{e:l:CDF-inverses-1}
   F(q_\alpha -) \le F(q_{\alpha+}-) \le \alpha \le F(q_\alpha)\le F(q_{\alpha+}),
 \end{equation}
 where $F(t-):= \lim_{s<t,\ s\to t} F(s)$ denotes the left-limit of $F$ at $t$.

 {\em (ii)} For all $0\le \alpha\le 1$, such that $q_\alpha<q_{\alpha+}$, we have
 \begin{equation}\label{e:l:CDF-inverses-2}
   F(t) = F(q_\alpha),\ \ \mbox{ for all $ q_\alpha \le t < q_{\alpha+}$.}
  \end{equation}
 \end{lemma}

\medskip
\noindent  The next lemma provides a technical ingredient needed for our first result.
  
 \begin{lemma}\label{l:q_alpha-optimiality} For all $0 \le \alpha \le 1$ and functions $\theta\mapsto h(\theta)$, we have:
 \begin{align}\label{e:l:q_alpha-optimiality}
     F_\theta(t) h(\theta) + F_\theta(\{t\}) h(\theta)_- &\ge \alpha h(\theta),\ \ \mbox{ for all } \theta\in \SX
     \mbox{ and }t\in [q_\alpha(\theta),q_{\alpha+}(\theta)],
 \end{align}
 where $x_\pm := \max\{0,\pm x\}$.
 \end{lemma}
 \begin{proof} If $h(\theta)\ge 0$, then the bound in \eqref{e:l:q_alpha-optimiality} holds 
 since $\alpha \le F_\theta(q_\alpha(\theta)) \le F_\theta(t),\ t\ge q_\alpha(\theta)$ (cf \eqref{e:l:CDF-inverses-1}) and
 since the term $F_\theta(t) h(\theta)_-$ vanishes.
  
  On the other hand, if $h(\theta)<0$, we have $h(\theta) = - h(\theta)_-$ and by 
  writing 
  $$
  F_\theta(t) = F_\theta(t-) + F_\theta(\{t\}),
  $$
  we obtain that the left-hand side of \eqref{e:l:q_alpha-optimiality} equals
  $$
   F_\theta(t) h(\theta) + F_\theta(\{t\}) h(\theta)_- = F_\theta(t-) h(\theta). 
  $$
  By \eqref{e:l:CDF-inverses-1} and the monotonicity of $F_\theta$, however, for all 
  $t\in [q_\alpha(\theta),q_{\alpha+}(\theta)]$, 
  \[
   F_\theta(t-) \le F_\theta(q_{\alpha+}(\theta)-) \le \alpha.
  \]
   Hence by multiplying the above inequalities with the negative number $h(\theta)$,  
   we obtain $F_\theta(t-) h(\theta) \ge \alpha h(\theta)$, which completes the proof of 
   \eqref{e:l:q_alpha-optimiality}.
   \end{proof}

The following proposition characterizes a broad class of solutions to Problem \ref{p:opt-lambda-minimization} under 
arbitrary positive values of the constraint functional.  It shows that the solutions can be obtained from functions
sandwiched between the left- and right-continuous quantile functions of the $F_\theta$ distributions.

\begin{proposition}\label{p:optim-general} Let $g \in L_+^1(b\cdot p_\Theta)$ and define $q(\theta):= g(\theta)/(1+g(\theta))$. 
Consider the following two conditions:

{\em (i)} For some fixed $0\le \alpha<1$, we have
\begin{equation}\label{e:p:optim-general-i}
 q_\alpha(\theta) \le q(\theta) \le q_{\alpha+}(\theta),\ \ \ \mbox{ for almost all $\theta$, modulo $b(\theta) p_\Theta(d\theta)$}.
\end{equation}

{\em (ii)} We have
\begin{equation}\label{e:p:optim-general-ii}
 q_{1}(\theta) \le q(\theta),\ \ \ \mbox{ for almost all $\theta$, modulo $b(\theta) p_\Theta(d\theta)$}.
\end{equation}

If either {\em (i)} or {\em (ii)} holds, then
$$
I(g) = \inf_{g' \in D(C(g))} I(g').
$$     
\end{proposition}
\begin{proof}
    Let $g' \in D(C(g))$ be arbitrary. Since $C(g) = C(g')$, with $C(\cdot)$ as in \eqref{e:C(g)}, we obtain
   \begin{align}\label{e:p:optim-general-1}
     C(g') - C(g) = \int_{\SX}  (g'(\theta) -g(\theta)) b(\theta) p_\Theta(d\theta) =: \int_{\SX} h(\theta) b(\theta) p_\Theta(d\theta) = 0.
   \end{align}
   
   Thus, by Lemma \ref{l:gateaux-general} applied with $h:=g'-g$, we obtain:
  \begin{align}\label{e:p:optim-general-2}
  \lim_{\delta \downarrow 0 }\frac{I( g + \delta\cdot(g'-g))- I(g)}{\delta} =
   2\int_{\SX}   b(\theta)  \Big\{ F_\theta(q(\theta)) h(\theta) +  F_\theta(\{q(\theta)\}) h(\theta)_-\Big\}  p_\Theta(d\theta),
   \end{align}
   where the second integral in \eqref{e:lower-bound-via-F} vanishes in view of \eqref{e:p:optim-general-1}.

  Now, focus on the integrand in \eqref{e:p:optim-general-2}.  If \eqref{e:p:optim-general-i} holds, by 
  Lemma \ref{l:q_alpha-optimiality} applied with $t:=q(\theta)$, we obtain 
  $F_\theta(q(\theta)) h(\theta) +  F_\theta(\{q(\theta)\}) h(\theta)_-\ge \alpha h(\theta)$.  This,
  in view of \eqref{e:p:optim-general-2} and \eqref{e:p:optim-general-1}, implies
  \begin{align}\label{e:p:optim-general-3}
 \lim_{\delta \downarrow 0 }\frac{I( g + \delta\cdot(g'-g))- I(g)}{\delta} \ge 
  2 \alpha \int_{\SX}   b(\theta) h(\theta) p_\Theta(d\theta) = 0.
 \end{align}

 On the other hand, if \eqref{e:p:optim-general-ii} holds, we have $F_\theta(q(\theta))=1$ and 
 by dropping the non-negative term $F_\theta(\{q(\theta)\}) h(\theta)_-$, we obtain 
 $F_\theta(q(\theta)) h(\theta) +  F_\theta(\{q(\theta)\}) h(\theta)_- \ge h(\theta)$.  Thus, 
 in view of \eqref{e:p:optim-general-2}, we obtain that \eqref{e:p:optim-general-3} holds with $\alpha:=1$.

  Note that the function $\varphi(\delta) := I( g + \delta\cdot(g'-g))- I(g)$ is convex on $\delta \in [0,\infty)$. By the Lebesgue
  DCT it readily follows that $\delta \mapsto \varphi(\delta)$ is also continuous. We have shown that Relation
  \eqref{e:p:optim-general-3} holds under either one of the conditions 
  \eqref{e:p:optim-general-i} or \eqref{e:p:optim-general-ii}.  This means that the right-derivative of $\varphi(\cdot)$ at 
  $0$ is non-negative, which, in view of the convexity of $\varphi$, implies that $\varphi(\delta)\ge 0 =\varphi(0)$, for 
  all $\delta\ge 0$.  By taking $\delta:=1$, we obtain
  $$ 
   I(g)   \le I(g'),
  $$
  which, since $g'\in D(C(g))$ was arbitrary,  completes the proof of the proposition.
\end{proof}

  Proposition \ref{p:optim-general} provides a rich family of solutions to 
   Problem \ref{p:opt-lambda-minimization}, where the constraint $\sigma_Y$ is replaced by a potentially 
   arbitrary finite value $C(g)$. Thus, to complete the solution to Problem \ref{p:opt-lambda-minimization},
   it remains to demonstrate that it is always possible to choose such a function $g$ that meets the constraint
   $
     \sigma_Y = \E[U] = C(g).
   $
   To this end, the conditional quantile functions $q_\alpha(\theta)$ and $q_{\alpha+}(\cdot)$ will play an important role.
   
   For all $0<\alpha<1$ and $\lambda \in [0,1]$, introduce the functions
   \begin{equation}\label{e:g-alpha-lambda}
     g_\alpha^{(\lambda)} (\theta):= \frac{q_\alpha^{(\lambda)}(\theta)}{1-  q_\alpha^{(\lambda)}(\theta)},
   \end{equation}
   where
   \begin{equation}\label{e:q-alpha-lambda}
    q_\alpha^{(\lambda)}(\theta):= (1-\lambda) q_\alpha(\theta) + \lambda q_{\alpha+}(\theta).
   \end{equation}
   Recall that by \eqref{e:q-alpha-+<1}, we have $q_{\alpha+}^{(\lambda)}(\theta)<1$ and hence 
   $g_\alpha^{(\lambda)}(\theta) <\infty$, for all $0\le \alpha <1$ and $\lambda \in [0,1]$. Note, however, that
   it is often possible to have $q_1(\theta) = q_{1+}(\theta) = 1$ in which case $g_1(\theta)=\infty$, by convention.

  The following lemma, proved in Section \ref{sec:proofs:main} below, is essential. 
  It establishes that for all $\alpha \in [0,1)$, the predictors 
  $g_\alpha^{(\lambda)}(\cdot) ,\ 0\le \lambda\le 1$ belong to $L_+^1(b\cdot p_\Theta)$ and shows other 
  useful properties of the constraint functional.

   \begin{lemma}\label{l:C(g)-properties} Let $g_\alpha^{(\lambda)}$ be as in \eqref{e:g-alpha-lambda}.  Then:\\

    {\em (i)} We have that $C(g_\alpha^{(\lambda)})<\infty$,  for all $0\le \alpha <1$ and $0\le \lambda \le 1$.\\
    
    {\em (ii)} The function $\lambda \mapsto C(g_\alpha^{(\lambda)})$ is monotone non-decreasing and continuous in $\lambda \in [0,1]$.\\

    {\em (iii)} The function $\alpha \mapsto C(g_\alpha)$ in monotone non-decreasing and left-continuous in $\alpha \in (0,1]$.\\

    {\em (iv)} For all $0\le \alpha < 1$, we have that $C(g_{\alpha+}) = \lim_{\beta\downarrow \alpha} C(g_\beta).$
   \end{lemma}

\medskip
    \noindent Now, we are ready to formulate and prove our solution to the
     equivalent Problems \ref{p:opt-lambda} \& \ref{p:opt-lambda-minimization}.

     \begin{theorem}\label{thm:final} Let $(U,\Theta)$ be a random vector taking values in $[0,1]\times \SX$ and having a
     probability distribution $p_{U,\Theta}(du,d\theta)$ written as in \eqref{e:p_U,Theta}. Assume $\P[ U=0]<1$ and $\P[U=1]<1$.\\

     {\em (i)} If $U = u(\Theta)$, almost surely, for some Borel function $u(\cdot)$, then
     \begin{equation}\label{e:thm:final-i-g_opt}
      g^{(\rm opt)} (\theta) := \frac{\E[U]}{\E[U\cdot 1\{U<1\}]} \cdot \frac{u(\theta)}{(1-u(\theta))} 1(\{ u(\theta) <1\}) 
    \end{equation}
     is a solution to Problem \ref{p:opt-lambda} and 
     \begin{equation}\label{e:thm:final-i}
      \Lambda(g^{\rm (opt)}) = \frac{\E[U\cdot 1\{U<1\}]}{\E[U]} = 1- \frac{\P[U=1]}{\E[U]}. 
      \end{equation}
   
   \noindent  In the next two cases, suppose that $U$ is not a deterministic function of $\Theta$.\\
    
    {\em (ii)}  If $\E[U] < C(g_1)$, then there exist an $\alpha \in [0,1)$ and $\lambda \in [0,1]$ such that for
    $g^{\rm (opt)} := g_\alpha^{(\lambda)}$, we have 
    $\E[U]=C(g^{\rm (opt)})$ and $g^{\rm (opt)}$ is a solution to Problem \ref{p:opt-lambda}.\\
   
     {\em (iii)} If $ C(g_1)\le \E[U]$, then $\theta\mapsto g^{\rm (opt)}(\theta) := (\E[U]/C(g_1)) \cdot g_1 (\theta)$ is a solution to 
     Problem \ref{p:opt-lambda} and  $\Lambda(g^{\rm (opt)})$ is as in \eqref{e:thm:final-i}.
     \end{theorem}
     \begin{proof}
   {\em Part (i):} We will show $g^{\rm (opt)}$ solves the equivalent Problem \ref{p:opt-lambda}.
    
   Observe first that by \eqref{e:thm:final-i-g_opt}, we have
   \begin{align*}
   C(g^{\rm (opt)}) =\E[(1-u(\Theta))g^{\rm (opt)}(\Theta)] 
   = \frac{\E[U]}{\E[U\cdot 1\{U<1\}]} \cdot \E [U\cdot 1\{U<1\}] = \E[U],
   \end{align*}
   which shows that $g^{\rm (opt)}$ satisfies the constraint.  Now, for every 
   $g' \in D(\E[U])$, we necessarily have 
   $$
   \Lambda(g') = \frac{1}{\E[U]} \E[ U \wedge (1-U) g'(\theta) ] \le  \frac{1}{\E[U]} \E[U \cdot 1\{U<1\}] = 1- \frac{\P[U=1]}{\E[U]}.
   $$
   On the other hand, with $g^{\rm (opt)}$ as in \eqref{e:thm:final-i-g_opt}, over the event $\{U<1\}$, we have 
   $$
    U = u(\Theta) \le (1-u(\Theta)) g^{\rm (opt)}(\Theta) = \frac{\E[U]}{\E[U\cdot 1\{U<1\}]} u(\Theta), 
   $$
   since $1\le \E[U]/\E[U\cdot 1\{U<1\}]$. Thus,
   $$
    \Lambda(g^{\rm (opt)})=\frac{1}{\E[U]} \E[ U \wedge (1-U)g^{\rm (opt)}(\Theta)] = \frac{1}{\E[U]}\E[U \cdot 1\{U<1\}] = 1-\frac{\P[U=1]}{\E[U]},
   $$
   which shows that $\Lambda(g^{\rm(opt)})\le \Lambda(g')$, for all $g'\in D(\E[U])$, completing the proof of (i).\\
   
   {\em Part (ii):} Let 
   \begin{equation}\label{e:thm:final-0}
   \wt\alpha := \sup\{ \alpha \ge 0\ :\,  C(g_\alpha) \le \E[U] \}
   \end{equation}
   and observe that since $C(g_0)=0$ the above supremum is over a non-empty set, and hence $\wt\alpha \ge 0$.  
   We will show next that:
   \begin{equation}\label{e:thm:final-1}
    \wt\alpha<1\ \ \mbox{ and }\ \  C(g_{\wt\alpha}) \le \E[U] \le C(g_{\wt\alpha + })<\infty.
   \end{equation}
  Indeed, since the function $\alpha \mapsto C(g_\alpha)$ is monotone non-decreasing 
   and left-continuous (cf part (iii) of Lemma \ref{l:C(g)-properties}), we obtain from \eqref{e:thm:final-0} that
   $$
     C(g_{\wt\alpha}) \le \E[U].
   $$
   On the other hand, since $\E[U] < C(g_1)$, we necessarily have $\wt \alpha<1$ and hence by 
   Lemma \ref{l:C(g)-properties}, it follows that $C(g_{\wt \alpha+})<\infty$.  To show \eqref{e:thm:final-1}, it remains
   to prove that $\E[U] \le C(g_{\wt\alpha+})$.
   
   Suppose that $C(g_{\wt\alpha+}) < \E[U]$. Since $\wt \alpha <1$, by part {\em (iv)} of Lemma \ref{l:C(g)-properties}, it follows that
   $$
    C(g_{\wt\alpha+}) = \lim_{\beta \downarrow \wt\alpha} C(g_{\beta}) <\E[U].
   $$
   Thus, there exists a $\beta$ such that $\wt\alpha<\beta<1$ and $\E[C(g_\beta)] < \E[U]$.  This, however, 
   contradicts the definition of $\wt\alpha$ in \eqref{e:thm:final-0} and shows that $ \E[U]\le C(g_{\wt\alpha+})$, completing 
   the proof of \eqref{e:thm:final-1}.

   Recall that $g_\alpha^{(\lambda)},\ \lambda\in [0,1]$ monotonically and continuously interpolates between 
   $g_\alpha =g_\alpha^{(0)}$ and $g_{\alpha+}= g_\alpha^{(1)}$ (cf \eqref{e:g-alpha-lambda}).  By part (ii) of Lemma \ref{l:C(g)-properties}, moreover, the function $\lambda \mapsto C(g_{\wt \alpha}^{(\lambda)})$ is continuous, and hence in 
   view of \eqref{e:thm:final-1}, it follows that 
   $C(g_{\wt\alpha}^{(\wt\lambda)})=\E[U]$, for some $\wt \lambda \in [0,1]$.  Note also that 
   $q(\theta):= g_{\wt\alpha}^{(\wt\lambda)}(\theta)/(1+g_{\wt\alpha}^{(\wt\lambda)}(\theta))$ 
   satisfies the condition \eqref{e:p:optim-general-i} with $\alpha$ replaced by $\wt\alpha$.  Thus,
   Proposition \ref{p:optim-general} implies that $g^{\rm (opt)}(\theta):=g_{\wt\alpha}^{(\wt\lambda)}(\theta),\ \theta\in \SX$
   is a solution to Problem \ref{p:opt-lambda-minimization}, as we have already established that $g^{\rm (opt)}$ satisfies the
   desired constraint $C(g^{\rm(opt)}) = \E[U]$.\\

   {\em Part (iii):} Since $c:=\E[U]/C(g_1)\ge 1$, we have that 
   $q(\theta):=  c \cdot g_1 (\theta)/ (1+c\cdot  g_1(\theta)) \ge q_1(\theta) = g_1(\theta)/(1+g_1(\theta))$. 
   Thus, by observing that $C(g^{\rm (opt)}) = C( c \cdot g_1) =\E[U]$, 
   part (ii) of Proposition \ref{p:optim-general} implies that $g^{\rm (opt)}$ is a solution to Problem \ref{p:opt-lambda-minimization}.
   It remains to show that $\Lambda(g^{\rm (opt)})$ is as in \eqref{e:thm:final-i}.

   Conditionally on $\Theta$, over the event $\{U<1\}$, we have
   $$
   (1-U) g^{\rm (opt)}(\Theta) = c \cdot (1-U) \frac{q_1(\Theta)}{1-q_1(\Theta)} \ge c q_1(\Theta) \ge c U\ge U,
   $$
   almost surely, since $c\ge 1$ and since $q_1(\Theta)\ge U$, almost surely.  This implies that
   $$
   \Lambda(g^{\rm (opt)}) = \frac{1}{\E[U]} \E[ U \wedge (1-U)  g^{\rm (opt)}(\Theta)] =\frac{1}{\E[U]} \E[ U\cdot 1\{U<1\}]
   = 1 -\frac{\P[U=1]}{\E[U]},
   $$
   which completes the proof of part {\em (iii)} and the theorem.
     \end{proof}

 \subsection{The generalized Breiman models}\label{sec:Breiman-models}

  In this section, we discuss the broad class of so-called Breiman's models, which realize all possible 
  tail dependence measures.  We will demonstrate that in the context of these models, the solutions to the 
  variational problems in Theorem \ref{thm:final} yield asymptotically optimal homogeneous predictors. 
  We begin with a sharp version of the classical Breiman's lemma, which is a reformulation of 
  Lemma 2.3 in \cite{davis:mikosch:2008}.

  \begin{lemma}[Lemma 2.3 in \cite{davis:mikosch:2008}] \label{l:Breiman} Let $\xi$ and $\eta$ be independent non-negative random variables, where
  $\xi$ has regularly varying right tail with index $\alpha>0$ and $\E[\eta^\alpha]<\infty$.
  Consider the following two conditions.
  \begin{enumerate}
       \item[(i)] For some $c_\xi>0$, we have
        \begin{equation}\label{e:l:Breiman-i}
          t^{\alpha} \P[ \xi > t] \to c_\xi\in (0,\infty), \mbox{ as }t\to\infty.
          \end{equation}
       \item[(ii)] For some $\delta>0$, we have $\E[\eta^{\alpha+\delta}] <\infty$.
    \end{enumerate}
  If either (i) or (ii) holds, then
  \begin{equation}\label{e:l:Breiman}
 \frac{\P[\xi\eta >t]}{\P[\xi>t]} \to \E[\eta^\alpha],\ \ \mbox{ as }t\to\infty.
  \end{equation}
  \end{lemma}

  \begin{remark} In contrast to the classical Breiman's lemma, 
  when the variable $\xi$ has Pareto-type tails, i.e., \eqref{e:l:Breiman-i} holds, one only needs $\E[\eta^\alpha]<\infty$ to 
  conclude \eqref{e:l:Breiman}.
  \end{remark}

  In the case when $\eta$ in \eqref{e:l:Breiman} is replaced by a random vector, one obtains a rich family of
  multivariate regularly varying models \citep[see, e.g.][]{basrak:davis:mikosch:2002,kulik2020heav}, which realize
  all possible asymptotic tail dependence regimes.  
  
  The next result shows a curious fact that the multivariate Breiman-type models are, in fact, regularly varying in a 
  stronger {\em total variation} sense.  This fact is of independent interest, which we will revisit in Section 
  \ref{sec:inference}.
  
  \begin{proposition}[The total regular variation of Breiman models]
  \label{p:total-Breiman}
  Let $Z:= \xi \eta,$
  where $\xi$ and $\eta = (\eta_i)_{i=1}^k$ are independent and $\xi$ is non-negative with regularly varying 
  right tail with index $\alpha>0$.  Let also $\tau:\R^k\to \R_+$ be a non-negative 1-homogeneous function and
  suppose that $0< \E[\tau(\eta)^\alpha] <\infty$.
  
  If either (i) Relation \eqref{e:l:Breiman-i} holds; or (ii) $\E[\tau(\eta)^{\alpha+\delta}]<\infty$, for some $\delta>0$, then
  \begin{equation}\label{e:p:total-Breiman-i}
  \frac{\P[ \tau(Z) > t ]}{\P[\xi>t]}  \longrightarrow  \E[\tau(\eta)^\alpha],\ \ \mbox{ as }t\to\infty,
  \end{equation}
   and 
   \begin{equation}\label{e:p:total-Breiman-ii}
       \frac{Z}{\tau(Z)} \, \Big\vert\, \{\tau(Z)>t\} \stackrel{TV}{\longrightarrow} \Theta_Z,\ \ \mbox{ as }t\to\infty,
   \end{equation}
 where $\stackrel{TV}{\longrightarrow}$ denotes convergence of the probability laws in total variation, and $\Theta_Z$ has the probability distribution 
 $\sigma$ on $S_\tau$ such that
 \begin{equation}\label{e:p:total-Breiman-ii-sigma(A)}
    \sigma(A) = \frac{1}{\E[\tau(\eta)^\alpha]} \E [ 1_A(\eta/\tau(\eta)) \tau(\eta)^\alpha],\ \ A\in {\cal B}(S_\tau).
 \end{equation}
In particular, $Z \in RV_\alpha(\R^k,\{a_n\},c_Z,\tau,\sigma)$, where $a_n\P[\xi>n] \sim 1$ and $c_Z =  \E[\tau(\eta)^\alpha]$.
  \end{proposition}

\noindent The proof of Proposition \ref{p:total-Breiman} is given in Section \ref{sec:proofs_breiman_model}.\\

{\bf Pareto-Breiman models.} In the rest of this section, we return to the optimal prediction context.  We 
shall consider the important special case of the Breiman models, where the heavy-tailed factor has asymptotically 
Pareto tails, i.e., \eqref{e:l:Breiman-i} holds and for simplicity $\alpha=1$.

Specifically, we let
\begin{equation}\label{e:Y_X_Breiman}
Z = (Y,X) = \xi\cdot (V,W),
\end{equation}
be given by a Breiman-type model as in Proposition \ref{p:total-Breiman} such that $V\ge 0$, $(V,W) = (V,W_1,\cdots,W_d)$ and $\xi$ are independent.
We assume, moreover that $\xi$ satisfies \eqref{e:l:Breiman-i} with $\alpha = 1$.

\begin{proposition}[Pareto-Breiman models] \label{p:Pareto-Breiman} Let $Y$ and $X$ be as in \eqref{e:Y_X_Breiman} with $\xi$ as in \eqref{e:l:Breiman-i} with
$\alpha=1$. Let also $\tau(y,x):= y_+ + \tau_X(x)$, be a non-negative, continuous $1$-homogeneous such that
$0<\E[V]<\infty$ and $0<\E[\tau_X(W)]<\infty$. Then:\\

{\em (i)} $Y$ and $X$ are jointly regularly varying in the sense of Assumption \ref{a:X-Y-joint} and
\begin{equation}\label{e:p:Pareto-Breiman-i}
\frac{(Y,X)}{\tau(Y,X)} \, \vert\, \{\tau(Y,X)>t\} \stackrel{TV}{\longrightarrow} \Theta_Z =:(U , (1-U) \Theta),\ \ \mbox{ as }t\to\infty.
\end{equation}
We have moreover that for all $B\in {\cal B}(\SX)$,
\begin{equation}\label{e:p:Pareto-Breiman-i-Theta}
\P[ \Theta \in B\, |\, U<1] =\frac{1}{\E[ 1_{\{\tau_X(W)>0\}}\cdot (V + \tau_X(W))]}\E\Big[1_B\Big(\frac{W}{\tau_X(W)}\Big) \cdot 1_{\{\tau_X(W)>0\}} (V+\tau_X(W)) \Big].
\end{equation}

{\em (ii)} For all non-negative, measurable homogeneous $h:\R^{1+d}\to\R_+$, such that $\E[ h(V,W) ]<\infty$, we have
\begin{equation}\label{e:p:Pareto-Breiman-ii}
 t\P[ h(Y,X) >t] \to c_\xi \E[h(V,W)],\ \ \mbox{ as }t\to\infty.
\end{equation}

{\em (iii)} If $h$ is as in part {\em (ii)} and such that $\{h>0\}\subset \{\tau >0\}$, then 
\begin{equation}\label{e:p:Pareto-Breiman-iii}
\E [ h(U,(1-U)\Theta) ]  = \frac{1}{\E[\tau(V,W)]} \E[h(V,W)].
\end{equation}
\end{proposition}

The proof of this result is given in Section \ref{sec:proofs_breiman_model}.  We shall comment next on the subtle difference between the
distribution of $\Theta$ in \eqref{e:p:Pareto-Breiman-i} the {\em marginal} angular distribution of $X$.

\begin{remark} In the context of Proposition \ref{p:Pareto-Breiman}, we have
\begin{equation}\label{e:X-angular}
\frac{X}{\tau_X(X)}\, |\, \{\tau_X(X)>t\} \stackrel{TV}{\longrightarrow} \Theta_X,\ \ \mbox{ as }t\to\infty,
\end{equation}
where (cf Proposition \ref{p:total-Breiman} with $Z:=X$) the distribution of $\Theta_X$ is given by:
$$
\sigma_X(B) = \P(\Theta_X\in B) = \frac{1}{\E[\tau_X(W)]} \E[ 1_B(W/\tau_X(W)) \tau_X(W) ],\ \ B\in {\cal B}(\SX).
$$
Relation \eqref{e:p:Pareto-Breiman-i-Theta} shows that the distribution of $\Theta|\{U<1\}$ is a {\em tilted} version of the
distribution of $\Theta_X$.  That is, even in the case when $\P[U<1]=1$ and $\Theta$ is always uniquely defined, the distributions of $\Theta$
and $\Theta_X$ are in general {\em different}. Indeed, this is naturally attributed to the fact in \eqref{e:p:Pareto-Breiman-i} and \eqref{e:X-angular} 
one conditions on different extreme events $\{Y+\tau_X(X)>t\}$ and $\{\tau_X(X)>t\}$, respectively.    
\end{remark}

The following corollary shows that in the context of the Pareto-Breiman models, 
Theorem \ref{thm:final} yields optimal homogeneous predictors over the entire class of all $\tau_X$-support dominated and asymptotically 
calibrated predictors.

\begin{corollary}\label{c:Pareto-Breiman} Let $Y$ and $X$ be as in 
Proposition \ref{p:Pareto-Breiman}.  Suppose that $g= g^{(\rm opt)}$ is as in Theorem \ref{thm:final}.  That is,
$g\in L_+^1(b\, \cdot\,  p_\Theta)$, with $b(\theta)$ and $p_\Theta$ as in \eqref{e:b(theta)} and \eqref{e:p_U,Theta}, 
$C(g) = \E[(1-U)g(\Theta)] = \E[U]$, and 
$$
\Lambda(g) = \sup_{f\in L_+^1(b\, \cdot\,  p_\Theta),\ C(f) = \E[U]} \Lambda(f).
$$

Define the homogeneous extension of $g$:
$$
 h_g(x):= \tau_X(x) g(x/\tau_X(x)),\ \  x\in \R^d,
$$
where by convention $h_g(x) = 0$ whenever $\tau_X(x)=0$.  Then:\\

{\em (i)} We have
\begin{equation}\label{e:c:Pareto-Breiman-i}
\lim_{t\to\infty} t\P[Y>t] =  \lim_{t\to\infty}  t\P[h_g(X)>t] = c_\xi \E[\tau(V,W)] C(g).
\end{equation}

{\em (ii)} For every non-negative, homogeneous and Borel measurable function $h:\R^d\to \R_+$ such that $\{h>0\} \subset\{\tau_X>0\}$ and
$\P[Y>t] \sim \P[ h(X)>t],\ t\to\infty$, we have $h\vert_{\SX}\in L_+^1(b\, \cdot\,  p_\Theta)$ and
\begin{equation}\label{e:c:Pareto-Breiman-ii}
\Lambda(h) = \lambda(Y,h(X)) = \lim_{t\to\infty} \P[Y>t | h(X) >t] \le \Lambda(g) = \lambda(Y,h_g(X)). 
\end{equation}
where the above limits exist and $\Lambda$ is as in \eqref{e:Lambda-functional}. 
\end{corollary}

\begin{remark}\label{r:subtlety-response} Corollary \ref{c:Pareto-Breiman} clarifies the subtlety highlighted in 
Remark \ref{r:subtlety}.  It shows that for a broad class of regularly varying models, the asymptotic tail 
dependence coefficient $\lambda(Y,h_g(X))$ is always well-defined so long as 
$g\in L_+^1(b\cdot p_\Theta)$.  Moreover, for all $\tau_X$-support dominated homogeneous predictors $h$ that are asymptotically
calibrated, we have $h\vert_\SX \in L_+^1(b\cdot p_\Theta)$. Thus, the domain $L_+^1(b\cdot p_\Theta)$ of the abstract 
optimization problems in Theorem \ref{thm:final} corresponds precisely to all possible calibrated measurable homogeneous predictors.
\end{remark}

\section{Statistical inference} \label{sec:inference}

In this section, the focus is on the statistical inference of the optimal predictors.  We will 
work in the established peaks-over-threshold framework (Section \ref{sec:PoT-and-Contiguity}) and 
obtain general results for the {\em universal extremal consistency} for homogeneous predictors 
constructed in this framework (Section \ref{sec:universal_consistency}). 

\subsection{Setup and main ideas} \label{sec:setup} 
Throughout this section, we let $(Y,X)$  be a jointly regularly random vector in $\R_+\times \R^{d}$ 
in the sense of Assumption \ref{a:X-Y-joint}, where for simplicity, we suppose that the radial function is
$$
 \tau(y,x):= y + \|x\|_1 = y + \sum_{i=1}^d |x_i|, \ \ (y,x)\in \R_+\times \R^d,
$$
with $\|\cdot\|_1$ denoting the $\ell_1$-norm.  In this case, with the 
notation in \eqref{e:tilde-U-and-tilde-Theta} and \eqref{e:Z/tau(Z)-via-tilde-U-and-tilde-Theta}, we have:
\begin{equation}\label{e:wt-U-and-wt-Theta-setup}
\wt U = \frac{Y}{ Y + \|X\|_1},\quad \mbox{ and }\quad (1-\wt U) \wt \Theta = \frac{X}{ Y + \|X\|_1 },\quad \mbox{ where } \wt \Theta:= \frac{X}{\|X\|_1}.
\end{equation}

For the purpose of conditional quantile inference, we shall assume that the convergence of the angular distributions 
in \eqref{e:U-THETA} holds in a stronger, total variation norm sense, i.e.,
\begin{equation} \label{e:wt-U-wt-Theta-convergence}
(\wt U, (1-\wt U)\wt\Theta) \mid \{\tau(Y,X)>t\}\overset{TV}{\longrightarrow}\Theta_{(Y,X)}=:(U,(1-U)\Theta),\quad as \ t\rightarrow\infty.
\end{equation}

Finally, we shall suppose that the joint limit distribution $P_{(U,\Theta)}$ of $U$ and $\Theta$ can be
written as
\begin{equation} \label{e:P_U,Theta-TV}
P_{(U,\Theta)}(du,d\theta) = f_{U|\Theta}(u|\theta) du p_\Theta(d\theta),
\end{equation}
where the Lebesgue density of the conditional distribution $U|\Theta$, is assumed to be
{\em positive}, i.e., 
\begin{equation}\label{e:f_U|Theta>0}
 f_{U|\Theta}(u|\theta) >0\quad \mbox{modulo}\ \ dup_\Theta(d\theta).
 \end{equation}
 
Recall that $q_\alpha(\theta)$, defined in \eqref{e:q-alpha}, stands for the left-continuous 
$\alpha$-quantile of the {\em weighted} conditional density $\propto (1-u)f_{U|\Theta}(u|\theta)$.  
Then, \eqref{e:f_U|Theta>0} implies that $\alpha \mapsto q_\alpha(\theta) \in (0,1)$
is {\em continuous} and {\em strictly increasing} in $\alpha\in (0,1)$ modulo $p_\Theta(d\theta)$.  In view of 
Theorem \ref{thm:final}, we then have that for all $0<\alpha<1$,
\begin{equation}\label{e:g_alpha-optimal}
g_\alpha(x):= \|x\| \cdot \frac{q_\alpha(x/\|x\|)}{1-q_\alpha(x/\|x\|)},\ \ x\in \R^d,
\end{equation}
yields an optimal homogeneous predictor.  Specifically,
\begin{equation}\label{e:Lambda-optimal}
 \Lambda(g_\alpha)  = \sup_{g \in D(g_\alpha)} \Lambda(g),
\end{equation}
where $\Lambda(\cdot)$ is as in \eqref{e:Lambda-functional}. Recall that the optimization
domain $D(g_\alpha)$ consists of all measurable positive homogeneous functions such that 
$$
C(g) = \E[ (1-U) g(\Theta) ] = C(g_\alpha),
$$
where $C(\cdot)$ stands for the constraint functional in \eqref{e:C(g)}. Finally, notice that the optimal value
is achieved by a {\em continuous} predictor $(\alpha,x) \mapsto g_\alpha(x)$.

\medskip
\noindent{\em  Summary of the main ideas.}  Given the above discussion, obtaining an asymptotically optimal homogeneous predictor of 
$\{Y>t\}$ amounts to  estimating the quantile functions $q_\alpha(\cdot)$, and finding the value of 
$\alpha \in (0,1)$ such that the predictor in \eqref{e:g_alpha-optimal} is calibrated.  This, in view of the continuity of $g_\alpha$, amounts to having 
\begin{equation}\label{e:C(g_alpha*)}
    C(g_{\alpha^*}) = \E[U],
\end{equation}
because $\P[g_\alpha(X)>t]/\P[Y>t] \to C(g_\alpha)/\E[U],$ as $t\to\infty$, by Proposition \ref{p:rv-via-h}.  This is the main idea
behind the optimal predictor inference methodology described in Sections \ref{sec:methodology} and \ref{sec:algorithm}. 
Theorem \ref{thm:univ_consitency} (below) shows that this methodology will yield predictors with asymptotically optimal precision.
To make this precise, however, one needs to address several details.

If one had iid samples $(U_i,\Theta_i),\ i=1,\cdots,n$ from the limiting angular distribution $P_\infty:= P_{(U,\Theta)}$, then $q_\alpha(\cdot)$
can be consistently estimated (as $n\to\infty$) in many ways. For example, one can use classical quantile 
regression with a weighted loss or by using weighted quantile random forest estimators \cite{meinshausen2006quantile,EMD2022qf} (cf Section \ref{sec:methodology}).

In practice, however, one does not have iid samples from the limiting angular distribution, 
but instead a triangular array of samples $(\wt U_i,\wt \Theta_i)$'s coming from the distribution 
$P_{t_n}$ of $(\wt U,\wt \Theta)\, |\, \{\tau(Y,X)>t_n\}$ indexed
by a threshold $t_n\to\infty$. Thus, to show that the quantile estimates $\wh q_\alpha$ based on these data are asymptotically 
consistent, one needs an additional step that controls the rate at which $P_{t_n}$ converges to $P_\infty$.  This will be done 
by appealing to a Le Cam-type contiguity argument in Section \ref{sec:PoT-and-Contiguity}.  This step will subsequently allow us
to establish general results that can be thought of as extremal counterparts to universal consistency (Theorem \ref{thm:univ_consitency} in 
Section \ref{sec:universal_consistency}).

\subsection{Peaks-over-threshold: A contiguity perspective}  \label{sec:PoT-and-Contiguity}


Let $\{(Y_i,X_i)\}_{i=1}^n$ be independent copies of $(Y,X)$ and define the radii
\[
R_i := \tau(Y_i,X_i) = Y_i + \|X_i\|_1,
\qquad i=1,\ldots,n.
\]
Given a threshold $t=t_n\to \infty$,  consider the associated index set of exceedances:
\[
\mathcal{I}_n := \{ i \in \{1,\ldots,n\} : R_i > t_n \}.
\]

Following \eqref{e:wt-U-and-wt-Theta-setup}, the angular components of these exceedance observations are denoted as
\begin{equation}\label{e:wt-U_i-wt-Theta_i}
\wt U_i := \frac{Y_i}{R_i},\quad \mbox{ and }\ \ \wt \Theta_i := \frac{X_i}{\|X_i\|_1}, \quad i \in \mathcal{I}_n.
\end{equation}
The {\em Peaks-over-Threshold (PoT)} methodology stipulates that by Relation \eqref{e:wt-U-wt-Theta-convergence}, as $t\to\infty$, 
one can regard the $(\wt U_i, \wt \Theta_i),\ i\in {\cal I}_n$ as independent samples coming from $P_{\infty}:= P_{(U,\Theta)}$.
Thus, if $t=t_n\to\infty$ and at the same time the number of exceedances
$$ 
   K(n):= |{\cal I}_n| \to\infty, \ \ \mbox{ as } n\to\infty,
$$ 
one expects to be able to consistently estimate functionals of $P_\infty$ with statistics based on the sample
$(\wt U_i,\wt \Theta_i),\ i\in {\cal I}_n$.  To formalize this idea, one needs to address two issues: (i) The sample size $K(n)$ is random;
(ii) The conditional distribution $P_{t_n}$ of $(\wt U,\wt \Theta)$ given $\tau(Y,X)>t_n$ may approach $P_\infty$ at 
a slow rate.  The first issue is elegantly addressed via the so-called {\em d\`ecoupage de L\'evy}, while the second by 
Proposition \ref{p:contiguity-restatement} using the ideas of contiguity.

\begin{proposition}[D\`ecoupage de L\'evy]\label{p:decpupage} Let $P_t$ be the conditional distribution
$(\wt U, \wt\Theta) | \tau(Y,X)>t$.  Let also $(\wt U_i,\wt \Theta_i),\ i\in {\cal I}_n$ be as in \eqref{e:wt-U_i-wt-Theta_i} and
independently, let $(U_j^*,\Theta_j^*),\ j=1,2,\dots$ be independent realizations from $P_t$.  Then,
$$
\{ (\wt U_i,\wt \Theta_i),\ i\in {\cal I}_n\} \stackrel{d}{=} \{ (U_j^*,  \Theta_j^*),\ j=1,\cdots, K(n)\},
$$
viewed as random point measures.
\end{proposition}

This result is known \citep[see e.g.][or Remark \ref{rem:decoupage-de-Levy} in Section \ref{subsec:contiguity_argument} 
for a proof]{resnick:1987}.  Proposition \ref{p:decpupage} shows that the exceedance 
``angles'' $(\wt U_i,\wt \Theta_i)$ can be regarded as independent draws from the distribution 
$P_t$. Moreover, the random number of exceedances $K(n)$ and their values are independent, i.e., 
can be decoupled, even though they come from the common sample $\{(Y_i,X_i),\ i\in [n]\}$.\\

The following result makes the intuition behind Peaks-over-Threshold formally precise.  It uses a contiguity-type argument 
in the spirit of Le Cam's First Lemma \citep[see e.g.][]{vaart:1998}.

\begin{proposition}[Contiguity] \label{p:contiguity-restatement}
Consider the distributions $P_{\infty}:=P_{(U,\Theta)}$ and 
$P_{t}:= P_{(\wt U,\wt \Theta) | \tau(Y,X)>t}$ above defined. 
Let $(U_i,\Theta_i),\ i=1,2,\dots$ be iid from $P_\infty$ and $(\wt U_i,\wt \Theta_i),\ i=1,2,\dots$ be iid from $P_{t_n}$.
If the random integer sequence $K(n)$ is independent from these samples and such that 
\begin{equation}\label{e:p:contiguity-restatement}
  K(n)\stackrel{\P}{\to}\infty\ \ \mbox{ and }\ \  K(n)\| P_{t_n} - P_{\infty}\|_{\rm TV} 
  \stackrel{\P}{\longrightarrow} 0\quad \text{ as } n\rightarrow \infty,
 \end{equation}
 then for any measurable $T_k = T_k((U_1,\Theta_1),\cdots,(U_k,\Theta_k))$ such that 
 $$
  T_k\stackrel{\P}{\to} 0,\ \ \mbox{ as }k\to\infty
 $$
we have
 $$
 \wt T_{K(n)} := T_{K(n)} \Big((\wt U_1,\wt \Theta_1),\cdots,(\wt U_{K(n)},\wt \Theta_{K(n)})\Big)\stackrel{\P}{\to} 0,\ \ \mbox{ as }n\to\infty.
 $$
\end{proposition}

The proof can be found in Section \ref{subsec:contiguity_argument} (cf Proposition \ref{p:contiguity}).  This result
shows that provided \eqref{e:p:contiguity-restatement} holds, any statistic based on the sample
$\{(\wt U_i,\wt \Theta_i),\ i\in {\cal I}_n\}$ has the same consistency behavior as the same statistic computed from 
$\{(U_i,\Theta_i)\}_{i=1}^{K(n)}$ --  an iid sample from the asymptotic angular distribution $P_\infty = P_{(U,\Theta)}$.
This general contiguity argument allows us to seamlessly extend consistency results for iid samples from the unattainable 
asymptotic distribution $P_\infty$ to the triangular array type setting of peaks-over-threshold. 

\begin{remark} The appearance of the total variation norm in \eqref{e:p:contiguity-restatement} may appear stringent but 
it is in fact rather mild.  Indeed, as shown in Proposition \ref{p:total-Breiman}, for 
the broad class of Breiman-type models we automatically have $\| P_{t_n} - P_{\infty}\|_{\rm TV} \to 0$, as $t_n\to\infty$ and
then ensuring \eqref{e:p:contiguity-restatement} is a matter of choosing the threshold in the latter convergence. Consequently, for every
Breiman-type model one can choose $t_n\to\infty$ so that \eqref{e:p:contiguity-restatement} is holds.    
\end{remark}
 
\begin{remark} By Proposition 1.1 in \cite{bobbia:dombry:varron:2025}, we have
\[
(\wt U_i,\wt \Theta_i)_{i\in \mathcal{I}_n}\overset{d}{=}(U_i^*,\Theta_i^*)_{i=1}^k,
\]
where $(U_i^*,\Theta_i^*)_{i=1}^k$ are iid from $P_{t_n}$, where $t_n$ is a {\em random threshold} given by the
largest $k(n)$-th order statistic of the $R_i$'s. Therefore, the results in this section can be shown to hold in the
case where the number of exceedances $|{\cal I}_n|=k(n)$ is deterministic, while the threshold 
is $t_n$ is random.  
\end{remark}

\subsection{On the universal consistency in optimal homogeneous prediction}\label{sec:universal_consistency}

In this section, we provide a general consistency result for the optimal homogeneous predictors under rather mild conditions.  In 
Section \ref{sec:methodology}, we demonstrate that these conditions are fulfilled for a class of random forest-based 
predictors.  These general consistency results can be viewed as an extreme value counterpart to the classical universal consistency 
theory of Stone \citep[see e.g. Ch.\ 6 in][]{devroye1996probabilistic}. 

 To be precise, the focus here is on showing convergence to the optimal precision (equivalently, loss) rather than on 
 the consistency of the predictors themselves.  Namely, consistency results are stated through the precision functional 
 $\Lambda$ defined in equation \eqref{e:Lambda-functional} for functions $g$ as 
$$
\Lambda(g) = \frac{1}{\E[U]} \E[U \wedge (1-U)g(\Theta)] = \lim_{t\to\infty} \frac{\P[Y>t, g(X)>t]}{\P[Y>t]}; 
$$
Observe that if $g(X)$ is calibrated then we have $\Lambda(g) = \lambda(Y,g(X))$. Hence the aim is to provide asymptotic calibration conditions
and consistency of $\Lambda(\hat g_n)$ for a suitable sequence of predictors $(\hat g_n)_{n\geq 1}$ based on a growing training data set.

 \begin{assumption}\label{assum:conv_TV_RV} Adopt the setup of Section \ref{sec:setup}. We suppose that 
 \eqref{e:wt-U-wt-Theta-convergence} holds, that is, 
 $$
  \|P_t - P_\infty\|_{\rm TV}\longrightarrow 0,\ \ \mbox{ as }t\to\infty,
 $$
 where $P_t = P_{(\wt U,\wt \Theta) | \tau(Y,X)>t}$ and $P_\infty = P_{(U,\Theta)}$.
 Suppose moreover that \eqref{e:P_U,Theta-TV}, and \eqref{e:f_U|Theta>0} hold.
 \end{assumption}

 \begin{assumption}\label{assum:conv_quantile_estim} Adopt the notation of Section \ref{sec:setup} and recall
  $q_\alpha(\theta)$ stands for the quantile of the weighted conditional density $\propto (1-u)f_{U|\Theta}(u|\theta)$,
 (cf \eqref{e:q-alpha}).  Assume that $(\alpha,\theta)\mapsto q_\alpha(\theta)$ is defined and continuous on the domain 
 $(0,1)\times S_\Theta$ for some compact set 
 $S_\Theta \subseteq S_X$ such that $p_\Theta(S_\Theta) = 1$ (e.g., $S_\Theta  = {\rm supp}(p_\Theta)$).
 
 Let $\{(U_i,\Theta_i)\}_{i=1}^k$ be independent samples from the distribution $P_\infty= P_{(U,\Theta)}$.  Assume there exists 
 an estimator $\wh q_{\alpha,k}(\theta)$ of $q_\alpha(\theta)$, based on the sample $\{(U_i,\Theta_i)\}_{i=1}^k$, such that: 
 (i) for all $\alpha\in (0,1)$ and $k\in\N$, the function $\theta \mapsto \wh q_{\alpha,k}(\theta)$ is defined and continuous 
 on the entire unit sphere $\theta\in S_X$; (ii) We have
 $$
 0< \wh q_\alpha(\theta) <1,\ \ \mbox{ for all } \alpha\in (0,1)\ \mbox{ and }\ \theta\in S_X\mbox{;}
 $$
 (iii) For all $0<\alpha_0<\alpha_1<1$:
\[
\sup_{\alpha \in [\alpha_0,\alpha_1] } \|\wh q_{\alpha,k}(\cdot) -q_\alpha(\cdot)\|_{\infty,S_\Theta} :=
\sup_{\alpha \in [\alpha_0,\alpha_1] } \sup_{\theta\in S_\Theta} |\wh q_{\alpha,k}(\theta) -q_\alpha(\theta)|
\stackrel{\mathbb{P}}{\longrightarrow} 0,\quad \mbox{ as }k\to\infty.
\]
 \end{assumption}

 Given the quantile estimators in Assumption \ref{assum:conv_quantile_estim}, define
 \begin{equation}\label{e:wh_g_alpha,k}
 \wh g_{\alpha, k} (x) := \|x\| \cdot \frac{\wh q_{\alpha,k}(x/\|x\|)}{1-\wh q_{\alpha,k}(x/\|x\|)},\ \ x\in \R^d,
 \end{equation}
 where $ \wh g_{\alpha,k} (x):= 0$.

 Now, consider the Peaks-over-Threshold methodology of Section \ref{sec:PoT-and-Contiguity}.  That is, let
 $(Y_i,X_i),\ i\in [n]$ be independent copies of a jointly regularly varying random vector $(Y,X)$ in $\R_+\times \R^d$
 and focus on the random sample of exceedance-based angles $\{(\wt U_i,\wt \Theta_i),\ i\in {\cal I}_n\}$ in 
 \eqref{e:wt-U_i-wt-Theta_i}, for a sequence of thresholds $t=t_n\to\infty$.  Let also $\alpha_n$ be a sequence of 
 estimators for $\alpha^*$ in \eqref{e:C(g_alpha*)} and define the predictors 
 \begin{equation}\label{e:wh_g_n}
 \wh g_n(x):= \wh g_{\alpha_n, K(n)}(x),
 \end{equation}
 computed with the $(U_i,\Theta_i)$'s replaced by the the $(\wt U_i,\wt \Theta_i)$'s, where $K(n) = |{\cal I}_n|$.
 
 The following theorem shows the {\em universal extremal consistency} 
 of the predictors $\{\wh g_{K(n)}(X)>t\}$ for the events $\{Y>t\}$, as $t\to\infty$ under mild conditions. That is,
 we will show that these event predictors are asymptotically calibrated and that their precision is asymptotically optimal.

\begin{theorem} \label{thm:univ_consitency} Let  $(Y,X)$ and $(Y_i,X_i),\ i=1,2,\cdots,n$ be iid.
Adopt Assumptions \ref{assum:conv_TV_RV} and \ref{assum:conv_quantile_estim} and let $\wh g_n(\cdot)=\wh g_{\alpha_n, K(n) }(\cdot)$ 
be as in \eqref{e:wh_g_n}.  

If $\alpha_n \stackrel{\P}{\to} \alpha^*$ with $\alpha^*$ as in \eqref{e:C(g_alpha*)}, and the thresholds 
$t_n\to\infty$ are such 
that
\begin{equation}\label{eq:CV_rate_UC_theorem-new}
n\mathbb{P}(\tau(Y,X)>t_n)\| P_{t_n} - P_\infty\|_{\rm TV} \longrightarrow 0\quad \text{ as } n\rightarrow \infty,
\end{equation}
then:

{\em (i)} We have that:
\begin{equation} \label{e:thm:univ_consitency-i} 
       \wh c_n:= \lim_{t\to\infty} \frac{\P[\wh g_n(X)>t | \wh g_n]}{\P[Y>t]} \stackrel{\P}{\longrightarrow} 1,\ \ \mbox{ as }
       n\to\infty,
\end{equation}
and consequently, for $\wh c_n>0$, 
\begin{equation} \label{e:thm:univ_consitency-i-lambda} 
 \lambda(Y,\wh g_n(X) | \wh g_n) = \lim_{t\to\infty} \P[ \wh g_n(X) > \wh c_n t | Y>t,\ \wh g_n] =  \Lambda(\wh g_n(\cdot)/\wh c_n).
\end{equation}

{\em (ii)} The predictors $\wh g_n (\cdot)$ and $ \wh g_n(\cdot)/\wh c_n$ are both precision-optimal:
\begin{align}\label{e:thm:univ_consitency-ii} 
   \mathop{\P\lim}_{n\to\infty} \Lambda(\wh g_n/\wh c_n) = \mathop{\P\lim}_{n\to\infty} \Lambda(\wh g_n)
   =\Lambda (g_{\alpha^*}),
\end{align}
where $\P\lim$ denotes the limit in probability.

{\em (iii)} If, moreover, for some sequence of non-negative, continuous, and homogeneous functions $\wh h_n(\cdot)$, independent from $(Y,X)$, 
we have that
\begin{equation} \label{e:thm:univ_consitency-iii} 
\wh d_n := \lim_{t\to\infty} \frac{\P[\wh h_n(X)>t | \wh h_n (\cdot) ]}{\P[Y>t]} \stackrel{\P}{\longrightarrow} 1,\ \ \mbox{ as }n\to\infty,
\end{equation}
then for all $\epsilon>0$, we have
\begin{equation}\label{e:thm:univ_consitency-iii-claim} 
\lim_{n\to\infty} \P[  \Lambda(\wh h_n) \le  \Lambda(\wh g_n) + \epsilon ] = 1.
\end{equation}
\end{theorem}

 We will first comment on the consequences of this result and then present its proof.
 
\begin{remark} Observe that Relations \eqref{e:thm:univ_consitency-i} and \eqref{e:thm:univ_consitency-iii} involve conditioning on 
the predictors $\wh g_n$ and/or $\wh h_n$, which are built by using the training data $\{(Y_i,X_i)\}_{i=1}^n$.  One can interpret these
results as follows.  Having a new observation $(Y,X)$, independent from $\{(Y_i,X_i)\}_{i=1}^n$, we seek a predictor $\{\wh g_n(X)>t\}$ for
$\{Y>t\}$, which is asymptotically calibrated, in the sense that \eqref{e:thm:univ_consitency-i} holds.  That is, given the data, the 
rate at which we sound an alarm $\P[\wh g_n(X)>t | \wh g_n]$ is asymptotically equivalent to the rate $\P[Y>t]$ of the event, as $t\to\infty$ and as
the sample size $n\to\infty$.

Adopting this {\em extremal perspective}, where the rate of the event $\P[Y>t]$ vanishes as $t\to\infty$, Relation 
\eqref{e:thm:univ_consitency-iii-claim} shows that the predictor $\wh g_n$ has an asymptotically optimal precision among
all continuous homogeneous predictors, calibrated in the sense of \eqref{e:thm:univ_consitency-iii}.
\end{remark}

\begin{remark} \label{rem:thm:univ_consitency-precision} In the context of Theorem \ref{thm:univ_consitency}, we have
\begin{equation} \label{e:rem:thm:univ_consitency-precision}
\lambda( Y, \wh h_n(X) | \wh h_n) \le \Lambda(g_{\alpha^*}) = 
\mathop{\P\lim}_{n\to\infty} \lambda( Y, \wh g_n(X) | \wh g_n).
\end{equation}
This  follows from the observation that
$\lambda( Y, \wh h_n(X) | \wh h_n) = \Lambda(\wh h_n(\cdot)/\wh d_n)$ and the fact that $\wh h_n(\cdot)/\wh d_n$ is calibrated,
which trivially implies the inequality in \eqref{e:rem:thm:univ_consitency-precision}. The equality therein follows 
from \eqref{e:thm:univ_consitency-i-lambda} and \eqref{e:thm:univ_consitency-ii}.

Relation \eqref{e:rem:thm:univ_consitency-precision} states that the asymptotic precision of {\em any calibrated} 
continuous homogeneous predictor $\wh h_n$, is no better than the optimal value $\Lambda(g_{\alpha^*})$.  
This, while appealing, is of limited interest in practice,
since one often does not have knowledge of the calibration constants $\wh d_n$ in \eqref{e:thm:univ_consitency-iii}.  In practice, 
one provides decisions based on whether the event $\{\wh h_n(X)>t\}$ occurs. Therefore,  for a 
finite $n$, the precision functional
$$
\Lambda(\wh h_n)= \mathop{\P\lim}_{n\to\infty} \P[ \wh h_n(X)>t |  Y> t,\ \wh h_n] 
$$
is of greater interest than the tail-dependence coefficient $\lambda(Y,\wh h_n(X)| \wh h_n)$.  This is why we stated
Theorem \ref{thm:univ_consitency} primarily in terms of the precision functionals $\Lambda$ rather than via tail-dependence 
coefficients.
\end{remark}

\begin{proof}[Proof of Theorem \ref{thm:univ_consitency}] We will argue first that
\begin{equation}\label{e:thm:univ_consitency-0}
\|\wh g_n - g_{\alpha^*} \|_{\infty,S_\Theta} \stackrel{\P}{\longrightarrow} 0,\ \ \mbox{ as }n\to\infty.
\end{equation}
Indeed, recalling \eqref{e:wh_g_n}, one can write 
\begin{align}\label{e:thm:univ_consitency-1}
\sup_{\theta\in S_\Theta} | \wh q_n(\theta) - q_{\alpha^*}(\theta) | &
  \le \sup_{\theta\in S_\Theta} | \wh q_{\alpha_n,K(n)} (\theta) - q_{\alpha_n}(\Theta)| + 
\sup_{\theta\in S_\Theta}| q_{\alpha_n}(\theta) -q_{\alpha^*}(\theta) |  \nonumber \\
&=: T_{K(n)}( (\wt U_1,\wt \Theta_1),\cdots, (\wt U_{K(n)},\wt \Theta_{K(n)})) + \| q_{\alpha_n}(\cdot) -q_{\alpha^*}(\cdot) \|_{\infty,S_\Theta},
\end{align}
for some measurable functions $T_k:((0,1)\times S_X)^{\times k} \to [0,1],\ k\in \N$.

Consider the first term in \eqref{e:thm:univ_consitency-1}.  By Assumption \ref{assum:conv_quantile_estim},  
since $\alpha_n\stackrel{\P}{\to} \alpha^* \in (0,1)$, as $n\to\infty$, we have 
$$
T_k((U_1,\Theta_1),\cdots, (U_{k}, \Theta_{k})) \stackrel{\P}{\longrightarrow} 0,\ \ \mbox{ as }k\to\infty,
$$
where the $(U_i,\Theta_i)$'s are now iid from $p_{(U,\Theta)}$.  On the other hand, 
$$ 
 K(n) = |{\cal I}_n|  = {\cal O}_\P(\E[K(n)]) =  {\cal O}_\P( n \P[ \tau(Y,X)>t_n]),\ \ \mbox{ as }n\to\infty,
$$
and therefore, the contiguity argument in Proposition \ref{p:contiguity-restatement} implies
$$
T_{K(n)}( (\wt U_1,\wt \Theta_1),\cdots, (\wt U_{K(n)},\wt \Theta_{K(n)})) \stackrel{\P}{\longrightarrow} 0,\ \ \mbox{ as } n\to\infty.
$$
This shows that the right-hand side of \eqref{e:thm:univ_consitency-1} converges in probability to $0$, as $n\to\infty$, because
$\| q_{\alpha_n} -q_{\alpha^*} \|_{\infty,S_\Theta}\stackrel{\P}{\to} 0,\ n\to\infty$, by 
the continuity of $(\alpha,\theta)\mapsto q_\alpha(\theta)$ on $(0,1)\times S_\Theta$ and since
$\alpha_n\stackrel{\P}{\to}\alpha^* \in (0,1)$.

We have thus shown that 
\begin{equation}\label{e:thm:univ_consitency-2}   
  \|\wh q_n - q_{\alpha^*}\|_{\infty,S_\Theta} \stackrel{\P}{\to} 0,\ \ \mbox{ as }n\to\infty.
\end{equation}
Now, recall
\eqref{e:wh_g_alpha,k} and observe that 
$$
 \wh g_n(\theta) := \Psi(\wh q_n)(\theta)\quad \mbox{ and }\quad g_{\alpha^*}(\theta) := \Psi(q_{\alpha^*})(\theta),
$$
where $\Psi : C(S_\Theta; (0,1)) \to C(S_\Theta; (0,\infty))$ is the functional $\Psi( q ) (\theta)  := q(\theta)/(1-q(\theta))$, and where
$C(S_\theta; (a,b))$ denotes the space of continuous functions $q:S_\Theta \to (a,b)$ defined on the compact $S_\Theta$ and taking values 
in $(a,b)$.  Observe that the functional $\Psi: C(S_\Theta; (0,1)) \to C(S_\Theta; (0,\infty)) $ is continuous in
$\|\cdot\|_{\infty,S_\Theta}$, at all functions $q\in C(S_\Theta; (0,1))$ such that $\max_{\theta\in S_\Theta}q(\theta) <1$.
Thus, by the fact that $\alpha^*<1$, we have $q_{\alpha^*}(\theta^*):=\max_{\theta\in S_\Theta} q_{\alpha^*}(\theta) < 1,$ and
the continuity of $\Psi$ at $q_{\alpha^*}(\cdot)$ combined with the established Relation \eqref{e:thm:univ_consitency-2} implies
\eqref{e:thm:univ_consitency-0}.\\

{\em Part (i).} We shall now prove \eqref{e:thm:univ_consitency-i}.   Assumption \ref{assum:conv_quantile_estim} implies that for all $\alpha_n\in (0,1)$, 
the quantile estimator $\wh q_{\alpha_n}(\theta)$ is continuous in $\theta$, positive, and less than $1$ for all 
$\theta\in S_X$.  Therefore, $\wh g_n (\theta) = \wh q_{\alpha_n}(\theta)/(1-\wh q_{\alpha_n}(\theta))$ is continuous 
positive and finite for all $\theta\in S_X$.  Hence, the homogeneous functions $(y,x) \mapsto y$ and  
$(y,x) \mapsto \wh g_n(x):=\|x\| \wh g_n(x/\|x\|),\ x\in \R^d$ satisfy the assumptions
of Proposition \ref{p:rv-via-h}, and thus
\begin{align*}\label{e:thm:univ_consistency-5-0}
\lim_{t\to\infty} b(t)\P[ \wh g_n(X) > t | \wh g_n ] &= c_Z \E[(1-U) \wh g_n(\Theta) | \wh g_n] \ \mbox{ and }\
\lim_{t\to\infty} b(t) \P[Y>t] = c_Z \E[U],
\end{align*}
which implies
\begin{equation}\label{e:thm:univ_consistency-5}
\lim_{t\to\infty} \frac{\P[ \wh g_n(X) > t | \wh g_n ]}{\P[Y>t]} = \frac{1}{\E[U]} \E [(1-U) \wh g_n(\Theta) | \wh g_n],\ \ \mbox{ as }t\to\infty.
\end{equation}

Therefore, to prove \eqref{e:thm:univ_consitency-i}, it remains to show that the right-hand 
side of \eqref{e:thm:univ_consistency-5} converges in probability to $1$, as $n\to\infty$.
By the calibration result in \eqref{e:C(g_alpha*)}, we have $\E (1-U)g_{\alpha^*}(\Theta)] = \E[U]$, and hence
\begin{align} \label{e:thm:univ_consitency-3}  
\Big|\frac{1}{\E[U]} \E[ (1-U) \wh g_n(\Theta) | \wh g_n] -1 \Big|  
&=\frac{1}{\E[U]}  \Big| \E[ (1-U) ( \wh g_n(\Theta) - g_{\alpha^*}(\Theta)) | \wh g_n] \Big| \nonumber \\
&\le \frac{1}{\E[U]}  \| \wh g_n - g_{\alpha^*}\|_{\infty,S_\Theta} \stackrel{\P}{\to} 0,
\end{align}
as $n\to\infty$, by the established Relation \eqref{e:thm:univ_consitency-0}.  This completes the proof of \eqref{e:thm:univ_consitency-i}.

To show \eqref{e:thm:univ_consitency-i-lambda}, observe that \eqref{e:thm:univ_consitency-i} implies that 
$t\mapsto \P[ \wh g_n(X) > t | \wh g_n]$ is regularly varying with tail exponent $-1$ whenever $\wh c_n>0$. Therefore,
for $\wh c_n>0$, we have
$$
\P[ \wh g_n(X) > \wh c_n t | \wh g_n ] \sim \wh c_n^{-1} \P[\wh g_n(X)> t],\ \ \mbox{ as }t\to\infty,
$$
which in view of \eqref{e:thm:univ_consitency-i}, entails that $\wh g_n(\cdot)/\wh c_n$ is asymptotically calibrated in the sense of 
\eqref{e:g-calibration-lemma}. Lemma \ref{lem:nonameidea}, then implies \eqref{e:thm:univ_consitency-i-lambda}.\\

{\em Part (ii).} The proof of \eqref{e:thm:univ_consitency-ii} is similar to that of \eqref{e:thm:univ_consitency-i}.  Indeed, conditionally 
on $\wh g_n$, by Proposition \ref{p:rv-via-h} applied to the continuous and homogeneous functions 
$(y,x)\mapsto (y\wedge \wh g_n(x))$ and $(y,x)\mapsto y$, we obtain:
\begin{equation}\label{e:thm:univ_consitency-4}  
\Lambda(\wh g_n) = \lim_{t\to \infty} \P[ \wh g_n(X)>t | Y>t,\ \wh g_n] = \frac{\E[ U \wedge (1-U) \wh g_n(\Theta) | \wh g_n]}{\E[U]}.
\end{equation}
The latter, as argued in \eqref{e:thm:univ_consitency-3}, in view of \eqref{e:thm:univ_consitency-0} and
the Dominated Convergence Theorem, converges as $n\to\infty$, to $\Lambda(g_{\alpha^*}) = \E[U\wedge (1-U) g_{\alpha^*}(\Theta)]$.
Relation \eqref{e:thm:univ_consitency-0} and the fact that $\wh c_n \stackrel{\P}{\to} 1,\ n\to\infty$, imply also that
$\| \wh g_n(\cdot) /\wh c_n - g_{\alpha^*}\|_{\infty, S_\Theta}\stackrel{\P}{\to}0$.  Thus a similar dominated convergence argument 
with $\wh g_n$ replaced by $\wh g_n(\cdot)/\wh c_n$ entails also
$\Lambda(\wh g_n(\cdot) /\wh c_n) \stackrel{\P}{\to} \Lambda(g_{\alpha^*}),$ as $n\to\infty$.\\

 {\em Proof of (iii):} Relations \eqref{e:P_U,Theta-TV} and \eqref{e:f_U|Theta>0} in Assumption \ref{assum:conv_TV_RV} imply that 
 the quantile functions
 $$
 \beta \mapsto q_\beta(\theta)
 $$
 are {\em strictly increasing} and {\em continuous}, and such that
 \begin{equation}\label{e:q_beta-limits}
 q_\beta(\theta) \downarrow 0\ \ (\beta\downarrow 0)\quad \mbox{ and }\quad q_\beta(\theta) \uparrow 1\ \ (\beta\uparrow 1),
 \end{equation}
 modulo $p_\Theta(d\theta)$.
 
 These observations, the Monotone Convergence Theorems, and the fact that $u\mapsto u/(1-u)$ is strictly increasing and continuous, entail that
 the functions
 $$
  \beta \mapsto C(g_\beta) = \E[ (1-U) g_\beta(\Theta)]\ \ \mbox{ and }\ \ \beta\mapsto \Lambda(g_\beta) = \E[ U \wedge (1-U) g_\beta(\Theta)]/\E[U],
 $$
 are continuous (see also Lemma \ref{l:C(g)-properties}).  Moreover, the function $\varphi(\beta):= C(g_\beta)$ is strictly increasing
 in $\beta$ and such that
 \begin{equation}\label{e:C(g_beta)-limits}
 C(g_\beta) \downarrow 0 \ \ (\beta\downarrow 0)\quad \mbox{ and }\quad C(g_\beta) \uparrow \infty\ \ (\beta\uparrow 1).
 \end{equation}

 We are now ready to prove \eqref{e:thm:univ_consitency-iii-claim}.  By \eqref{e:C(g_beta)-limits} and the continuity of  
 $\beta \mapsto C(g_\beta)$, it follows that
  \begin{equation}\label{e:wh-hn-constraint}
  C(\wh h_n) = C(g_{\wh \beta_n}),
 \end{equation}
 where in fact the sequence of random variables $\wh \beta_n\in (0,1)$ is unique.  Further, the fact that for all $\beta\in (0,1)$, 
 we have that $g_\beta$ is optimal in the sense of \eqref{e:Lambda-optimal} and since by \eqref{e:wh-hn-constraint}, we have $\wh h_n\in D(g_{\wh \beta_n})$,
 we obtain
 \begin{equation}\label{e:Lambda-h_n<Lambda-g_beta_n}
 \Lambda(\wh h_n) \le \Lambda (g_{\wh \beta_n}).
 \end{equation}
 Now, notice that  Proposition \ref{p:rv-via-h}, applied to the homogeneous functions $(x,y)\mapsto \wh h_n(x)$ and $(x,y)\mapsto y$ and 
 the conditional measure $\P(\cdot | \wh h_n)$, implies that 
 $$
 \lim_{t\to\infty} \frac{\P[\wh h_n(X) >t  | \wh h_n]}{\P[Y>t]} =  \frac{C(\wh h_n)}{\E[U]}.
 $$
 Thus, by \eqref{e:thm:univ_consitency-iii}, yields 
 $C(\wh h_n) = C(g_{\wh \beta_n}) \stackrel{\P}{\to} \E[U] = C(g_{\alpha^*})$, as $n\to\infty$.  
 The strict monotonicity of $\varphi(\beta)= C(g_\beta)$ in $\beta$ implies, however, that $\varphi^{-1}$ is continuous and hence
 $$
  \wh \beta_n = \varphi^{-1}\Big(C(g_{\wh \beta_n}) \Big) \stackrel{\P}{\longrightarrow} \varphi^{-1}\Big( C(g_{\alpha^*}) \Big) = \alpha^*,
  \ \ \mbox{ as }n\to\infty.
 $$
 The last convergence, the continuity of $\beta\mapsto g_\beta(\theta)$ modulo $p_\Theta(d\theta)$, and the Dominated Convergence 
 Theorem, entail
 \begin{equation}\label{e:Lambda_g_beta_n-to-Lambda_g_opt}
     \Lambda(g_{\wh \beta_n}) = \frac{1}{\E[U]} \E[ U \wedge (1-U) g_{\wh \beta_n}(\Theta) | \wh \beta_n ] \stackrel{\P}{\longrightarrow}
     \Lambda(g_{\alpha^*}),\ \ \mbox{ as }n\to\infty.  
 \end{equation}

 Relations \eqref{e:thm:univ_consitency-ii}, \eqref{e:Lambda-h_n<Lambda-g_beta_n},
 and \eqref{e:Lambda_g_beta_n-to-Lambda_g_opt} imply \eqref{e:thm:univ_consitency-iii-claim}, 
 completing the proof of the theorem.
\end{proof}

 \begin{remark} The existence of a positive density (cf  \eqref{e:P_U,Theta-TV} and \eqref{e:f_U|Theta>0}) in Assumption \ref{assum:conv_TV_RV}
 can be relaxed. The key idea behind proving \eqref{e:thm:univ_consitency-iii-claim} is the continuity of the 
 optimal value $\inf_{g'\in D(g)} \Lambda(g')$ relative to the value of the constraint $C(g)$. Thus, one can extend 
 \eqref{e:thm:univ_consitency-iii-claim} to the case of discontinuous quantile functions $q_\alpha(\cdot)$ by considering 
 \eqref{e:q-alpha-lambda} and applying the method of interpolation used in the proof of Theorem \ref{thm:final}.  
 \end{remark}

\section{Methodology and numerical examples}
  \label{sec:methodology}

   \subsection{Weighted quantile inference} 
    
\label{sec:quant_reg_generality}

Consider a random variable $\xi$ whose conditional density given
$\Theta=\theta$ is defined by
\[
f_{\xi \mid \Theta=\theta}(u)
\propto
(1-u)f_{U\mid \Theta}(u\mid \theta),
\]
and assume that Assumption~\ref{assum:conv_quantile_estim} holds.

Let
\[
\rho_\alpha(u)
=
\bigl(\alpha-\mathds{1}_{\{u\le 0\}}\bigr)u,
\qquad \alpha\in(0,1),
\]
denote the standard check loss function.

Under Assumption~\ref{assum:conv_quantile_estim}, the conditional density
$f_{\xi\mid \Theta=\theta}$ is continuous and strictly positive in a neighborhood of its $\alpha$-quantile
\[
q_{\alpha,\xi}(\theta)
:=
F^{-1}_{\xi\mid\Theta=\theta}(\alpha),
\qquad
\theta\in\mathbb S^{d-1}.
\]
It is well known (see, e.g., \cite{Portnoy_Koenker_1997}) that, for every fixed
$\theta\in\mathbb S^{d-1}$, the conditional quantile
$q_{\alpha,\xi}(\theta)$ is the unique minimizer of
\[
v
\longmapsto
\mathbb E\!\left[
\rho_\alpha(\xi-v)
\,\middle|\,
\Theta=\theta
\right].
\]

Since $q_\alpha$ is assumed to be continuous, the conditional quantile admits the representation
\begin{equation}
\label{e:q-argmin}
q_\alpha(\theta)
=
\underset{v\in(0,1)}{\operatorname{argmin}}
\,
\mathbb E\!\left[
\rho_\alpha(\xi-v)
\,\middle|\,
\Theta=\theta
\right].
\end{equation}

Observe that $q_\alpha(\theta) = {\rm argmin}_{v\in (0,1)} M_\alpha(v|\theta)$, where 
$
M_\alpha(v|\theta )
:=
\mathbb E\!\left[
\rho_\alpha(U-v)(1-U) | \Theta= \theta
\right].
$
Since $v\mapsto M_\alpha(v|\theta)$ is convex and subdifferentiable, its minimizer is characterized by the first-order optimality condition
\[
\alpha
=
\frac{
\mathbb E\!\left[
\mathds{1}_{\{U\le v\}}(1-U)
\,\middle|\,
\Theta=\theta
\right]
}{
\mathbb E\!\left[
1-U
\,\middle|\,
\Theta=\theta
\right]
}.
\]
The right-hand side is precisely the conditional distribution function of
$\xi$ given $\Theta=\theta$. This observation leads to the following weighted quantile estimator \cite{yu1998local,fan1996local}:
\begin{equation}
\label{e:q-hat}
\widehat q_\alpha(\theta)
:=
\widehat F^{-1}_{k,\xi\mid\Theta=\theta}(\alpha),
\end{equation}
where
\[
\widehat F_{k,\xi\mid\Theta=\theta}(a)
=
\sum_{i=1}^{k}
\frac{
\omega_h(\Theta_i,\theta)(1-U_i)
}{
\sum_{j=1}^{k}
\omega_h(\Theta_j,\theta)(1-U_j)
}
\,
\mathds{1}_{\{U_i\le a\}}.
\]

Here $(U_i,\Theta_i)_{i=1}^{k}$ are independent copies of $(U,\Theta)$, and
\[
\omega_h(\theta,\theta')
=
\omega\!\left(
\frac{d(\theta,\theta')}{h}
\right),
\]
with
\[
d(\theta,\theta')
=
\arccos(\theta^\top\theta')
\]
denoting the geodesic distance on the sphere $\mathbb S^{d-1}$ and $h$ a bandwidth depending on $k$.

For a fixed $\theta\in\mathbb S^{d-1}$, consistency of the estimator
$\widehat q_\alpha(\theta)$ follows, as long as $kh(k)^{d-1}$ goes to $0$ as $n$ goes to infinity, from standard arguments because the density is strictly positive in kernel estimation; see, for instance, Section~1.2 of \cite{tsybakov2009} for kernel estimator and \cite{resnick:1987} Proposition~0.1 for the convergence of inverse functions. Note that a wide range of weights functions $\omega$ can be chosen in practice. 
As an example, implementation with quantile regression forests \cite{meinshausen2006quantile} are discussed in Supplements 
\ref{sec:algorithm} and \ref{sec:solar-flares}.

However, pointwise consistency alone is not sufficient to guarantee the uniform convergence required in Assumption~\ref{assum:conv_quantile_estim}. The estimator $\widehat q_\alpha$ can nevertheless be modified to obtain a 
uniformly consistent version.

\begin{proposition}
Let $q_{n,\alpha}(\theta)$ be a consistent (in probability) estimator of
$q_\alpha(\theta)$ for all $\theta\in K$. Assume that the mapping
\[
(\alpha,\theta)
\longmapsto
q_\alpha(\theta)
\]
is continuous on $(0,1)\times K$.
Then there exist a sequence $m(n)\to\infty$ and a transformed estimator
$\widehat q_{n,\alpha}(\theta)$ such that:

\begin{enumerate}
\item[(i)]
The mapping
\[
(\alpha,\theta)
\longmapsto
\widehat q_{n,\alpha}(\theta)
\]
is continuous on $(0,1)\times {\mathbb S}^{d-1}$.\\

\item[(ii)]
The uniform convergence
\begin{equation}
\label{e:unif_conv_wh_qn}
\sup_{\theta\in K}
\left|
\widehat q_{n,\alpha}(\theta)
-
q_\alpha(\theta)
\right|
\stackrel{\mathbb P}{\longrightarrow}
0,
\qquad
n\to\infty,
\end{equation}
holds.
\end{enumerate}
\end{proposition}

\begin{proof}
The result follows directly from Proposition~\ref{p:unif_cont}.
\end{proof}

An explicit construction of the estimator $\widehat q_{n,\alpha}(\cdot)$ is provided in the proof 
of Proposition~\ref{p:unif_cont}. Note that the sequence $m(n)$ therein is not very explicit
since one only requires the pointwise consistency of $q_{n,\alpha}(\theta)$.  If one has more 
information on the rate of this consistency, then a conservative choice of $m(n)$ can be obtained. 
For more details, see Remark \ref{rem:m(n)-explicit}.

\begin{remark}
As it is typical in the PoT framework (cf Section \ref{sec:PoT-and-Contiguity}), 
the formal consistency result for $\wh q_\alpha(\cdot)$ to the asymptotic quantity 
$q_\alpha(\cdot)$, would require taking the threshold $r=r(n)\to\infty$ so that $n/r(n) \to \infty$. 

On the other hand, although we are dealing with extreme event prediction, the value of 
$\alpha$ is not necessarily extremely small or large so the quantile regression inference 
methodology  is conventional.
\end{remark}

\subsection{Examples and simulations}\label{sec:numerical-examples}

     In this section, we illustrate the performance of our optimal homogeneous predictors that
     were {\em estimated} from data in comparison with the oracle predictors in several simulated models. 
     We implemented the peaks-over-threshold methodology described in Section \ref{sec:setup}.
     While the quantile $(\theta,\alpha)\mapsto q_\alpha(\theta)$ involved therein can be estimated 
     using a variety of  methods, we opted for the versatile nonparametric method of \cite{meinshausen2006quantile} 
     known as {\em quantile regression forest}. This method works remarkably well in multiple dimensions and is easy 
     to tune. Specifically, we implemented our estimator in R \citep[][]{optXpred:2026,optXpred_shiny:2026} using a 
     state-of-the-art computationally and memory efficient implementation of quantile regression forests 
     in the R package {\tt ranger} by \cite{wright:ziegler:2017}. For more details, see Section \ref{sec:algorithm}.

We present two types of models: spectrally discrete and spectrally continuous ones, that illustrate and the general case 
described in Section~\ref{sec:inference}. For more details and proofs, see Section~\ref{sec:explicit_examples}.\\

{\bf The spectrally discrete case.} Consider the linear factor model
\begin{equation}\label{e:linear-model-preview}
  Y = \sum_{i=1}^{p} b_i \xi_i \quad \text{and} \quad X = \sum_{i=1}^{r} a_i \xi_i,\ \ (r\le p),
\end{equation}
where $\xi_1,\ldots,\xi_p$ are independent standard $1$-Pareto random variables, viewed as {\em latent factors};
$b_i \ge 0$, and $a_i \in \mathbb{R}^d$.  Then, vector $Z=(Y,X)$ is regularly varying
with a \emph{discrete} angular measure concentrated on the $p$ directions
$c_i/\|c_i\|_1$, where $c_i=(b_i,a_i)$ (see Proposition~\ref{p:spec-discrete-lambda}
in Section~\ref{sec:spec-discr_mod}).

\begin{remark} The latent factor model in \eqref{e:linear-model-preview} includes as a special case
the linear model of the form $Y = X^\top \beta + \epsilon$, where $X$ and $\epsilon \sim {\rm Pareto}(1)$ 
are independent (with $X$ as in \eqref{e:linear-model-preview}).  The latent factor formulation, however, 
is in general much richer than the simple linear model. In fact, as $p\to\infty$, the angular measures of 
such spectrally discrete models can be shown to be dense in the space of all possible angular measures for 
jointly regularly varying $(Y,X)$.
\end{remark}

\begin{remark} The linear factor models in \eqref{e:linear-model-preview} are just one instance of models
with discrete spectra. Others such as max-linear or generalized Breiman-type models can have such 
spectra and Proposition \ref{p:spec-discrete-lambda} and Corollary \ref{c:spec-discrete} hold for them.
\end{remark}

Under mild non-proportionality assumptions on the $a_i$'s
(Corollary~\ref{c:spec-discrete}), the optimal asymptotic precision takes the
explicit form
\begin{equation}\label{e:lambda-opt-discrete-preview}
  \lambda^{(\mathrm{opt})}_{\mathcal{G}}(Y,X)
  = \frac{\sum_{i=1}^{r} b_i}{\sum_{i=1}^{p} b_i},
\end{equation}
where $r \le p$ is the number of factors with $a_i \ne 0$.

\begin{figure}[t!]
  \centering
  \includegraphics[width=0.48\textwidth]{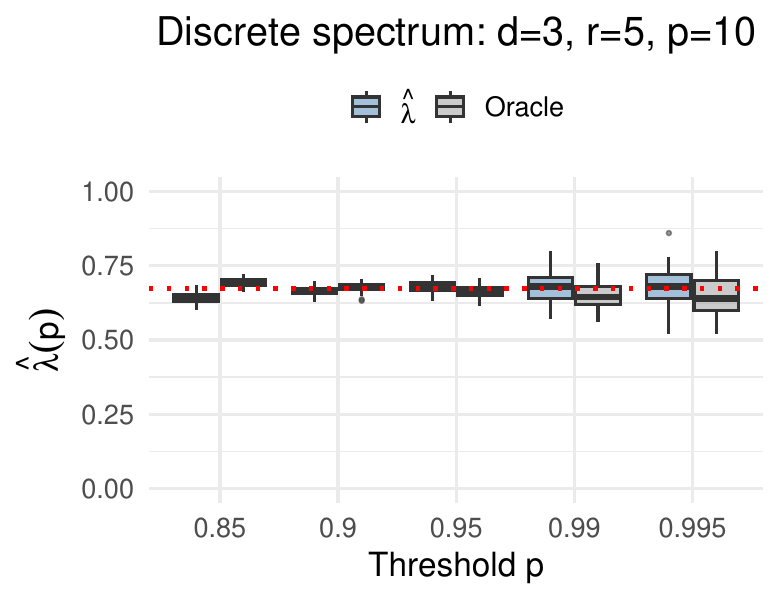}
  \hfill
  \includegraphics[width=0.48\textwidth]{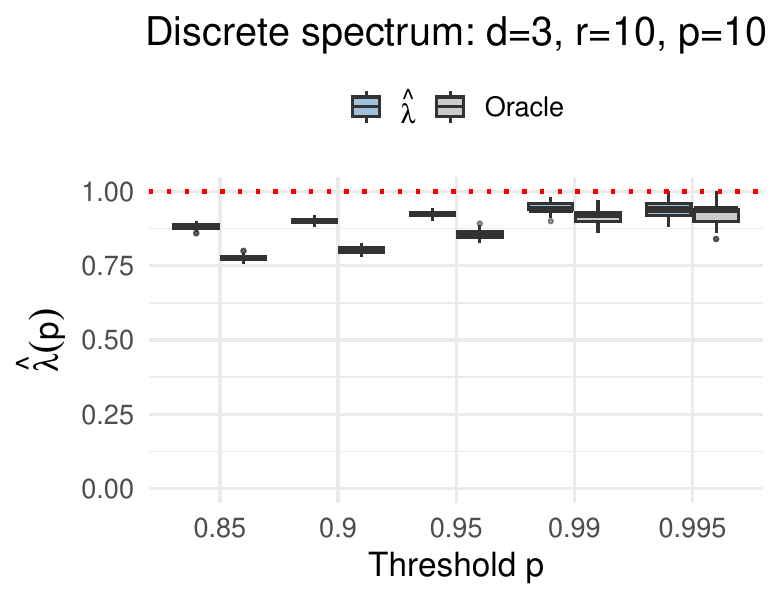}
  \caption{Pairs of boxplots of the empirical tail dependence coefficient $\wh{\lambda}(p)$ of the
  {\em estimated predictor} (left) and the oracle predictor (right) for several thresholds $p \in\{0.85, 0.9, 0.95, 0.99, 0.995\}$, based
  on $100$ independent replications of a spectrally discrete model. 
{\em Left panel:} Unobserved covariates present ($r=5<p=10$). {\em Right panel:} A perfect 
  asymptotic precision scenario ($r=p=10$). The dashed horizponline marks the optimal 
  asymptotic precision for the class of homogeneous predictors $\lambda_{\cal H}^{\mathrm{opt}}$.}
  \label{fig:boxplot_discrete}
\end{figure}

\begin{remark}[perfect asymptotic precision]
Relation \eqref{e:lambda-opt-discrete-preview} reveals a remarkable \emph{perfect asymptotic precision} 
phenomenon.  If all $p$ latent factors contribute to $X$ $(r=p)$ and no two directions 
$\{a_i/\|a_i\|_1,\ i=1,\cdots,p\}$ coincide, then \eqref{e:lambda-opt-discrete-preview} implies  $\lambda^{(\mathrm{opt})}_{\mathcal{G}}(Y,X) = 1.$ This means that, asymptotically, one can predict the 
extremes of $Y$ with {\em perfect precision}!  This surprising phenomenon can be intuitively explained 
via the \emph{single large jump heuristic} as follows. Conditionally on $\{\|X\|_1>t\}$, as 
$t\to\infty$, one and only one of the factors $\xi_i$ will dominate in the 
sum \eqref{e:linear-model-preview}. Thus, the direction $X/\|X\|_1$ identifies, asymptotically with probability 
one, the unique factor $\xi_i$ responsible for the extreme event, which in turn determines $Y$ exactly.  For more
details, see Section \ref{sec:spec-discr_mod}.
\end{remark}

Given the $b_i$'s and $a_i$'s, one can construct a continuous, non-negative, $1$-homogeneous function 
$h^{(\mathrm{opt})}$ achieving \eqref{e:lambda-opt-discrete-preview}, which we refer to as the {\em oracle predictor}
(see \eqref{e:h-opt-perfect}). In practice the $b_i$'s and $a_i$'s are unknown.  Nevertheless,
the spectral measure of $(Y,X)$ in \eqref{e:p:linear-model} can be estimated consistently
from data by using the peaks-over-threshold methodology combined with a non-parametric estimator of the quantiles
of $q_\alpha(\theta)$ the tilted conditional angular distribution (cf Algorithm \ref{algo:g_opt_hat}).

Figure~\ref{fig:boxplot_discrete} illustrates the finite-sample behavior of our nonparametric estimator 
across repeated experiments: The left panel corresponds to the under-determined case
$r < p$, where \eqref{e:lambda-opt-discrete-preview} is strictly less than one,
while the right panel shows the perfect precision case $r = p$, where 
$\lambda^{(\mathrm{opt})}_{\mathcal{G}} = 1$. The estimators are based on a training sample $\{(Y_i,X_i)\}$ 
of $n_{\rm train}= 10^4$ using a random threshold based on the $0.95$th empirical quantile. 
The empirical precision is then computed over an independently generated 
testing sample $\{(Y_i^*,X_i^*)\}_{i=1}^{n_{\rm test}}$ of size $n_{\rm test}=10^4$, where
\begin{equation}\label{e:wh-lambda(p)}
\wh \lambda(p):= \frac{1}{1-p} \sum_{i=1}^{n_{\rm test}} I\{ \wh g(X_i^*) > \wh F_{\wh g}(p) \} 
I\{ Y_i^* > \wh F_{Y}(p) \},\ \ p\in (0,1),
\end{equation}
where $\wh F_{\wh g}$ and $\wh F_Y$ are the empirical CDFs of $\{\wh g(X_i^*)\}$ and $\{Y_i^*\}$, 
respectively.  The boxplots are based on $100$ independent replications of the $\wh \lambda(p)$'s.
The coefficients $b_i$ and the components of the vectors $a_{i}$ in \eqref{e:linear-model-preview}
were fixed across different replications, but drawn independently and at random from the 
Uniform$(0,1)$ distribution, therefore ensuring that $a_i/\|a_i\| \not = a_j/\|a_j\|$, for all $
i\not = j$, with probability one.

Figure \ref{fig:boxplot_discrete} shows a close agreement of the estimated optimal 
homogeneous predictor and the oracle, and both are rather close to the optimal homogeneous extremal 
precision (dashed horizontal line).  It is remarkable that the empirical tail dependence coefficient 
is nearly one for thresholds $p$ as low as $0.85$ in the asymptotically perfect precision scenario.\\

{\bf The Pareto-Dirichlet case.} As an illustration of the spectrally \emph{continuous} setting, consider
$Z = (Y,X) = \xi\cdot(V,W)$, where $\xi$ is a standard $1$-Pareto random variable
independent of $(V,W) \sim \mathrm{Dirichlet}(\beta)$ with parameter
$\beta = (\beta_0,\beta_1,\ldots,\beta_d) \in (0,\infty)^{d+1}$.  We call this
the \emph{Pareto-Dirichlet}$(\beta)$ model (cf Proposition \ref{p:Pareto-Dirichlet}).  
It is $\tau$-regularly varying with respect to the radial function $\tau(y,x) = |y| + \|x\|_1$, 
and its angular measure is the Dirichlet distribution on the unit simplex $\Delta_d$.

The neutrality property of the Dirichlet distribution \citep[][]{ConnorMosimann1969} implies that
$U := V$ and $\Theta := W/\|W\|_1$ are independent, so the conditional
quantile $q_\alpha(\theta)$ in Theorem~\ref{thm:final} is constant in $\theta$.
This yields a remarkably simple optimal predictor
(Proposition~\ref{p:Pareto-Dirichlet} in Section~\ref{sec:explicit_examples}):
\begin{equation}\label{e:h-opt-Dirichlet-preview}
  h^{(\mathrm{opt})}(X) = c\,\|X\|_1,
  \qquad c := \frac{\beta_0}{\sum_{i=1}^{d}\beta_i},
\end{equation}
with optimal precision
\begin{equation}\label{e:lambda-opt-Dirichlet-preview}
  \lambda^{(\mathrm{opt})}_{\mathcal{G}}(Y,X)
  = \mathbb{E}\!\left[\frac{U}{\mu_U} \wedge \frac{1-U}{1-\mu_U}\right],
  \qquad U \sim \mathrm{Beta}\!\left(\beta_0,\,\textstyle\sum_{i=1}^d \beta_i\right),
  \quad \mu_U = \mathbb{E}[U].
\end{equation}
Moreover, $\lambda^{(\mathrm{opt})}_{\mathcal{G}}(Y,X) = \lambda_p(Y, h^{(\mathrm{opt})}(X))$
for all $p > p_0 := \mu_U \vee (1-\mu_U)$, so the precision is \emph{exactly constant} 
beyond a threshold.  This is clearly observed in the right panel of
Figure~\ref{fig:boxplot_spec_continuous}, where the empirical tail-dependence coefficient
$\hat\lambda_p$ boxplots for both the oracle and the nonparametric estimator plateau
at $\lambda^{(\mathrm{opt})}_{\mathcal{G}}$ across a wide range of $p$ values (see also 
Figure~\ref{fig:Spectrally Dirichlet model precision}).

\begin{figure}[t!]
  \centering
  \includegraphics[width=0.48\textwidth]{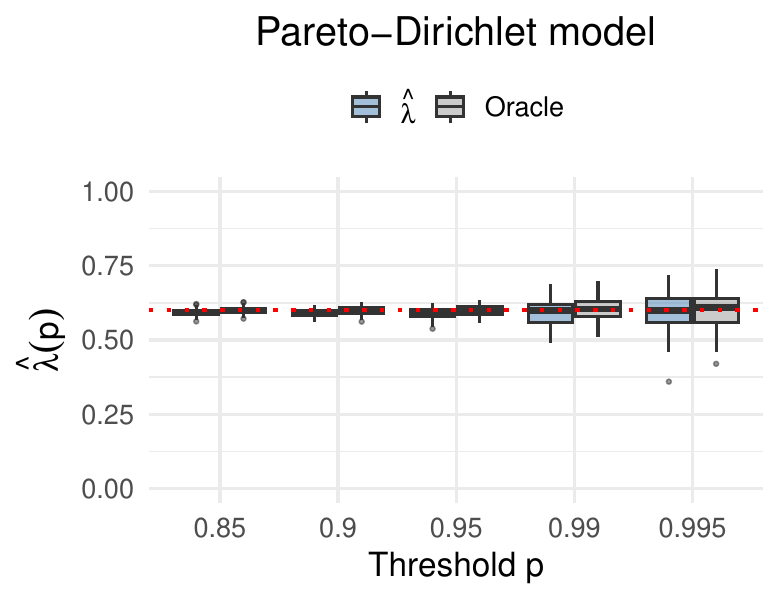}
  \hfill
  \includegraphics[width=0.48\textwidth]{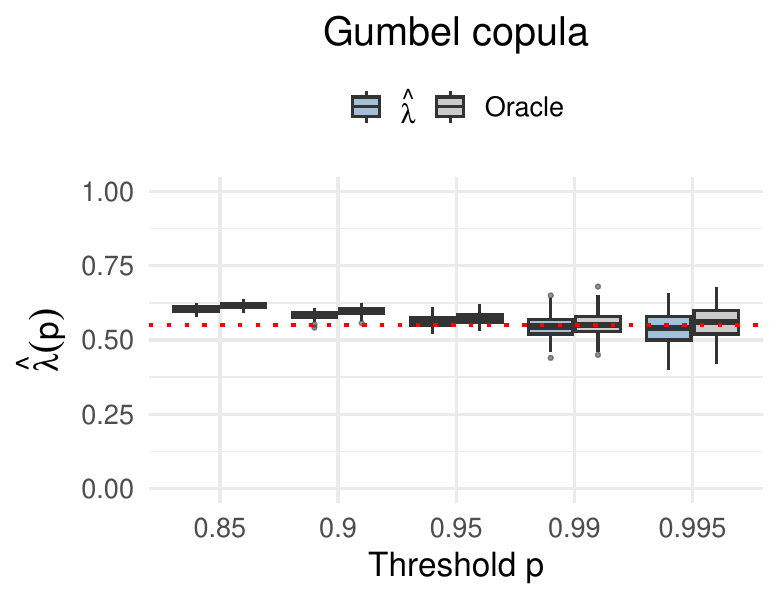}
  \caption{Pairs of boxplots of the empirical tail dependence coefficient 
  $\wh{\lambda}(p)$ of the {\em estimated predictor} (left) and the oracle predictor 
  (right) for several thresholds $p$, based on $100$ independent replications 
  for two spectrally continuous models. 
  {\em Left panel:} Pareto--Dirichlet spectral model ($d=9$).
  {\em Right panel:} Gumbel copula model ($\beta=1.5$, $d=10$).
   The horizontal dotted red line marks the theoretically optimal asymptotic precision
   $\lambda_{\cal H}^{\mathrm{opt}}$ in each case.}
  \label{fig:boxplot_spec_continuous}
\end{figure}

Figure \ref{fig:boxplot_spec_continuous} (left panel) shows boxplots of the empirical tail-dependence
coefficients as in \eqref{e:wh-lambda(p)} based on $100$ replications of an estimated optimal predictor
and the corresponding oracle. As in Figure \ref{fig:boxplot_discrete}, we have $n_{\rm train} = n_{\rm test} = 10^4$ 
and a random peaks-over-threshold of the $0.95$-th percentile.  The Pareto-Dirichlet model involves 
coefficients $\beta_0 = 1$ and $\beta_i = (i+1)/10,\ i=1,\cdots,9$.  As in the spectrally discrete case,
we have a remarkably close agreement in the precision of the estimated predictor and the oracle, which in this case is 
simply proportional to $\|X\|_1$. Both agree closely with the asymptotically optimal precision as indicated 
above.

The right panel of Figure \ref{fig:boxplot_spec_continuous} shows the same type of experiment where
$(Y,X)$ come from a Gumbel copula, where $d=10$.  By Corollary \ref{c:logistic}, in this case the optimal 
homogeneous predictor is in fact the ultimate optimal predictor for all values of $p\in (0,1)$.  In this scenario,
the non-parametric estimator is remarkably close to the oracle.

These brief simulation results illustrate the success of the optimal homogeneous prediction 
methodology. The quantile regression forest based implementation is nearly tuning--free and quite adaptive -- working 
well in both spectrally discrete and continuous settings.  We have implemented this methodology in the context of 
the challenging open problem of Solar flare prediction, where we found it to work out-of-the-box and provide 
state-of-the-art prediction scores for the most extreme X-class flares (cf Figure \ref{fig:solar-flare-prediction}).
We plan to pursue more applied and methodological directions of this research in a future work, but the preliminary 
findings are available in Section \ref{sec:solar-flares}. The code used to produce all simulation in the paper is freely 
available from the GitHub repository \cite{optXpred:2026}. An R Shiny app illustrating the methodology over a 
challenging Solar flare prediction problem is deployed on \cite{optXpred_shiny:2026}.


\section{Discussion and future work}
In this paper, we considered the fundamental problem of optimal extreme event prediction 
in terms of covariates. Starting from a general Neyman-Pearson-type characterization of 
the optimal predictors (Theorem \ref{thm:general-opt}), we formulated an optimality framework that is 
closely related to the notion of tail-dependence.  This naturally led us to explore the problem of 
optimal extremal prediction
in the context of jointly regularly varying response and covariates.  For the general case where the
predictors are positive homogeneous functions of the covariates, the problem was cast into a
variational problem -- a constrained optimization problem for a convex integral functional 
with respect to the angular measure.  Our first main contribution was an explicit solution to this 
problem (Theorem \ref{thm:final}), which can be represented in terms of the quantiles of a certain
tilted conditional distribution arising from the angular measure.  This abstract result let us to 
concrete and yet general inference methodology for the optimal homogeneous predictors. Our second
contribution was to establish the consistency for estimators of these optimal predictors in the framework
of peaks-over-threshold (Theorem \ref{thm:univ_consitency}).  This result, modulo mild regularity conditions,
may be viewed as an extremal counterpart to the celebrated Stone's Universal Consistency Theorems for 
the class of homogeneous predictors.  Last but not least, using established quantile regression forests 
tools, the estimators were implemented into a broadly applicable and nearly tuning free statistical 
software, which was shown to achieve the oracle performance over simulated examples.

The problem of extreme event prediction is fundamental and ubiquitous.  It can be seen as a series of
increasingly {\em imbalanced classification} problems, which is rather challenging. Notably, \cite{JSsS2018} 
have identified the shortcomings of the traditional empirical risk minimization approach and proposed a 
solution leading to a rapidly growing area of research \citep[see, e.g., Section 3 in][]{clemencon2026weak}.  
While motivated by similar problems, our approach is different since we focus on characterizing 
the optimal predictors {\em first} and {\em then} developing inference for them. By conditioning on the 
regime of extremes, the loss function is directly related to tail-dependence, which allowed us to take a 
different look through the lens of calculus of variations, and characterize the optimal homogeneous predictors.  

We provide a rather complete solution to the optimal extremal prediction with homogeneous functions in the
context of jointly regularly varying $Y$ and $X$.  While in many cases the  class of homogeneous predictors 
do contain the optimal extremal predictors (cf Example \ref{ex:linear-paper} and Corollary \ref{c:logistic}), 
this is not always the case.  Indeed, as shown in Example \ref{ex:suboptimality}, even under joint regular 
variation, the `extremes' of $X$ do not necessarily predict the extremes of $Y$.  To the  best of our knowledge, 
the problem of finding  {\em universally consistent predictors} in the extreme value sense, remains open both 
in terms of its theoretical formulation and solution.  

Our study was originally motivated by the scientific problem of extreme solar flare prediction 
\citep[][]{verma:stoev:chen:2024}.  In such contexts, it is natural to ask what is the ultimate optimal precision 
that {\em any estimator} can achieve in predicting a physical event using a certain set of covariates.  While far 
from a complete answer to this question, our study makes a step in this direction.  A natural follow-up step is to 
study the generalization error as well as statistical error \citep[e.g. as in][]{Legrand03072025} under perhaps 
second-order joint regular variation conditions on $Y$ and $X$.

\section*{Acknowledgments}
We thank Victor Verma for providing curated the GOES flux time series and SHARP parameter data for the Solar flare 
prediction experiment in Section \ref{sec:solar-flares}.  SS was partially supported by the 
NSF grant CNS/CSE-2319592 ``Collaborative Research: IMR: MM-1A: Scalable Statistical Methodology for 
Performance Monitoring, Anomaly Identification, and Mapping Network Accessibility from Active Measurements''.

\bibliographystyle{agsm}
\bibliography{bibliography}

@article{bobra2015sola,
	author = {M. G. Bobra and S. Couvidat},
	doi = {10.1088/0004-637X/798/2/135},
	journal = {The Astrophysical Journal},
	month = {1},
	number = {2},
	pages = {135},
	publisher = {The American Astronomical Society},
	title = {SOLAR FLARE PREDICTION USING {SDO}/{HMI} VECTOR MAGNETIC FIELD DATA WITH A MACHINE-LEARNING ALGORITHM},
	url = {https://dx.doi.org/10.1088/0004-637X/798/2/135},
	volume = {798},
	year = {2015}}

@article{bobra2014theh,
	author = {Bobra, M. G. and Sun, X. and Hoeksema, J. T. and Turmon, M. and Liu, Y. and Hayashi, K. and Barnes, G. and Leka, K. D.},
	date = {2014-09-01},
	doi = {10.1007/s11207-014-0529-3},
	id = {Bobra2014},
	isbn = {1573-093X},
	journal = {Solar Physics},
	number = {9},
	pages = {3549--3578},
	title = {The {H}elioseismic and {M}agnetic {I}mager ({HMI}) Vector Magnetic Field Pipeline: {SHARP}s --Space-Weather {HMI}
    Active Region Patches},
	url = {https://doi.org/10.1007/s11207-014-0529-3},
	volume = {289},
	year = {2014}}

@article{Bobra2021,
  author  = {Bobra, Monica G. and Wright, Paul J. and Turmon, Michael J.
             and Vertanen, John and Dissauer, Karin and Schrijver, Carolus J.
             and Cheung, Mark C. M. and Wheatland, Michael S.
             and Leka, K. D. and Barnes, Graham},
  title   = {{SMART}s and {SHARP}s: Two Solar Cycles of Active Region Data},
  journal = {The Astrophysical Journal Supplement Series},
  year    = {2021},
  volume  = {256},
  number  = {2},
  pages   = {26},
  doi     = {10.3847/1538-4365/ac1f1d},
}

@book{sugiyama2012dens,
	author = {Sugiyama, Masashi and Suzuki, Taiji and Kanamori, Takafumi},
	edition = {1},
	publisher = {Cambridge University Press},
	title = {Density {Ratio} {Estimation} in {Machine} {Learning}},
	year = {2012}}

@book{kulik2020heav,
	address = {New York, NY},
	author = {Kulik, Rafal and Soulier, Philippe},
	doi = {10.1007/978-1-0716-0737-4},
	isbn = {978-1-07-160735-0 978-1-07-160737-4},
	language = {en},
	publisher = {Springer},
	series = {Springer {Series} in {Operations} {Research} and {Financial} {Engineering}},
	title = {Heavy-{Tailed} {Time} {Series}},
	url = {http://link.springer.com/10.1007/978-1-0716-0737-4},
	urldate = {2023-08-10},
	year = {2020}}

@book{bingham1987regu, place={Cambridge}, series={Encyclopedia of Mathematics and its Applications}, title={Regular Variation}, DOI={10.1017/CBO9780511721434}, publisher={Cambridge University Press}, author={Bingham, N. H. and Goldie, C. M. and Teugels, J. L.}, year={1987}, collection={Encyclopedia of Mathematics and its Applications}}

@book{resnick2007heav,
  title={Heavy-Tail Phenomena: Probabilistic and Statistical Modeling},
  author={Resnick, S.I.},
  number={v. 10},
  isbn={9780387242729},
  lccn={2006934621},
  series={Heavy-tail phenomena: probabilistic and statistical modeling},
  url={https://books.google.com/books?id=p8uq2QFw9PUC},
  year={2007},
  publisher={Springer}
}

@article {mcneil:neslehova:2009,
    AUTHOR = {McNeil, Alexander J. and Ne\v{s}lehov\'{a}, Johanna},
     TITLE = {Multivariate {A}rchimedean copulas, {$d$}-monotone functions
              and {$l_1$}-norm symmetric distributions},
   JOURNAL = {Ann. Statist.},
  FJOURNAL = {The Annals of Statistics},
    VOLUME = {37},
      YEAR = {2009},
    NUMBER = {5B},
     PAGES = {3059--3097},
      ISSN = {0090-5364,2168-8966},
   MRCLASS = {62E10 (60E05 62H05 62H20)},
  MRNUMBER = {2541455},
MRREVIEWER = {Moshe\ Shaked},
       DOI = {10.1214/07-AOS556},
       URL = {https://doi.org/10.1214/07-AOS556},
}

@book{schilling2012bern,
	address = {Berlin, Boston},
	author = {Ren{\'e} L. Schilling and Renming Song and Zoran Vondracek},
	doi = {doi:10.1515/9783110269338},
	isbn = {9783110269338},
	lastchecked = {2024-04-19},
	publisher = {De Gruyter},
	title = {Bernstein Functions: Theory and Applications},
	url = {https://doi.org/10.1515/9783110269338},
	year = {2012}
}

@article{fougeres2009mode,
	author = {Foug{\`e}res, Anne-Laure and Nolan, John P. and Rootz{\'e}n, Holger},
	copyright = {{\copyright} 2008 Board of the Foundation of the Scandinavian Journal of Statistics},
	doi = {10.1111/j.1467-9469.2008.00613.x},
	issn = {1467-9469},
	journal = {Scandinavian Journal of Statistics},
	language = {en},
	note = {\_eprint: https://onlinelibrary.wiley.com/doi/pdf/10.1111/j.1467-9469.2008.00613.x},
	number = {1},
	pages = {42--59},
	title = {Models for {Dependent} {Extremes} {Using} {Stable} {Mixtures}},
	url = {https://onlinelibrary.wiley.com/doi/abs/10.1111/j.1467-9469.2008.00613.x},
	urldate = {2024-05-07},
	volume = {36},
	year = {2009}
}

@ARTICLE{scheffler:stoev:2017,
    AUTHOR = {Scheffler, Hans-Peter and Stoev, Stilian},
     TITLE = {Implicit extremes and implicit max-stable laws},
   JOURNAL = {Extremes},
  FJOURNAL = {Extremes. Statistical Theory and Applications in Science,
              Engineering and Economics},
    VOLUME = {20},
      YEAR = {2017},
    NUMBER = {2},
     PAGES = {265--299},
      ISSN = {1386-1999},
   MRCLASS = {60F05 (60G70 62G32)},
  MRNUMBER = {3638367},
MRREVIEWER = {Wieslaw Dziubdziela},
       URL = {https://doi.org/10.1007/s10687-016-0278-9},
}

@article {lindskog:resnick:roy:2014,
    AUTHOR = {Lindskog, Filip and Resnick, Sidney I. and Roy, Joyjit},
     TITLE = {Regularly varying measures on metric spaces: hidden regular
              variation and hidden jumps},
   JOURNAL = {Probab. Surv.},
  FJOURNAL = {Probability Surveys},
    VOLUME = {11},
      YEAR = {2014},
     PAGES = {270--314},
      ISSN = {1549-5787},
   MRCLASS = {60B05 (28A33 60G17 60G51 60G70)},
  MRNUMBER = {3271332},
       DOI = {10.1214/14-PS231},
       URL = {http://dx.doi.org/10.1214/14-PS231},
}

@book{resnick:2024,
	Address = {New York},
	Author = {Resnick, Sidney I.},
	Publisher = {Springer},
	Title = {The art of finding hidden risks},
	DOI = {\url{https://doi.org/10.1007/978-3-031-57599-0}},
	Note = {Hidden Regular Variation in the 21st Century},
	Year = 2024}

@misc{verma:stoev:chen:2024,
      title={On the optimal prediction of extreme events in heavy-tailed time series with applications to solar flare forecasting}, 
      author={Victor Verma and Stilian Stoev and Yang Chen},
      year={2024},
      eprint={2407.11887},
      archivePrefix={arXiv},
      primaryClass={math.ST},
      url={https://arxiv.org/abs/2407.11887}, 
}

@book{resnick:1987,
	Address = {New York},
	Author = {S. I. Resnick},
	Publisher = {Springer-Verlag},
	Title = {Extreme Values, Regular Variation and Point Processes},
	Year = {1987}}

@Article{wright:ziegler:2017,
    title = {{ranger}: A Fast Implementation of Random Forests for High
      Dimensional Data in {C++} and {R}},
    author = {Marvin N. Wright and Andreas Ziegler},
    journal = {Journal of Statistical Software},
    year = {2017},
    volume = {77},
    number = {1},
    pages = {1--17},
    doi = {10.18637/jss.v077.i01},
  }

@article{meinshausen2006quantile,
  title={Quantile Regression Forests},
  author={Meinshausen, Nicolai},
  journal={Journal of Machine Learning Research},
  volume={7},
  pages={983--999},
  year={2006},
  url={https://www.jmlr.org/papers/v7/meinshausen06a.html}
}

@Article{Dyszewski2020,
  author  = {Piotr Dyszewski and Thomas Mikosch},
  journal = {Ann. Appl. Probab.},
  title   = {Homogeneous mappings of regularly varying vectors},
  year    = {2020},
  issn    = {1050-5164},
  pages   = {2999--3026},
  volume  = {30},
  doi     = {10.1214/20-aap1579},
  url     = {\url{https://doi.org/10.1214/20-AAP1579}},
}

@article{jansen:neblung:stoev:2023,
	title = {Tail-dependence, exceedance sets, and metric embeddings},
	issn = {1572-915X},
	url = {https://doi.org/10.1007/s10687-023-00471-z},
	doi = {10.1007/s10687-023-00471-z},
	journal = {Extremes},
	author = {Janßen, Anja and Neblung, Sebastian and Stoev, Stilian},
	month = may,
	year = {2023},
}

@article {dombry:ribatet:2015,
    AUTHOR = {Dombry, Cl\'{e}ment and Ribatet, Mathieu},
     TITLE = {Functional regular variations, {P}areto processes and peaks
              over threshold},
   JOURNAL = {Stat. Interface},
  FJOURNAL = {Statistics and its Interface},
    VOLUME = {8},
      YEAR = {2015},
    NUMBER = {1},
     PAGES = {9--17},
      ISSN = {1938-7989},
   MRCLASS = {60G70 (62G05)},
  MRNUMBER = {3320385},
MRREVIEWER = {Wieslaw Dziubdziela},
       DOI = {10.4310/SII.2015.v8.n1.a2},
       URL = {https://doi.org/10.4310/SII.2015.v8.n1.a2},
}

@misc{NOAA_GOES_Xray,
  author       = {{NOAA Space Weather Prediction Center}},
  title        = {{GOES X-ray Flux Data}},
  year         = {2023},
  howpublished = {\url{https://www.swpc.noaa.gov/products/goes-x-ray-flux}},
  note         = {Accessed: 2025-06-13}
}

@misc{basrak:milincevic:molchanov:2025,
      title={Foundations of regular variation on topological spaces}, 
      author={Bojan Basrak and Nikolina Milinčević and Ilya Molchanov},
      year={2025},
      eprint={2503.00921},
      archivePrefix={arXiv},
      primaryClass={math.PR},
      url={https://arxiv.org/abs/2503.00921}, 
}

@article {sibuya:1960,
    AUTHOR = {Sibuya, Masaaki},
     TITLE = {Bivariate extreme statistics. {I}},
   JOURNAL = {Ann. Inst. Statist. Math. Tokyo},
  FJOURNAL = {Annals of the Institute of Statistical Mathematics},
    VOLUME = {11},
      YEAR = {1960},
     PAGES = {195--210},
      ISSN = {0020-3157,1572-9052},
   MRCLASS = {62.00},
  MRNUMBER = {115241},
MRREVIEWER = {W.\ Hoeffding},
       DOI = {10.1007/bf01682329},
       URL = {https://doi.org/10.1007/bf01682329},
}

@book {lecam:yang:2000,
    AUTHOR = {Le Cam, Lucien and Yang, Grace Lo},
     TITLE = {Asymptotics in Statistics},
    SERIES = {Springer Series in Statistics},
   EDITION = {Second},
      NOTE = {Some basic concepts},
 PUBLISHER = {Springer-Verlag, New York},
      YEAR = {2000},
     PAGES = {xiv+285},
      ISBN = {0-387-95036-2},
   MRCLASS = {62F12},
  MRNUMBER = {1784901},
       DOI = {10.1007/978-1-4612-1166-2},
       URL = {https://doi.org/10.1007/978-1-4612-1166-2},
}

@book {pollard:2002,
    AUTHOR = {Pollard, David},
     TITLE = {A user's guide to measure theoretic probability},
    SERIES = {Cambridge Series in Statistical and Probabilistic Mathematics},
    VOLUME = {8},
 PUBLISHER = {Cambridge University Press, Cambridge},
      YEAR = {2002},
     PAGES = {xiv+351},
      ISBN = {0-521-80242-3; 0-521-00289-3},
   MRCLASS = {60-01 (60Axx 60E05 60F15 60G42 60G44 60J65)},
  MRNUMBER = {1873379},
MRREVIEWER = {Beloslav\ Rie\v can},
}

@book{vaart:1998,
	Address = {Cambridge},
	Author = {van der Vaart, A. W.},
	Isbn = {0-521-49603-9; 0-521-78450-6},
	Mrclass = {62-02 (62E20 62F05 62F12 62G07 62G09 62G20)},
	Mrnumber = {MR1652247 (2000c:62003)},
	Mrreviewer = {Nancy Reid},
	Pages = {xvi+443},
	Publisher = {Cambridge University Press},
	Series = {Cambridge Series in Statistical and Probabilistic Mathematics},
	Title = {Asymptotic statistics},
	Volume = 3,
	Year = 1998}

@inproceedings{JSsS2018,
 author = {JALALZAI, Hamid and Cl\'{e}men\c{c}on, Stephan and Sabourin, Anne},
 booktitle = {Advances in Neural Information Processing Systems},
 editor = {S. Bengio and H. Wallach and H. Larochelle and K. Grauman and N. Cesa-Bianchi and R. Garnett},
 pages = {},
 publisher = {Curran Associates, Inc.},
 title = {On Binary Classification in Extreme Regions},
 url = {\url{https://proceedings.neurips.cc/paper_files/paper/2018/file/0ebcc77dc72360d0eb8e9504c78d38bd-Paper.pdf}},
 volume = {31},
 year = {2018}
}

@article {hult:lindskog:2006,
    AUTHOR = {Hult, Henrik and Lindskog, Filip},
     TITLE = {Regular variation for measures on metric spaces},
   JOURNAL = {Publ. Inst. Math. (Beograd) (N.S.)},
  FJOURNAL = {Institut Math\'ematique. Publications. Nouvelle S\'erie},
    VOLUME = {80(94)},
      YEAR = {2006},
     PAGES = {121--140},
      ISSN = {0350-1302},
   MRCLASS = {28A33 (28C15)},
  MRNUMBER = {2281910 (2008g:28016)},
MRREVIEWER = {Paul Raynaud de Fitte},
       DOI = {10.2298/PIM0694121H},
       URL = {http://dx.doi.org/10.2298/PIM0694121H},
}

@unpublished{meinguet:segers:2012,
   author = {Meinguet, Thomas and Segers, Johan},
   title = {Regularly varying time series in {B}anach spaces},
   note = {Preprint available from arXiv:1001.3262},
   url = {http://arxiv.org/pdf/1001.3262v1},
   year = 2012}

@book{dudley:1989,
	Author = {R. M. Dudley},
	Publisher = {Wadsworth and Brook/Cole},
	Title = {Real Analysis and Probability},
	Year = {1989}}

@book {resnick:1999book,
    AUTHOR = {Resnick, Sidney I.},
     TITLE = {A probability path},
 PUBLISHER = {Birkh\"auser Boston Inc.},
   ADDRESS = {Boston, MA},
      YEAR = {1999},
     PAGES = {xii+453},
      ISBN = {0-8176-4055-X},
   MRCLASS = {60-01},
  MRNUMBER = {1664717 (2000b:60002)},
MRREVIEWER = {Lajos Horv{\'a}th},
}

@article{EMD2022qf,
  title={Random forest estimation of conditional distribution functions and conditional quantiles},
  author={Elie-Dit-Cosaque, Kevin and Maume-Deschamps, V{\'e}ronique},
  journal={Electronic Journal of Statistics},
  volume={16},
  number={2},
  pages={6553--6583},
  year={2022},
  publisher={The Institute of Mathematical Statistics and the Bernoulli Society}
}

@article{Portnoy_Koenker_1997,
author = {Stephen Portnoy and Roger Koenker},
title = {{The Gaussian hare and the Laplacian tortoise: computability of squared-error versus absolute-error estimators}},
volume = {12},
journal = {Statistical Science},
number = {4},
publisher = {Institute of Mathematical Statistics},
pages = {279 -- 300},
year = {1997},
doi = {10.1214/ss/1030037960},
URL = {https://doi.org/10.1214/ss/1030037960}
}

@article {davis:mikosch:2008,
    AUTHOR = {Davis, Richard A. and Mikosch, Thomas},
     TITLE = {Extreme value theory for space-time processes with
              heavy-tailed distributions},
   JOURNAL = {Stochastic Process. Appl.},
  FJOURNAL = {Stochastic Processes and their Applications},
    VOLUME = {118},
      YEAR = {2008},
    NUMBER = {4},
     PAGES = {560--584},
      ISSN = {0304-4149},
     CODEN = {STOPB7},
   MRCLASS = {62G32 (62M10 62M30)},
  MRNUMBER = {2394763 (2009b:62112)},
MRREVIEWER = {El Sayed M. Nigm},
       DOI = {10.1016/j.spa.2007.06.001},
       URL = {http://dx.doi.org/10.1016/j.spa.2007.06.001},
}

@article{basrak:davis:mikosch:2002,
	Author = {Basrak, Bojan and Davis, Richard A. and Mikosch, Thomas},
	Doi = {10.1214/aoap/1031863174},
	Fjournal = {The Annals of Applied Probability},
	Issn = {1050-5164},
	Journal = {Ann. Appl. Probab.},
	Mrclass = {60E05 (60G10 60G55 60G70 62M10)},
	Mrnumber = {MR1925445 (2003h:60022)},
	Mrreviewer = {M. Sreehari},
	Number = 3,
	Pages = {908--920},
	Title = {A characterization of multivariate regular variation},
	Url = {http://dx.doi.org/10.1214/aoap/1031863174},
	Volume = 12,
	Year = 2002}

@misc{optXpred:2026,
  author       = {Stoev, Stilian},
  title        = {opt{X}pred: R code for optimal extreme event prediction via homogoneous functions},
  year         = {2026},
  publisher    = {GitHub},
  version      = {v1.0.0},
  url          = {https://github.com/cctoeb/optXpred.git},
  note         = {Private.  Ask the author for access. Provides R code and a Shiny app for 
  optimal extreme event prediction via homogoneous functions}
}

@misc{optXpred_shiny:2026,
  author       = {Stoev, Stilian},
  title        = {An {R} {S}hiny app illustrating the opt{X}pred software for optimal extreme 
  event prediction via homogeneous functions},
  year         = {2026},
  publisher    = {},
  version      = {v1.0.0},
  url          = {https://rada.stat.lsa.umich.edu/shiny/sstoev/optXpred/}
}

@article{BoldiDavison2007,
  author = {Boldi, Marc-Olivier and Davison, Anthony C.},
  title = {A Mixture Model for Multivariate Extremes},
  journal = {Journal of the Royal Statistical Society: Series B (Statistical Methodology)},
  volume = {69},
  number = {2},
  pages = {217--229},
  year = {2007},
  doi = {10.1111/j.1467-9868.2007.00585.x}
}

@article{SabourinNaveau2014,
  author = {Sabourin, Anne and Naveau, Philippe},
  title = {Bayesian {D}irichlet Mixture Model for Multivariate Extremes: A Re-parametrization},
  journal = {Computational Statistics \& Data Analysis},
  volume = {71},
  pages = {542--567},
  year = {2014},
  doi = {10.1016/j.csda.2013.04.021}
}

@article{CorradiniStrokorb2024,
  author = {Corradini, Michela and Strokorb, Kirstin},
  title = {Stochastic Ordering in Multivariate Extremes},
  journal = {Extremes},
  volume = {27},
  pages = {357--396},
  year = {2024},
  doi = {10.1007/s10687-024-00486-0}
}

@book{devroye1996probabilistic,
  title     = {A Probabilistic Theory of Pattern Recognition},
  author    = {Devroye, Luc and Gy{\"o}rfi, L{\'a}szl{\'o} and Lugosi, G{\'a}bor},
  year      = {1996},
  publisher = {Springer},
  address   = {New York},
  series    = {Stochastic Modelling and Applied Probability},
  volume    = {31},
  isbn      = {978-0-387-94618-4}
}

@article{bobbia:dombry:varron:2025,
author = {Bobbia, B. and Dombry, C. and Varron, D.},
title = {A Donsker and Glivenko-Cantelli theorem for random measures linked to extreme value theory},
journal = {Scandinavian Journal of Statistics},
volume = {52},
number = {4},
pages = {1708-1734},
doi = {https://doi.org/10.1111/sjos.70007},
url = {https://onlinelibrary.wiley.com/doi/abs/10.1111/sjos.70007},
eprint = {https://onlinelibrary.wiley.com/doi/pdf/10.1111/sjos.70007},
year = {2025}
}

@article{ConnorMosimann1969,
  author  = {Connor, Robert J. and Mosimann, James E.},
  title   = {Concepts of Independence for Proportions with a Generalization
             of the {D}irichlet Distribution},
  journal = {Journal of the American Statistical Association},
  year    = {1969},
  volume  = {64},
  number  = {325},
  pages   = {194--206}
}

@book{samorodnitsky:taqqu:1994book,
	Address = {New York, London},
	Author = {G. Samorodnitsky and M. S. Taqqu},
	Publisher = {Chapman and Hall},
	Title = {{\it {S}table {N}on-{G}aussian {P}rocesses: {S}tochastic {M}odels with {I}nfinite {V}ariance}},
	Year = {1994}}

@book {beirlant:goegebeur:teugels:segers:2004,
    AUTHOR = {Beirlant, Jan and Goegebeur, Yuri and Teugels, Jozef and
              Segers, Johan},
     TITLE = {Statistics of extremes},
    SERIES = {Wiley Series in Probability and Statistics},
      NOTE = {Theory and applications,
              With contributions from Daniel De Waal and Chris Ferro},
 PUBLISHER = {John Wiley \& Sons, Ltd., Chichester},
      YEAR = {2004},
     PAGES = {xiv+490},
      ISBN = {0-471-97647-4},
   MRCLASS = {62-01 (60G70 62G32)},
  MRNUMBER = {2108013 (2005j:62002)},
MRREVIEWER = {L{\'a}szl{\'o} Viharos},
       DOI = {10.1002/0470012382},
       URL = {http://dx.doi.org/10.1002/0470012382},
}

@book{joe2015dependence,
  author    = {Joe, Harry},
  title     = {Dependence Modeling with Copulas},
  series    = {Chapman \& Hall/CRC Monographs on Statistics \& Applied Probability},
  publisher = {CRC Press},
  year      = {2015},
  address   = {Boca Raton, FL},
  pages     = {480},
  isbn      = {978-1-4665-8322-1},
}

@article{yu1998local,
  author    = {Yu, Keming and Jones, M. C.},
  title     = {Local Linear Quantile Regression},
  journal   = {Journal of the American Statistical Association},
  year      = {1998},
  volume    = {93},
  number    = {441},
  pages     = {228--237},
  doi       = {10.1080/01621459.1998.10474104}
}

@book{fan1996local,
  author    = {Fan, Jianqing and Gijbels, Ir{\`e}ne},
  title     = {Local Polynomial Modelling and Its Applications},
  publisher = {Chapman \& Hall},
  address   = {London},
  year      = {1996},
  series    = {Monographs on Statistics and Applied Probability},
  volume    = {66}
}

@book{Tsybakov2009,
      author        = "Tsybakov, Alexandre B",
      title         = "{Introduction to Nonparametric Estimation}",
      publisher     = "Springer",
      address       = "Dordrecht",
      series        = "Springer series in statistics",
      year          = "2009",
      url           = "https://cds.cern.ch/record/1315296",
      doi           = "10.1007/b13794",
}

@article{rigollet2011neyman,
  author  = {Rigollet, Philippe and Tong, Xin},
  title   = {Neyman--{P}earson Classification, Convexity and
             Stochastic Constraints},
  journal = {Journal of Machine Learning Research},
  year    = {2011},
  volume  = {12},
  pages   = {2831--2855},
  url     = {https://jmlr.org/papers/v12/rigollet11a.html}
}

@article{valavi2022predictive,
  author  = {Valavi, Roozbeh and Guillera-Arroita, Gurutzeta and
             Lahoz-Monfort, Jos\'{e} J. and Elith, Jane},
  title   = {Predictive Performance of Presence-Only Species
             Distribution Models: A Benchmark Study with
             Reproducible Code},
  journal = {Ecological Monographs},
  year    = {2022},
  volume  = {92},
  number  = {1},
  pages   = {e01486},
  doi     = {10.1002/ecm.1486}
}

@inproceedings{irvin2019chexpert,
  author    = {Irvin, Jeremy and Rajpurkar, Pranav and Ko, Michael and
               Yu, Yifan and Ciurea-Ilcus, Silviana and Chute, Chris and
               Marklund, Henrik and Haghgoo, Behzad and Ball, Robyn and
               Shpanskaya, Katie and Lungren, Matthew P. and Ng, Andrew Y.},
  title     = {{CheXpert}: A Large Chest Radiograph Dataset with Uncertainty
               Labels and Expert Comparison},
  booktitle = {Proceedings of the AAAI Conference on Artificial Intelligence},
  year      = {2019},
  volume    = {33},
  pages     = {590--597},
  doi       = {10.1609/aaai.v33i01.3301590}
}

@article{tong2013plugin,
  author  = {Tong, Xin},
  title   = {A Plug-in Approach to {N}eyman--{P}earson Classification},
  journal = {Journal of Machine Learning Research},
  year    = {2013},
  volume  = {14},
  number  = {92},
  pages   = {3011--3040},
  url     = {https://jmlr.org/papers/v14/tong13a.html}
}

@article{tong2020parametrics,
  author  = {Tong, Xin and Xia, Lucy and Wang, Jiacheng
             and Feng, Yang},
  title   = {Neyman--{P}earson Classification: Parametrics
             and Sample Size Requirement},
  journal = {Journal of Machine Learning Research},
  year    = {2020},
  volume  = {21},
  number  = {12},
  pages   = {1--48},
  url     = {https://jmlr.org/papers/v21/18-577.html}
}

@article{nguyen2010estimating,
  author  = {Nguyen, XuanLong and Wainwright, Martin J.
             and Jordan, Michael I.},
  title   = {Estimating Divergence Functionals and the Likelihood
             Ratio by Convex Risk Minimization},
  journal = {IEEE Transactions on Information Theory},
  year    = {2010},
  volume  = {56},
  number  = {11},
  pages   = {5847--5861},
  doi     = {10.1109/TIT.2010.2068870}
}

@article{clemencon2025regression,
  author  = {Cl\'{e}men\c{c}on, St\'{e}phan and Huet, Nathan
             and Sabourin, Anne},
  title   = {On Regression in Extreme Regions},
  journal = {Electronic Journal of Statistics},
  year    = {2025},
  volume  = {19},
  number  = {2},
  pages   = {4784--4828},
  doi     = {10.1214/25-EJS2441}
}

@article{clemencon2023concentration,
  author  = {Cl\'{e}men\c{c}on, St\'{e}phan and Jalalzai, Hamid
             and Lhaut, St\'{e}phane and Sabourin, Anne
             and Segers, Johan},
  title   = {Concentration Bounds for the Empirical Angular Measure
             with Statistical Learning Applications},
  journal = {Bernoulli},
  year    = {2023},
  volume  = {29},
  number  = {4},
  pages   = {2797--2827},
  doi     = {10.3150/22-BEJ1562}
}

@inproceedings{aghbalou2024sharp,
  author    = {Aghbalou, Anass and Portier, Fran\c{c}ois
               and Sabourin, Anne},
  title     = {Sharp Error Bounds for Imbalanced Classification:
               How Many Examples in the Minority Class?},
  booktitle = {Proceedings of The 27th International Conference
               on Artificial Intelligence and Statistics},
  year      = {2024},
  volume    = {238},
  series    = {Proceedings of Machine Learning Research},
  pages     = {838--846},
  publisher = {PMLR},
  url       = {https://proceedings.mlr.press/v238/aghbalou24a.html}
}

@article{clemencon2026weak,
  author  = {Cl\'{e}men\c{c}on, St\'{e}phan and Sabourin, Anne},
  title   = {Weak Signals and Heavy Tails: Learning Theory
             Meets Extreme Value Analysis},
  journal = {Extremes},
  year    = {2026},
  doi     = {10.1007/s10687-025-00519-8}
}

@article{Legrand03072025,
author = {Juliette Legrand and Philippe Naveau and Marco Oesting},
title = {Evaluation of Binary Classifiers for Asymptotically Dependent and Independent Extremes},
journal = {Journal of the American Statistical Association},
volume = {120},
number = {551},
pages = {1558--1568},
year = {2025},
publisher = {Taylor \& Francis},
doi = {10.1080/01621459.2025.2529024},
URL = {https://doi.org/10.1080/01621459.2025.2529024},
eprint = {https://doi.org/10.1080/01621459.2025.2529024}
}

@misc{leroux2026oodquantile,
  title         = {Out-of-Distribution Generalization of Quantile Regression
                   with Heavy Tailed Inputs: An {SVM} Approach},
  author        = {Leroux, Baptiste and Dombry, Cl{\'e}ment and Sabourin, Anne},
  year          = {2026},
  eprint        = {2606.00265},
  archivePrefix = {arXiv},
  primaryClass  = {stat.ML},
  url           = {https://arxiv.org/abs/2606.00265},
  note          = {48 pages, 5 figures}
}

@misc{decarvalho2025cascading,
  title         = {A {Kolmogorov}--{Arnold} Neural Model for Cascading Extremes},
  author        = {de Carvalho, Miguel and Ferrer, Clemente and Vallejos, Ronny},
  year          = {2025},
  eprint        = {2505.13370},
  archivePrefix = {arXiv},
  primaryClass  = {stat.ME},
  url           = {https://arxiv.org/abs/2505.13370},
}

\appendix

\section{Examples}\label{sec:examples}

\subsection{Examples of optimal predictors} \label{sec:examples-strong-optimality}

In this section, we elaborate on the closed-form optimal predictors discussed in Section \ref{sec:Neyman-Pearson}.
Specifically, we prove the stated characterizations of the optimal predictors in the bi-variate max-stable 
and Archimedean copula models with completely monotone generators.

 \subsubsection{Bivariate extreme value distributions}
 \label{sec:bivariate-max-stable}
 Suppose that $(Y,X)$ follow the a bivariate extreme value distribution 
 $G$, where without loss of generality the marginal distributions of $X$ and $Y$ are 
 standard unit Fr\'echet. That is,
 $$
  F_X(u) =  \P[X\le u] =  F_Y(u) = \P[Y\le u] = \exp\{-1/u\},\ u>0.
 $$
 It is known that \cite[see, e.g., Ch.\ 5 in][]{resnick:1987}, for all $x,y\ge 0,$
  \begin{equation}\label{e:G-EVD}
  G(x,y) = e^{-\mu(A_{x,y})},
 \end{equation}
 where $A_{x,y} = [0,\infty)^2 \setminus ([0,x]\times[0,y])$ and $\mu$ is a Borel measure on $[0,\infty)^2\setminus \{(0,0)\}$, such
 that $\mu(A_{x,y})<\infty$ provided $x>0$ or $y>0$.  It can be shown that $\mu(t A) = t^{-1} \mu(A)$, for all $t>0$ and 
 Borel sets $A\subset  [0,\infty)^2\setminus \{(0,0)\}$ and hence for all $x,y \ge 0$
 \begin{equation}\label{e:G-EVD-spectral}
  G(x,y) = \exp\Big\{ - V(1/x,1/y) \Big\},\ \ \mbox{ where }V(a,b):= \int_0^1 \max\{ w a, (1-w) b)\} \sigma(dw)
 \end{equation}
 where $\sigma$ is a finite measure on $[0,1]$ and by convention 
 $1/\infty = 0$ and $1/0 = \infty$.   The function $V$ is referred to as the 
 stable tail-dependence function of the bivariate extreme value distribution.

 The following result is rather natural and it shows that the events $\{X>x_0\}$ are optimal predictors of $\{Y>y_0\}$.
 
 \begin{proposition}\label{p:max-stable} Let $(X,Y)$ be a bivariate max-stable random vector with joint distribution function $G$ as in \eqref{e:G-EVD}.
 Suppose that $G(x,y)$ is continuously differentiable in $x$ and $y$, so that $g(x,y):= \partial^2_{x,y}G(x,y)$ is the joint density of $(X,Y)$ 
 with respect to the Lebesgue measure. Then, for all $y_0>0$, the predictor $\{X> F_X^{-1}(q)\}$ is an optimal predictor of $\{Y>y_0\}$, calibrated 
 at level $q$.
 \end{proposition}
 \begin{proof} 
  In view of Theorem \ref{thm:general-opt}, an optimal predictor of $\{Y>y_0\}$ is of the form $\{r(X)>\tau\}$, where 
  $r(x) = f_0(x)/f_1(x)$, where $f_0$ and $f_1$ are the conditional densities of $X|\{Y>y_0\}$ and $X|\{Y\le y_0\}$, respectively. 
  Let $p:= \P[ Y>y_0]$ and observe that for all $x>0$:
 $$
  f_0(x) = \frac{1}{1-p} \int_{y_0}^\infty g(x,y) dy = \frac{1}{1-p} \Big( \partial_x G(x,\infty) - \partial_x G(x,y_0) \Big)
$$
Similarly,
$$
f_1(x) = \frac{1}{p} \int_{0}^{y_0} g(x,y) dy = \frac{1}{p} \Big( \partial_x G(x,y_0) - \partial_x G(x,0) \Big) = \frac{1}{p}  \partial_x G(x,y_0),
 $$
 where we used that $\partial_x G(x,0)  = 0$.

 This shows that 
 $$
 r(x) = \frac{f_0(x)}{f_1(x)} = \frac{p}{(1-p)}\cdot \frac{\partial_x G(x,\infty)}{ \partial_x G(x,y_0)} - \frac{p}{1-p}.
 $$
 We will show that the function 
 \begin{equation}\label{e:p:max-stable-0}
  x\mapsto \frac{\partial_x G(x,\infty)}{ \partial_x G(x,y_0)}
 \end{equation}
 is increasing in $x$.  This would complete the proof since it will show that the events $\{r(X)>\tau\}$ and $\{ X> \tau'\}$ are the same, 
 when calibrated at the same level $q$.

 To this end, since $G(x,y) = \exp\{ -V(1/x,1/y)\}$, by the chain rule, we obtain
 $$
 \partial_x G(x,y) = x^{-2} \partial_a V(1/x,1/y) e^{-V(1/x,1/y)}. 
 $$
 Note, however, $\partial_a V(a,0) = 1$ since $V(a,0) = a$.  Thus, $\partial_x G(x,\infty) = x^{-2} \exp\{ -V(1/x,0)\}$ and
 \begin{align}\label{e:p:max-stable-1}
     \frac{\partial_x G(x,\infty)}{ \partial_x G(x,y_0)} = 
     \frac{\exp\{ V(1/x,1/y_0) - V(1/x,0)\}}{\partial_a V(1/x,1/y_0)} = \frac{\exp\{ \mu(A_{x,y_0}) - \mu(A_{x,\infty}) \}} {\partial_a V(1/x,1/y_0)},
 \end{align}
 where we used the fact that $V(1/x,1/y) = \mu(A_{x,y})$ with $A_{x,y} = [0,\infty)^2 \setminus ([0,x]\times [0,y]),\ x,y\in (0,\infty)$,
 and $A_{x,\infty} = [0,\infty)^2 \setminus ([0,x]\times [0,\infty])$.  

Now, for the numerator in the right-hand side of \eqref{e:p:max-stable-1}, we obtain
$$
\exp\{ \mu(A_{x,y}) - \mu(A_{x,\infty}) \} = \exp\{ \mu( [0,x]\times (y_0,\infty) ) \},
$$
which is a monotone increasing function of $x$ since the sets $[0,x] \times (y_0,\infty)$ grow as $x$ grows.  On the other hand, $a\mapsto \partial_a V(a,b)$ is a 
monotone non-decreasing function of $a$.  Indeed, in view of \eqref{e:G-EVD-spectral}, we have that $a\mapsto V(a,b)$ is a convex function, and therefore its partial 
derivative is non-decreasing in $a$.  This implies that the denominator $x \mapsto \partial_a V(1/x,1/y_0)$ in the right-hand side of \eqref{e:p:max-stable-1} is a 
monotone non-increasing function of $x.$  This, combined with the established monotonicity of the numerator implies that the function 
in \eqref{e:p:max-stable-0} is 
monotone non-decreasing in $x$, which  completes the proof of the proposition.
\end{proof}

\begin{remark} Let $(Y,X)$ be as in Proposition \ref{p:max-stable}.  This result
implies the optimal extremal precision $\lambda^{\rm(opt)}(Y,X)$ equals the 
tail-dependence coefficient $\lambda(Y,X)$ (see also Example 
\ref{ex:bivariate-max-stable-contd}).
\end{remark}

\subsubsection{Archimedean copula: Proofs for Section \ref{sec:archimedean-statements}}
\label{sec:archimedean}

We start with the following auxiliary result.

 \begin{lemma}\label{l:Bernstein-log-convex} The positive completely monotone 
 functions are log-convex.
 \end{lemma}
 \begin{proof}
    Let $\psi$ be as in \eqref{e:Bernstein}. Recall that the function $\psi:(0,\infty)\to (0,\infty)$ is log-convex if $\log(\psi(\cdot))$ is convex, or, equivalently:
    \begin{align}\label{e:l:Bernstein-log-convex-0}
    \psi(\lambda u + (1-\lambda)v) \le \psi(u)^\lambda \psi(v)^{1-\lambda},
    \end{align}
    for all $u,v >0$ and $\lambda \in (0,1)$. In view of \eqref{e:Bernstein}, for all $\lambda \in (0,1)$,
    \begin{align}\label{e:l:Bernstein-log-convex}
    \psi(\lambda u + (1-\lambda)v)  &= \int_0^\infty e^{- \lambda ux } e^{-(1-\lambda)vx} \mu(dx)\nonumber\\
    & \le \Big(\int_0^\infty e^{- ux }\mu(dx)\Big)^{1/p} \Big(\int_0^\infty e^{- vx }\mu(dx)\Big)^{1/q}, 
    \end{align}
    where the last bound follows from the H\"older inequality applied to the functions
    $f(x):=e^{-\lambda ux}\in L^p(\mu(dx))$ and $g(x):= e^{-(1-\lambda)vx} \in L^q(\mu(dx))$, with  $p:=1/\lambda$ and $q:=1/(1-\lambda)$.  This completes the proof since the right-hand sides of \eqref{e:l:Bernstein-log-convex-0} and \eqref{e:l:Bernstein-log-convex} are the same.
    \end{proof}

 \begin{proposition}[Statement of Proposition \ref{p:archimedean}]\label{p:archimedean-statement} 
 Let $(Y,X_1,\cdots,X_d)$ be a random vector in $\R^{1+d}$ with continuous marginal cdfs $F_Y$ and $F_{X_i},\ i=1,\cdots,d$, respectively,
 and the Archimedean copula $C$ as in \eqref{e:archimedean}, where $\psi$ is a completely monotone
 function as in \eqref{e:Bernstein}.  Then, for all $\tau>0$, the events
 $$
  \Big\{ \psi \Big(  \psi^{-1}(F_{X_1}(X_1))+\cdots + \psi^{-1}(F_{X_d}(X_d))  \Big) > \tau \Big\}
 $$
 are optimal predictors of the events $\{Y>y_0\}$ in terms of $X$.  
 \end{proposition}
 \begin{proof}[Proof of Proposition \ref{p:archimedean}.]
 Without loss of generality, assume that the marginals of $(Y,X)$ are uniform, i.e., the joint cdf of $(Y,X)$ is 
 the copula $C$ in \eqref{e:archimedean}.  Let for convenience $\varphi(t):=\psi^{-1}(t)$ and
 observe that the joint density of $(X,Y)$ with respect to the Lebesgue measure is:
 $$
  f_{Y,X}(y,x) = \partial_{y,x}^{d+1}C(y,x) 
  = \psi^{(d+1)}\Big(\varphi(y)+\varphi(x_1)+\cdots+\varphi(x_d) \Big)  \varphi'(y)
  \prod_{i=1}^d \varphi'(x_i),
 $$
 where $x =(x_i)_{i=1}^d \in (0,1)^d,\ y\in (0,1)$.
 Therefore, with $p:= \P[Y >y_0] \in (0,1)$, for the conditional densities $f_0$ and $f_1$
 of $X|\{Y>y_0\}$ and $X|\{Y\le y_0\}$ in Theorem \ref{thm:general-opt}, we obtain
 \begin{align}\label{p:archimedean-1}
  f_0(x) & = \frac{\prod_{i=1}^d \varphi'(x_i)}{(1-p)} \int_{y_0}^1 \psi^{(d+1)}(\varphi(x_1)+\cdots+\varphi(x_d) + \varphi(y))  \varphi'(y) dy\nonumber \\
  & = \frac{\prod_{i=1}^d \varphi'(x_i)}{(1-p)} \int_{y_0}^1 h'( c + \varphi(y)) \varphi'(y) dy,
 \end{align}
 where for clarity, we let $h(t):= \psi^{(d)}(t)$ and $c:= \varphi(x_1)+\cdots+\varphi(x_d)$.  The fundamental
 theorem of calculus applied to the integral in the right-hand side of \eqref{p:archimedean-1} 
 then implies
 \begin{align} \label{e:p:archimedean-f0}
  f_0(x) & = \frac{\prod_{i=1}^d \varphi'(x_i)}{(1-p)} \Big( h(c) - h(c+\varphi(y_0))\Big),
\end{align}
where we used that $\varphi(1) = 0$.
Similarly, for $f_1$, we obtain
\begin{align}\label{e:p:archimedean-f1}
 f_1(x) & = \frac{\prod_{i=1}^d \varphi'(x_i)}{p} \int_{0}^{y_0} h'( c + \varphi(y)) \varphi'(y) dy \nonumber\\
 & =  \frac{\prod_{i=1}^d \varphi'(x_i)}{p}  \Big(h(c+\varphi(y_0)) - h(\infty)\Big),
\end{align}
where we used that $\varphi(0) = \infty.$ In view of \eqref{e:Bernstein}, we have
$$
 h(u) = \psi^{(d)}(u) = (-1)^d \int_0^\infty e^{-ux} x^d \mu(dx)
$$
and hence $h(\infty) = \lim_{u\to\infty} h(u) = 0$.  Thus, for the ratio $r(x) = f_0(x)/f_1(x)$, in view of
\eqref{e:p:archimedean-f0} and \eqref{e:p:archimedean-f1}, we obtain:
\begin{equation}\label{e:p:archimedean-r}
r(x) = \frac{f_0(x)}{f_1(x)} \propto \frac{\psi^{(d)}( c )} { \psi^{(d)}( c  + \varphi(y_0))} -1,
\end{equation}
where recall that $c = c(x) := \psi^{-1}(x_1) +\cdots+\psi^{-1}(x_d)$.  In view of 
Theorem \ref{thm:general-opt}, the events $\{r(X)>\tau\}$ are optimal predictors of $\{Y>y_0\}$. 
Thus, in view of \eqref{e:p:archimedean-r}, to complete the proof it is enough to establish that 
the function
$$
 c := \frac{\psi^{(d)}( c )} { \psi^{(d)}( \varphi(y_0)+c )} 
$$
is decreasing in $c$.  We will show that this is the case by using the log-convexity of 
the positive completely monotone functions (Lemma \ref{l:Bernstein-log-convex}).  

Indeed, observe that 
$$
 g(c)= \frac{\psi_d ( c )} { \psi_d(  \varphi(y_0) + c  )},
$$
where 
$
\psi_d(u) := (-1)^d \psi^{(d)}(u) = \int_0^\infty e^{-ux} x^d \mu(dx).
$
Notice that the function $\psi_d$ is also completely monotone and hence it is log-convex. 

We have that
\begin{equation}\label{e:p:archimedean-1}
\frac{d}{dc}\log( g (c)) = 
\frac{\psi_d'(c)}{\psi_d(c)} - \frac{\psi_d'( \varphi(y_0)+c)}{\psi_d( \varphi(y_0)+c)}.
\end{equation}
Since $\psi_d$ is log-convex, we have that 
$$
c\mapsto \frac{\psi_d'(c)}{\psi_d(c)} = \frac{d}{dc}\log(\psi_d(c))
$$
is an increasing function.  Since $\varphi(y_0)\ge 0$, this implies that the right-hand 
side of \eqref{e:p:archimedean-1} is non-positive, $d/dc ( \log(g(c)) \le 0$ and this establishes 
that the function $c\mapsto g(c)$ is decreasing and the proof is complete.
\end{proof}

Proposition \ref{p:archimedean} implies the following result, which is of independent interest since it concerns an important 
class of multivariate max-stable models.

\begin{corollary}\label{c:logistic} Assume that $(Y,X)$ follow the multivariate 
logistic max-stable distribution in \eqref{e:logistic-old}, with parameter $\beta>1$.
Then, for all $p,q\in (0,1)$, an optimal 
predictor of $\{Y>F_Y^{-1}(p)\}$ is of the form:
$$
\{h(X) > F_{h(X)}^{-1}(q)\},\ \ \mbox{ where }\ h(X) =\Big(\sum_{i=1}^d \frac{1}{X_i^\beta} \Big)^{-1/\beta}.
$$

Moreover, the optimal extremal precision is given by:
\begin{equation}\label{e:c:logistic}
\lambda^{\rm (opt)}(Y,X) = \lambda(Y,h(X)) = \E\Big[ \min\Big\{ \Gamma_1^{-1/\beta}/c_{1,\beta},\ \Gamma_d^{-1/\beta}/c_{d,\beta} \Big\} \Big], 
\end{equation}
where $\Gamma_1\sim {\rm Gamma}(1,1)$ and $\Gamma_d \sim {\rm Gamma}(d,1)$ are 
independent Gamma-distributed random variables and where $c_{d,\beta} = 
\Gamma(d-1/\beta)/\Gamma(d)$.
\end{corollary}
\begin{proof}[Proof of Corollary \ref{c:logistic}]
Let $\beta >1$ and $Z$ a $1/\beta$-stable subordinator.  That is,
a positive random variable with the Laplace transform
$$
\psi_Z(t) = \E[ e^{-tZ} ] = e^{-t^{1/\beta}},\ t\ge 0.
$$
Let also $\xi_i,\ i=1,\cdots,d$ and $\eta$ be independent standard $\beta-$Fr\'echet random 
variables, i.e., such that
$$
\P[\xi_i\le x] = \P[\eta \le x] = e^{-x^{-\beta}},\ x>0.
$$
Assume that $Z$ and $\xi = (\xi_i)_{i=1}^d$ and $\eta$ are independent.
Then, the multivariate logistic max-stable random vector $(X,Y)$ in \eqref{e:logistic-old} 
has the following stochastic representation:
$$
(X,Y) = Z^{1/\beta}\cdot( \xi,\eta),
$$
where $\xi = (\xi_i)_{i=1}^d$ (see e.g., \cite{fougeres2009mode}).

Proposition \ref{p:archimedean} shows that the optimal predictor of 
$\{Y>y_0\}$ is of the form
$\{\|1/X\|_\beta^{-1}>\tau\}$, where $\|1/X\|_\beta:= (\sum_{i=1}^d X_i^{-\beta})^{1/\beta}$.
We will derive next the optimal extremal precision 
$$
\lambda^{\rm (opt)}(Y,X) = \lambda (Y, 1/\|1/X\|_\beta).
$$
We have that $1/\eta^\beta$ and $1/\xi_i^\beta =:E_i \sim {\rm Exponential}(1)$ are standard Exponentially distributed.  Therefore,
$$
1/\|1/X\|_\beta = Z^{1/\beta} \Big(\sum_{i=1}^d E_i\Big)^{-1/\beta} =: Z^{1/\beta} \frac{1}{\Gamma_d^{1/\beta}},
$$
where $\Gamma_d = E_1+\cdots+E_d\sim{\rm Gamma}(d,1)$. Similarly,
$$
Y = Z^{1/\beta} \eta =: Z^{1/\beta} \frac{1}{\Gamma_1^{1/\beta}},
$$
where $\Gamma_1:= \eta^{-\beta}\sim {\rm Gamma}(1,1)$.

Now, we have that
$$
\P[Z^{1/\beta}>x] \sim c_\beta^{-1} \frac{1}{x}, \ \ \mbox{ as }x\to\infty,
$$
where in fact $c_\beta = \Gamma(1-1/\beta)$ 
\citep[see e.g.\ Proposition 1.2.12 and Property 1.2.15 in][]{samorodnitsky:taqqu:1994book}.  Therefore, by Breiman's Lemma, we have
$$
\P[Y>x] = \P[Z^{1/\beta} \eta > x] \sim c_\beta^{-1} \frac{\E[\eta] }{x} = \frac{1}{x},\ \ x\to\infty,
$$
since $\E[\eta] = \int_0^\infty x^{1} d e^{-x^{-\beta}}=\Gamma(1-1/\beta) = c_\beta$.
By Breiman's lemma, we also have
$$
\P[1/\|1/X\|_\beta > x] = \P[ Z^{1/\beta} \Gamma_d^{-1/\beta}>x] \sim c_\beta^{-1} \E[\Gamma_d^{-1/\beta}] \frac{1}{x}
= \frac{\Gamma(d-1/\beta)}{\Gamma(d)c_\beta} \frac{1}{x} =: \sigma_{d,\beta} \frac{1}{x}.
$$
Thus, the random variable $\|1/X\|_\beta^{-1}/\sigma_{d,\beta}$ has unit asymptotic scale, i.e., it is asymptotically calibrated:
$$
\P [\|1/X\|_\beta^{-1}/\sigma_{d,\beta} > x ] \sim \P[Y> x] \sim \frac{1}{x},\ \ x\to\infty.
$$
This, implies that
$$
\P[ \|1/X\|_\beta^{-1}/\sigma_{d,\beta} \wedge Y > x] \sim  \lambda(Y, \|1/X\|_\beta^{-1}) \frac{1}{x} 
= c_\beta^{-1}
\E[ \Gamma_d^{-1/\beta}/\sigma_{d,\beta} \wedge \Gamma_1^{-1/\beta}] \frac{1}{x},
$$
by another application of Breiman's lemma.  Therefore, observing that $c_\beta = c_{1,\beta} = \Gamma(1-1/\beta)$ and 
$c_\beta \sigma_{d,\beta} = \Gamma(d-1/\beta)/\Gamma(d),\ d=1,2,\ldots$, we obtain 
\eqref{e:c:logistic}.
\end{proof}


\section{Proofs}

\subsection{Multivariate regular variation}\label{sec:RV}

  In this section, we provide a brief review of the theory of multivariate regular variation 
   of Borel measures on $\R^k,\ k\ge 1$. We consider the general case of $\tau$-regular variation, where 
   \begin{equation}\label{e:tau}
   \tau:\R^k \to \R_+
   \end{equation}
   is a {\em fixed} non-negative, continuous $1$-homogeneous function 
   \citep[see e.g.][]{scheffler:stoev:2017,dombry:ribatet:2015}. For an abstract and general 
   theory of regular variation with respect to the notion of a {\em bornology}, see e.g.,
   \cite{basrak:milincevic:molchanov:2025} and the references therein.

\subsection{The space $\M_\tau$ and the Tail Index Theorem}

Let $\M_\tau = \{ \mu\, :\, \mu(\tau >\epsilon)<\infty,\ \mbox{ for all }\epsilon>0\}$ 
be the collection of Borel measures $\mu$ on $\R^k\setminus\{\tau=0\}$ such that $\mu(\{\tau>\epsilon\})<\infty$, 
for all $\epsilon>0$.  Such measures assign finite masses to all 
Borel sets $A$ that are bounded away from $0$, relative to the $\tau$-loss, i.e.,
$\tau(A):=\inf_{x\in A}\tau(x)$ is positive. We shall denote this class of Borel sets by
${\cal B}_0(\tau)$ and ${\cal B}(\tau) = \sigma({\cal B}_0(\tau))$ is the $\sigma$-field of all Borel sets in 
$\R^k\setminus\{\tau=0\}$.

If $\mu_n,\mu\in \M_\tau$, then we write $\mu_n\stackrel{M_0}{\to} \mu$, as $n\to\infty$, if
\begin{equation}\label{e:M0-convergence-via-f}
\int f(x)\mu_n(dx) \to \int f(x)\mu(dx),\ \ \mbox{ as }n\to\infty,
\end{equation}
for all bounded and continuous functions $f$ such that $\{|f|>0\} \subset \{\tau>\epsilon\}$ for some
$\epsilon>0$. It can be shown that $\mu_n\stackrel{M_0}{\to} \mu$ if and only if
$$
 \mu_n(A)\to \mu(A),\ \ \mbox{ as }n\to\infty,
$$
for all $A\in {\cal B}_0(\tau)$ such that $\mu(\partial A) =0$, where 
$\partial A:=\overline{A}\setminus A^\circ$ denotes the topological boundary of the set $A$
\citep[see e.g.\ Theorem 3.1 in][]{hult:lindskog:2006}.  The above-defined
$M_0$-convergence of measures in $\M_\tau$ is generated by a metric, which renders $\M_\tau$ a complete
separable metric space \citep[see e.g. Theorem 2.3 in][]{hult:lindskog:2006}.

\begin{definition}\label{def:RV-mu} Let $\tau$ be a fixed continuous $1$-homogeneous function as in \eqref{e:tau}. 
Consider a random vector $Z$ in $\R^k,\ k\ge 1$.  If there is a positive 
sequence $\{a_n\}$ and a non-zero measure $\mu \in \M_\tau$ such that
$$
\mu_n:= n \P[Z/a_n\in \cdot] \stackrel{M_0}{\to} \mu,\ \ \mbox{ as }n\to\infty,
$$
then we say that $Z$ is $\tau$-regularly varying and write $Z\in RV(\{a_n\},\mu)$.
\end{definition}

\begin{theorem}[The tail index theorem] \label{thm:tail-index}
If $Z\in RV(\{a_n\},\mu)$ according to Definition \ref{def:RV-mu}, then:

\begin{enumerate}
    \item There exists a positive $\alpha>0$ 
    such that $a_n \sim \ell(n) n^{1/\alpha}$, for some slowly varying function $\ell(\cdot)$.

    \item The measure $\mu$ satisfies the scaling relation:
    \begin{equation}\label{e:thm:tail-index:scaling}
     \mu(t\cdot A) = t^{-\alpha} \mu(A),\ \ \mbox{ for all $t>0$ and $A\in {\cal B}(\tau)$}.
    \end{equation}
    \item If also $Z\in RV(\{b_n\},\nu)$, then 
    $$
    \frac{a_n}{b_n} \to c \in (0,\infty),\ \ \mbox{ and }\ \ \mu(A) = c^{-\alpha} \nu(A),\ \mbox{ for all $A\in {\cal B}(\tau)$}.
    $$
    Consequently, the parameter $\alpha>0$, referred to as the tail exponent of $Z$ is unique,
    and so is the measure $\mu$, up to a rescaling factor.

    \item We have that $Z\in {\rm RV}_\alpha(\R^k,\{a_n\},\ c_Z, \tau, \sigma)$ 
    according to Definition \ref{d:tau-rv}, where
    $$
     c_Z:= \mu(\{\tau>1\}) \ \ \mbox{ and } \ \ 
     \sigma(B) := \frac{1}{\mu(\{\tau>1\})}  \mu(\{ x/\tau(x) \in B\}),\ \ B\in {\cal B}(S_\tau).
    $$
    \item Conversely, if $Z\in {\rm RV}_\alpha(\R^k,\{a_n\},\ c_Z, \tau, \sigma)$ 
    according to Definition \ref{d:tau-rv}, then 
    $Z\in RV_\alpha(\{a_n\},\mu)$ according to Definition \ref{def:RV-mu}, where 
    \begin{equation}\label{e:thm:tail-index:disintegration}
     \mu\circ T_\tau^{-1}( (r,\infty)\times B) ) = c_Z r^{-\alpha} \sigma(B),\ \ \mbox{ for all } r>0,\ B\in {\cal B}(S_\tau),
    \end{equation}
    and where $T_\tau : \R^k\setminus\{\tau=0\} \to (0,\infty)\times S_\tau$ is the generalized polar coordinated
    homeomorphism defined as $T_\tau(x) = (\tau(x), x/\tau(x))$.
\end{enumerate}
\end{theorem}

    The proof of this result can be derived from many excellent treatments in the literature.  Specifically, 
    see, for example, Theorem 3.1 and Corollary 4.4 in \cite{lindskog:resnick:roy:2014}, 
    as well as \cite{hult:lindskog:2006,meinguet:segers:2012}, the monographs \cite{kulik2020heav,resnick:2024}, the 
    general theory in \cite{basrak:milincevic:molchanov:2025}, and the references therein for more details.

\subsection{Proofs for Section \ref{sec:opt_ext_pred}}

\subsubsection{Proofs for Section \ref{sec:joint-RV}}
  \label{sec:proofs:sec:joint-RV}

 For completeness, we will present the proof of Proposition \ref{p:rv-via-h} below although
 similar results can be found in Proposition 2.5 of \cite{jansen:neblung:stoev:2023} and
 Theorem 2.1 of \cite{Dyszewski2020}. The proof is a consequence of the following characterization of multivariate regular variation \citep[see e.g. Theorem 2.2.1 in][]{kulik2020heav}.
 
\begin{proposition} \label{p:rv-via-Pareto-Theta}
We have $Z \in RV_\alpha(\R^k,\{a_n\},c_Z, \tau,\sigma)$ if
and only if 
\begin{equation}\label{e:p:rv-via-Pareto-Theta}
\Big( \frac{\tau(Z)}{t}, \frac{Z}{\tau(Z)} \Big) \Big\vert \tau(Z)>t 
\stackrel{d}{\longrightarrow} (V, \Theta), \ \ (t\to\infty),
\end{equation}
where $V$ and $\Theta$ are independent and $\P[V>x]= x^{-\alpha},\ x\ge 1$, while
$\Theta\sim \sigma$.
\end{proposition}

\begin{proposition}[Statement of Proposition \ref{p:rv-via-h}] Let $Z \in RV_\alpha(\R^k,\{a_n\},b(\cdot),c_Z,\tau,\sigma)$. Let 
also ${\cal H}_+(\R^k,\tau)$ denote the class of all non-negative
continuous $1$-homogeneous functions $h:\R^d\to \R_+$  such that $\{h>0\} \subset \{\tau>0\}$.  Then,
for all $h\in {\cal H}_+(\R^k,\tau)$, such that
\begin{equation}\label{e:p:rv-via-h-assumption-A} 
\|h\|_{\infty,S_\tau} := \sup_{\theta \in S_\tau} |h(\theta)| <\infty,
\end{equation}
we have
\begin{equation}\label{e:p:rv-via-h-A}
\lim_{n\to\infty} n\P[ h(Z) > a_n ] = \lim_{t\to\infty} b(t)\P[ h(Z) > t ] = c_Z \sigma(h),
\end{equation}
 where
$$
\sigma(h) := \E[h(\Theta_Z)^\alpha] = \int_{S_\tau} h^\alpha(\theta) \sigma(d\theta).
$$
\end{proposition}
\begin{proof}[Proof of Proposition \ref{p:rv-via-h}] 
Without loss of generality suppose that the positive normalizing sequence $\{a_n\}$ is such that $c_Z=1$.
Observe that by the support domination condition $\{h>0\} \subset \{\tau>0\}$, 
we have $\{h(Z)>a_n\}$ implies $\{\tau(Z)>0\}$. Also, note that since by \eqref{e:p:rv-via-h-assumption-A},
$h$ is bounded on the set $S_\tau$, we have 
$$
\tau(Z) h(Z/\tau(Z)) >a_n \ \ \mbox{ implies }\ \ \tau(Z)>a_n/\|h\|_{\infty,S_\tau},
$$ 
where $\|h\|_{\infty,S_\tau}$ is as in \eqref{e:p:rv-via-h-assumption-A}.

Note that in the case $\|h\|_{\infty,S_\tau}=0$, the result trivially holds.  On the other hand 
for $c_h:= \|h\|_{\infty,S_\tau}>0$, we have
$$
\P [ h(Z)>a_n] = \P [\tau(Z) h(Z/\tau(Z)) > a_n, \tau(Z) > a_n/c_h].
$$
Hence, using that $n\P[\tau(Z)> t a_n]\to t^{-\alpha}$
\begin{align}\label{e:e:p:rv-via-Pareto-Theta-1}
n \P [ h(Z)>a_n] &= n \P[\tau(Z)>a_n/c_h] \cdot 
\P\Big[ \frac{\tau(Z)}{a_n/c_h}\cdot  h\Big( \frac{Z}{\tau(Z)} \Big) > c_h 
\Big\vert \tau(Z)>a_n/c_h \Big]\nonumber\\
& \longrightarrow c_h^\alpha \cdot \P[ V h(\Theta)  > c_h ],\ \ (a_n\to\infty)
\end{align}
where the last convergence follows from \eqref{e:p:rv-via-Pareto-Theta} and 
the fact that $V h(\Theta)$ is a continuous random variable.

It remains to observe that by the independence of $V$ and $\Theta_Z:=\Theta$, we have
$$
\P\Big[ V h(\Theta_Z)  > c_h \Big] = \E[ I(V> c_h/h(\Theta_Z)] 
=\E\Big[ \Big(c_h/h(\Theta_Z)\Big)^{-\alpha}\Big] = c_h^{-\alpha} \E[ h(\Theta_Z)^\alpha].
$$
By combining this calculation with Relation \eqref{e:e:p:rv-via-Pareto-Theta-1}, we obtain that the first limit in 
\eqref{e:p:rv-via-h} equals $c_Z\E[h^\alpha(\Theta_Z)]$.  The fact that the two limits in \eqref{e:p:rv-via-h} 
coincide follows from the equivalence $b(a_n)\sim n$ as $n\to\infty$, which completes the proof.
\end{proof}

\begin{lemma}[Statement of Lemma \ref{lem:nonameidea}] \label{l:copy:lem:nonameidea} Let $(Y,X)$ and ${\cal G}(\tau_X)$ be as in Definition \ref{def:asymp-calibration}.
Then, for all $g\in {\cal G}(\tau_X)$ the bivariate tail dependence coefficient 
$\lambda(Y,g(X))$ defined in \eqref{e:lambda-tdep-def} exists and
\begin{equation}\label{e:lem:nonameidea-restatement}
\lambda(g(X),Y) = \lim_{t\uparrow \infty} \P[ Y > t |  g(X) > t ] = \lim_{t\uparrow \infty} \P[ g(X) > t | Y > t ].
\end{equation}
\end{lemma}
\begin{proof}[Proof of Lemma \ref{lem:nonameidea}]
By the {\em transfer of regular variation theorem}  \citep[see e.g. Proposition 2.1.12 in][]{kulik2020heav}, 
we have $(\xi,\eta) := (g(X),Y)$ is regularly varying with exponent measure $\nu := \mu\circ T_g^{-1}$, where $T_g(y,x):= (g(x),y)$.  
More precisely,
$$
n \P[ (\xi,\eta) \in a_n A] \to \nu(A),\ \ \mbox{ as }n\to\infty,
$$
for all Borel $\nu$-continuity sets.  Equivalently \citep[see, e.g. Theorem 2.1.3 in ][]{kulik2020heav}, 
\begin{equation}\label{e:lem:nonameidea-RV-via-b}
b(t) \P[ (\xi,\eta) \in t\cdot A] \to \nu(A),\ \ \mbox{ as }t\to\infty,
\end{equation}
for a suitable function $b(\cdot)$, which is regularly varying at infinity with exponent $1$.

The calibration property implies that 
\begin{equation}\label{e:tail-equivalence-xi-eta}
 \P[\xi>t]\sim \P[\eta>t],\ \ \mbox{ as }t\to\infty.
\end{equation}
Thus, by Lemma \ref{l:xi-eta-tdep}, since $\xi$ and $\eta$ have regularly varying right tails, to prove the result, it is enough to
show that 
\begin{equation}\label{e:lem:nonameidea-lambda-via-nu}
\lim_{t\to\infty} \P[\xi>t | \eta>t] = \lim_{t\to\infty} \P[\eta>t | \xi>t] = \frac{\nu(A_1\cap A_2)}{\nu(A_1)},
\end{equation}
where $A_i= \{ (u_1,u_2)\in \R^2\, :\, u_i\ge 1\}$.

To establish \eqref{e:lem:nonameidea-lambda-via-nu}, observe that for all $t,s>0$, the sets 
$A_{(t,s)}:= [t,\infty)\times [s,\infty)$ are $\nu$-continuity sets.
Indeed, this follows from the homogeneity of the measure $\nu$. For all $t,s>0$ and $\epsilon\in (0,1)$, we have that 
$A_{\epsilon\cdot(t,s)} = \cup_{u \in [\epsilon,\infty)} \partial A_{u\cdot(t,s)}$, where the latter is a union of 
{\em pairwise disjoint} sets indexed by $u$.  This, since $\nu(A_{\epsilon\cdot(t,s)})<\infty$,  implies that all but countably many of the terms
$\nu(\partial A_{u\cdot(t,s)})$ vanish. Notice, however, that by the homogeneity of $\nu$, since $\partial A_{u\cdot(t,s)} = u\cdot \partial A_{(t,s)}$, we have $\nu(\partial A_{u\cdot (t,s)}) = u^{-1} \nu(\partial A_{(t,s)}),$ for all $u>0$.  This means that
$\nu(\partial A_{u\cdot (t,s)}) = 0$, for all $u\ge \epsilon$.  Similarly, the sets $A_i,\ i=1,2$ are continuity sets.

The continuity of the sets $A_1, A_2,$ and $A_1\cap A_2 =A_{(1,1)}$ and Relation \eqref{e:lem:nonameidea-RV-via-b}
entail that, as $t\to\infty$,
\begin{equation}\label{e:lem:nonameidea-11}
  b(t)\P[\xi>t] \to \nu(A_1),\ \  b(t)\P[\eta>t]\to \nu(A_2),\ \mbox{ and }\ b(t)\P[\xi>t,\eta>t] \to \nu(A_1\cap A_2).
\end{equation}
Relation \eqref{e:tail-equivalence-xi-eta} implies $\nu(A_1) = \nu(A_2)$, which are necessarily positive, since 
$\nu$ is non-zero. Thus, the convergences in \eqref{e:lem:nonameidea-11} imply
\eqref{e:lem:nonameidea-lambda-via-nu}, which completes the proof.
 \end{proof}

\begin{lemma}\label{l:xi-eta-tdep}  Let $\xi$ and $\eta$ be random variables with regularly varying right tails such that
\begin{equation}\label{e:l:xi-eta-tdep}
\P[ \xi >t ] = L_\xi(t) t^{-\alpha}\ \ \mbox{ and }\ \ \P[\eta>t] = L_\eta(t) t^{-\alpha},\ \ t\ge 0,
\end{equation}
where $L_\xi$ and $L_\eta$ are slowly varying functions at $\infty$.  Suppose that $\xi$ and $\eta$ are
tail-equivalent, i.e.
\begin{equation}\label{e:xi-eta-tail-equiv}
\P[\xi>t] \sim \P[\eta>t],\ \ \mbox{ as }t\to\infty.
\end{equation}
If the limit 
\begin{equation}\label{e:l:xi-eta-tdep-1} 
\lambda = \lim_{t\to\infty} \P[\xi>t | \eta> t].
\end{equation}
exists, then so does the tail-dependence coefficient $\lambda(\xi,\eta)$  and
$$
 \lambda = \lambda(\xi,\eta):= \lim_{p\uparrow 1} \P[ \xi> F_\xi^{\leftarrow}(p) | \eta > F_\eta^{\leftarrow}(p)].
$$
\end{lemma}
\begin{proof} It is enough to show that for every $p_n\uparrow 1$, we have 
\begin{equation}\label{e:l:xi-eta-tdep-2}
 \P[\xi > s_n | \eta> t_n] \to \lambda,\ \ \mbox{ as }n\to\infty,
\end{equation}
where $s_n:= F_\xi^\leftarrow(p_n)$ and $t_n:= F_\eta^\leftarrow(p_n)$. 

For the so-defined sequences, we will first show that 
\begin{equation}\label{e:lem:nonameidea-sn~tn}
s_n= F_\xi^\leftarrow(p_n) \sim t_n =  F_\eta^\leftarrow(p_n).
\end{equation}

This last fact is consequence of the classical theory on regularly varying functions.  Indeed, letting 
$G_\xi(t):= 1/(1-F_\xi(t))$, and similarly  $G_\eta(t) = 1/(1-F_\eta(t))$, with $u_n:=1/(1-p_n)$, we get
$$
 s_n = G_\xi^\leftarrow(u_n)\ \ \mbox{ and }\ \ t_n = G_\eta^\leftarrow(u_n).
$$ 
Since $G_\xi(t) = L_\xi^{-1} (t) t^\alpha$ and $G_\eta(t) = L_\eta(t)^{-1} t^{\alpha}$,
Relation \eqref{e:xi-eta-tail-equiv} entails $G_\xi(t) \sim G_\eta(t),\ t\to\infty$.  This equivalence, since $G_\xi$ and $G_\eta$ are regularly varying functions,
implies $G_\xi^\leftarrow(u) \sim G_\eta^\leftarrow(u)$ as $u\to\infty$ \citep[see, e.g.,  Proposition 1.5.15 in ][]{bingham1987regu}, which entails \eqref{e:lem:nonameidea-sn~tn}.

Now, for every $\epsilon\in (0,1)$, we have
$$
\frac{\P[\xi \wedge \eta > (1+\epsilon) t_n]}{\P[\eta>t_n]} \le \frac{\P[\xi> s_n, \eta> t_n]}{\P[\eta>t_n]} 
\le \frac{\P[\xi \wedge \eta > (1-\epsilon) t_n]}{\P[\eta>t_n]}.
$$
In view of \eqref{e:l:xi-eta-tdep-1}, taking a limit, for the lower/upper bounds above, we get
$$
\lim_{t_n\to\infty} \frac{\P[\xi \wedge \eta > (1\pm\epsilon) t_n]}{\P[\eta>t_n]} = \lambda \lim_{t_n\to\infty} \frac{\P[\eta>(1\pm \epsilon)t_n]}{\P[\eta>t_n]} =
\lambda \cdot (1\pm \epsilon)^{-\alpha},
$$
where the last relation follows from \eqref{e:l:xi-eta-tdep}.  This, since $\epsilon\in (0,1)$ can be taken arbitrarily small,
yields \eqref{e:l:xi-eta-tdep-2}.
\end{proof}

\begin{proposition}[Statement of Proposition \ref{p:lambda-via-U-Theta}]
\label{p:restatement:lambda-via-U-Theta} Let $(Y,X)$ be jointly $\tau$-regularly varying
in the sense of Assumption \ref{a:X-Y-joint}.  Let ${\cal G}(\tau_X)$ 
be the class of predictors in Definition \ref{def:asymp-calibration}. Then, for all 
$g\in {\cal G}(\tau_X)$, and $(U,\Theta)$ as in \eqref{e:U-THETA}, 
we have
\begin{equation}\label{e:constraint-via-U-Theta-repeated}
\E[U] = \E[(1-U)g(\Theta)] =c>0
\end{equation}
and
\begin{equation}\label{e:lambda-via-U-Theta-repeated}
\lambda(Y,g(X)) = \frac{1}{c} \E[ U \wedge (1-U) g(\Theta)]
\end{equation}
\end{proposition}

\begin{proof}[Proof of Proposition \ref{p:lambda-via-U-Theta}] 
Since $Z=(Y,X)$ is regularly varying, Proposition
\ref{p:rv-via-h} implies that, for all non-negative, continuous and $1$-homogeneous functions $h$,
such that $\{h>0\} \subset \{\tau >0\}$ with $\|h\|_{\infty,S_\tau}<\infty$, we have
\begin{equation}\label{e:h-YX-relation}
n \P[ h(Y,X) > a_n ] \to \sigma_Z(h) = \E[h(U,(1-U)\Theta)],\ \ \mbox{ as }n\to\infty.
\end{equation}

Let now $g\in {\cal G}(\tau_X)$ so that \eqref{e:g-calibration-lemma} holds and
consider the $1$-homogeneous functions:
$$ 
 h_Y(y,x):= y_+,\ \  h_g(y,x):= g(x)\ \mbox{ and }\ h_{Y,g} (y,x):= y_+ \wedge g(x).
$$
Observe that these three homogeneous functions are continuous, non-negative, bounded 
on the unit sphere $S_\tau$ and support-dominated by $\tau$. That is,
$\{ h_Y>0\} \cup \{ h_g >0\} \cup \{h_{Y,g}>0\} \subset \{\tau >0\}$ and hence Relation \eqref{e:h-YX-relation}
applies with $h$ replaced by $h_Y,\ h_g,$ and $h_{Y,g}$.  This implies that, as $n\to\infty$, we have
\begin{equation}\label{e:p:lambda-via-U-Theta-1}
n\P [ Y > a_n] \to \sigma_Z(h_Y)= \E[U],\ \ \  \ n\P[g(X)>a_n] \to \sigma_Z(h_g) = \E[ g((1-U)\Theta)],
\end{equation}
and 
\begin{equation}\label{e:p:lambda-via-U-Theta-2}
n \P[Y\wedge g(X) > a_n] \to \sigma_Z(h_{Y,g}) = \E[ U\wedge g((1-U)\Theta)].
\end{equation}
Since  $g(X)$ is an asymptotically calibrated extremal predictor for $Y$
(recall \eqref{e:g-calibration-lemma}), Relations \eqref{e:p:lambda-via-U-Theta-1} imply 
that $\E[U] = \E[(1-U)\Theta] = c$.  Observe that $c>0$ since by Assumption \ref{a:X-Y-joint},
$\sigma(\{0\}\times \{\tau_X=1\}) = \P[U=0] <1$.
This proves \eqref{e:constraint-via-U-Theta-repeated}.

On the other hand, by Lemma \ref{lem:nonameidea}, we have
$$
\lambda(Y,g(X)) = \lim_{n\to\infty} \P[g(X)>a_n | Y>a_n] = \lim_{n\to\infty}
\frac{\P [ Y\wedge g(X) > a_n] }{\P[Y>a_n]}  =  \frac{\E[ U \wedge g((1-U)\Theta)]}{\E[U]},  
$$
where the last relation follows from \eqref{e:p:lambda-via-U-Theta-1} and \eqref{e:p:lambda-via-U-Theta-2}.
This shows \eqref{e:lambda-via-U-Theta-repeated} and completes the proof of the proposition.
\end{proof}








\subsubsection{Proofs for Section \ref{sec:solution}}

\label{sec:proofs:main}
 \begin{lemma}[Statement of Lemma \ref{l:gateaux-general}] \label{l:gateaux-general-statement} Let $I(\cdot)$ be as in \eqref{e:I(g)}, the $b(\theta)$'s and $F_\theta(t)$'s be as in 
 \eqref{e:b(theta)} and \eqref{e:F-theta}, respectively.  Then, for all $g\ge 0$ and $g,\, h\in L^1(b\cdot p_\Theta)$, we have
 \begin{align}\label{e:lower-bound-via-F-statement}
&\lim_{\delta\downarrow 0} \frac{I(g+\delta h)-I(g)}{\delta} \nonumber\\
&\quad \quad= 2\int_{\SX}   b(\theta)  \Big\{ F_\theta(q(\theta)) h(\theta) +  F_\theta(\{q(\theta)\}) h(\theta)_-\Big\}  p_\Theta(d\theta) 
 - \int_\SX b(\theta) h(\theta) p_\Theta(d\theta),
 \end{align}
 where $F_\theta(\{t\}) = F_\theta(t) - F_\theta(t-)$ and
$$
 q(\theta):= \frac{g(\theta)}{1+g(\theta)}\ \ \  \mbox{ and }\ \ \ h(\theta)_- = \max\{0,-h(\theta)\}.
$$
 \end{lemma}
 \begin{proof} In view of \eqref{e:I(g)}, and the Fubini theorem for probability kernels \citep[see, e.g. Theorem 10.2.1 in ][]{dudley:1989}, 
 we have that 
 \begin{align}
    \frac{I(g+\delta h)-I(g)}{\delta} &= \int_{\SX} \Big\{\int_{u\in [0,1]}  \frac{|u - (u-1)(g(\theta) + \delta h(\theta))| - |u - (u-1)g(\theta)|}{\delta} 
     p(du|\theta) \Big\} p_{\Theta}(d\theta)\nonumber\\
     &=: \int_{\SX} \Big\{\int_{u\in [0,1]} f_\delta(u,\theta) p(du|\theta) \Big\} p_{\Theta}(d\theta), \label{e:l:gateaux-general-f-delta} 
 \end{align}
 Relation \eqref{e:lower-bound-via-F-statement} will follow from two consecutive applications of the Lebesgue Dominated Convergence Theorem (DCT).  Indeed, 
 let us first establish the point-wise limit of $f_\delta(u,\theta)$ as $\delta\downarrow 0$ by considering cases.\\

 {\em (Case 1):} Suppose that $u - (1-u)g(\theta)>0$.  Then, there exists a $\delta_{\theta,u,h}>0$ such that for all 
 $0<\delta<\delta_{\theta,u,h}$, we have $u - (1-u)(g(\theta)+\delta h(\theta))>0$ and hence
 $$
  f_\delta(u,\theta) = \frac{1}{\delta} \Big( u - (1-u)(g(\theta)+\delta h(\theta)) - (u-(1-u)g(\theta)) \Big) = -(1-u)h(\theta).
 $$

 {\em (Case 2:)} Similarly, if $u - (1-u)g(\theta)<0$, for possibly smaller $\delta_{\theta,u,h}>0$, we have $u - (1-u)(g(\theta)+\delta h(\theta))<0$,
 for all $0<\delta < \delta_{\theta,u,h}$ and hence
 $$
  f_\delta(u,\theta) = (1-u)h(\theta).
  $$

 {\em (Case 3:)} If $u - (1-u)g(\theta)=0$, we readily obtain
 $$
f_\delta(u,\theta) = \frac{1}{\delta} |(1-u)\delta h(\theta)| = (1-u) |h(\theta)|.
 $$

 By combining the above three cases, and observing that 
 $$\{u - (1-u)g(\theta)<0\}=\{u< q(\theta):= g(\theta)/(1+g(\theta))\},$$
 we get
 \begin{align}\label{e:l:gateaux-general-3} 
 & \lim_{\delta\downarrow 0} f_\delta(u,\theta) =  (1_{\{ u < q(\theta)\}} - 1_{\{ u>q(\theta)\}}) (1-u) h(\theta) + 1_{\{u=q(\theta)\}} (1-u)|h(\theta)|\nonumber\\
& =2\cdot1_{\{u \le q(\theta)\}} (1-u) h(\theta) - (1-u)h(\theta)  + 2\cdot 1_{\{u=q(\theta)\}} (1-u)(h(\theta)_-) =: f_{0}(u,\theta),
 \end{align}
 where in the last relation we used the facts that 
 $$1_{\{ u> q(\theta)\}} = 1 - 1_{\{  u\le q(\theta)\}}\ \ \mbox{ and }\ \ 
 |h(\theta)| - h(\theta) = 2 h(\theta)_-.$$

 Next, we will verify the conditions of the Lebesgue DCT.  In view of \eqref{e:l:gateaux-general-f-delta} by applying the
 inequality $||a|-|b||\le |a-b|$, we get
 \begin{equation}\label{e:l:gateaux-general-f-delta-upper}
  |f_\delta(u,\theta)| \le |(1-u) h(\theta)|\le |h(\theta)|.
 \end{equation}
 Since for all $\theta\in \SX$ the upper bound in \eqref{e:l:gateaux-general-f-delta-upper} belongs to $L^1(p(du|\theta))$, by the DCT, 
 the point-wise convergence \eqref{e:l:gateaux-general-3} implies
 \begin{equation}\label{e:l:gateaux-general-4}
  \int_{u\in [0,1]} f_\delta(u,\theta) p(du |\theta) \longrightarrow \int_{u\in [0,1]} f_0(u,\theta) p(du |\theta),\ \ \mbox{ as }\delta\downarrow 0.
 \end{equation}
 Recalling that $p(du|\theta)$ and $p_\Theta(d\theta)$ are probability measures, by using \eqref{e:l:gateaux-general-f-delta-upper}, and the
 established point-wise convergence in \eqref{e:l:gateaux-general-4}, the Lebesgue DCT over the space $L^1(p_\Theta)$, entails
 $$
  \int_{\SX} \Big\{ \int_{u\in [0,1]} f_\delta(u,\theta) p(du |\theta)\Big\} p_\Theta(d\theta) \to \int_\SX \int_{[0,1]} 
  f_0(u,\theta) p(du |\theta) p_\Theta(d\theta),
 $$
 as $\delta\downarrow 0$. In view of \eqref{e:l:gateaux-general-f-delta}, \eqref{e:l:gateaux-general-3} and
 \eqref{e:b(theta)}, the last convergence yields:
 \begin{align} \label{e:l:gateaux-general-1}
     \lim_{\delta\downarrow 0} \frac{I(g+\delta h)-I(g)}{\delta} &= 
     2\int_{\SX} \Big\{ \int_{\{u \le q(\theta)\} } (1-u) p(du|\theta)  \Big\} h(\theta) p_\theta(d\theta)  - \int_{\SX} b(\theta) h(\theta) p_\Theta(d\theta) \\
     & + 2\int_{\SX} \Big\{ \int_{\{u = q(\theta)\} } (1-u) p(du|\theta)  \Big\} (h(\theta)_-) p_\theta(d\theta), \label{e:l:gateaux-general-2}
 \end{align}
 In view of \eqref{e:F-theta}, Relations \eqref{e:l:gateaux-general-1}-\eqref{e:l:gateaux-general-2} can be written as in 
 \eqref{e:lower-bound-via-F-statement}, which completes the proof of the lemma.
 \end{proof}

 \begin{lemma}\label{l:CDF-inverses-statement} Let $F$ be the CDF of a probability distribution on $[0,1]$ and let
 $\alpha \mapsto q_\alpha$ and $\alpha \mapsto q_{\alpha+}$ be its left- and right-continuous generalized inverses, defined as in 
 \eqref{e:q-alpha} and \eqref{e:q-alpha-plus}, respectively.  The following statements hold:

 {\em (i)} For all $0\le \alpha \le 1$, we have
 \begin{equation}\label{e:l:CDF-inverses-1-statement}
   F(q_\alpha -) \le F(q_{\alpha+}-) \le \alpha \le F(q_\alpha)\le F(q_{\alpha+}),
 \end{equation}
 where $F(t-):= \lim_{s<t,\ s\to t} F(s)$ denotes the left-limit of $F$ at $t$.

 {\em (ii)} For all $0\le \alpha\le 1$, such that $q_\alpha<q_{\alpha+}$, we have
 \begin{equation}\label{e:l:CDF-inverses-2-statement}
   F(t) = F(q_\alpha),\ \ \mbox{ for all $ q_\alpha \le t < q_{\alpha+}$.}
  \end{equation}
 \end{lemma}
 \begin{proof} Since $q_\alpha \le q_{\alpha+}$, the monotonicity of $F$ implies $F(q_\alpha -) \le F(q_{\alpha+}-)$.  
 Also, by \eqref{e:q-alpha} we have that $t_n\downarrow q_\alpha$, for some $t_n\ge q_\alpha$ such that $F(t_n)\ge \alpha$.  The right-continuity
 of $F$ then entails $F(t_n)\to F(q_\alpha)$, as $n\to\infty$, and hence $F(q_\alpha)\ge \alpha$.  To complete the proof of \eqref{e:l:CDF-inverses-1-statement},
 it remains to show that $F(q_{\alpha+}-) \le \alpha$.  

 If $q_{\alpha+} = 0$, then by convention $F(q_{\alpha+}-) = 0 \le \alpha$, for all $\alpha\in [0,1]$.  Let now $q_{\alpha+} >0$, and 
 suppose that $F(q_{\alpha+}-) > \alpha$.  Let $0<t_n < q_{\alpha+}$ be such that $t_n\uparrow q_{\alpha+}$. Then, since 
 $F(t_n) \to F(q_{\alpha+}-) >\alpha$, as $n\to\infty$, there exists a $t_n<q_{\alpha+}$ such that $F(t_n) >\alpha$.  This, however, contradicts 
 the definition of $q_{\alpha+}$ in \eqref{e:q-alpha-plus}.  We have thus shown that $F(q_{\alpha+}-) \le \alpha$, completing the proof of 
 \eqref{e:l:CDF-inverses-1-statement}.

 {\em We now prove \eqref{e:l:CDF-inverses-2-statement}.}  By the monotonicity of $F$, we have $F(t)\ge F(q_\alpha)$, for all $t\in[q_\alpha,q_{\alpha+})$. 
 Suppose now that
 \begin{equation}\label{e:l:CDF-inverses-3}
 F(t) > F(q_\alpha) \ge \alpha, \ \mbox{ for some }q_\alpha < t <  q_{\alpha+}.
 \end{equation}
 Relation \eqref{e:l:CDF-inverses-3} contradicts the definition of $q_{\alpha+}$ in \eqref{e:q-alpha-plus}.  
 This completes the proof of \eqref{e:l:CDF-inverses-2}.
 \end{proof}
 
 \begin{lemma}[Statement of Lemma \ref{l:C(g)-properties}] \label{l:C(g)-properties-statement} 
 Let $g_\alpha^{(\lambda)}$ be as in \eqref{e:g-alpha-lambda}.  Then:\\

    {\em (i)} We have that $C(g_\alpha^{(\lambda)})<\infty$,  for all $0\le \alpha <1$ and $0\le \lambda \le 1$.\\
    
    {\em (ii)} The function $\lambda \mapsto C(g_\alpha^{(\lambda)})$ is monotone non-decreasing and continuous in $\lambda \in [0,1]$.\\

    {\em (iii)} The function $\alpha \mapsto C(g_\alpha)$ in monotone non-decreasing and left-continuous in $\alpha \in (0,1]$.\\

    {\em (iv)} For all $0\le \alpha < 1$, we have that $C(g_{\alpha+}) = \lim_{\beta\downarrow \alpha} C(g_\beta).$
   \end{lemma}
   \begin{proof} 
   
   {\em Part (i):} Note that $g_\alpha^{(0)}(\theta) = g_\alpha(\theta) = q_\alpha(\theta)/(1-q_\alpha(\theta))$.
   We begin by showing that $C(g_\alpha) = C(g_\alpha^{(0)}) <\infty,$ for all $0\le \alpha<1$.

    Let $V\sim {\rm Uniform}(0,1)$ be independent from $\Theta \sim p_\Theta$. We will show that 
    \begin{equation}\label{e:l:C(g_alpha)-finite-1}
     \E\Big[ \frac{b(\Theta)}{1-q_V(\Theta)} \Big] =\int_{0}^1  \int_{\SX}  \frac{b(\theta)}{1-q_\alpha(\theta)} p_\Theta(d\theta)  d\alpha  <\infty.
    \end{equation}
    This, in view of the Tonelli-Fubini theorem implies that 
    $$
     \E\Big[ \frac{b(\Theta)}{1-q_\alpha(\Theta)} \Big] = \int_\SX \frac{b(\theta)}{1-q_\alpha(\theta)} p_\Theta(d\theta) <\infty,
    $$
    for Lebesgue almost all $\alpha\in (0,1)$.  The monotonicity of the function $\alpha\mapsto 1/(1-q_\alpha(\theta))$ then entails that,
     for all $\alpha \in [0,1)$, we have
    \begin{equation}\label{e:l:C(g_alpha)-finite-2}
     C(g_\alpha) = \int_{\SX} b(\theta) g_\alpha(\theta) p_\Theta(d\theta)  = \int_{\SX} b(\theta) \frac{q_\alpha(\theta)}{1-q_\alpha(\theta)}
     p_\Theta(d\theta)   \le \int_\SX \frac{b(\theta)}{1-q_\alpha(\theta)} p_\Theta(d\theta) <\infty, 
    \end{equation}
    where we used that $q_\alpha(\theta)\in [0,1]$.  Thus, to show that $C(g_\alpha)<\infty,\ 0\le \alpha<1$, it remains to
    prove \eqref{e:l:C(g_alpha)-finite-1}, which we do next.
    
    Note that since $\alpha\mapsto q_\alpha(\theta)$ is the left-continuous inverse of $F_\theta$, we have that $q_V(\theta)$ has the 
    probability distribution $b(\theta)^{-1}(1-u)p(du|\theta)$ \citep[see, e.g., p. 61 in ][]{resnick:1999book}.  
    Thus, by the independence of $\Theta$ and $V$, we obtain 
    $$
    \E \Big[ \frac{b(\Theta)}{1-q_V(\Theta)} \, \vert\, \Theta \Big] = \int_{[0,1)} \frac{b(\Theta)}{(1-u)} \frac{(1-u)}{b(\Theta)} p(du |\Theta) =
    p([0,1)|\Theta) = 1.
    $$
    Hence, by taking an expectation over $\Theta$ in the last relation, we obtain 
    $\E[ b(\Theta)/(1-q_V(\Theta)) ] = 1 <\infty$, which implies
    \eqref{e:l:C(g_alpha)-finite-1} and proves that $C(g_\alpha) <\infty$ for all $0\le \alpha <1$.\\

   Now that we have shown $C(g_\alpha)<\infty$, for all $0\le \alpha<1$, we proceed with also 
   proving that $C(g_\alpha^{(\lambda)})<\infty$, for all $0\le \lambda \le 1$.  (Note that we exclude the 
   case $\alpha=1$, since in fact one often has $C(g_1) = \infty$.) Since the function $x\mapsto x/(1-x)$
   is increasing for $x\in [0,1)$, we have 
   \begin{equation}\label{e:l:C(g)-properties-1} 
    0\le g_\alpha^{(\lambda)}(\theta) = \frac{ q_\alpha^{(\lambda)}(\theta)}{ 1- q_\alpha^{(\lambda)}(\theta)} \le g_\beta(\theta),
   \end{equation}
   for all $\alpha<\beta<1$, since $q_\alpha^{(\lambda)}(\theta)\le q_{\beta}(\theta)$.  Since we have already shown that 
   $g_\beta \in L_+^1(b\cdot p_\Theta(d\theta))$,  the latter inequality implies $g_\alpha^{(\lambda)}(\theta)$, i.e.,
    $C(g_\alpha^{(\lambda)})<\infty$ and completes the proof of part (i).  Observe that $g_\alpha^{(1)} = g_{\alpha+}$ and so
    we have $C(g_{\alpha+})<\infty$, for all $\alpha \in [0,1)$.\\

Now, we prove {\em (ii)}.  Observe that the function $\lambda \mapsto q_\alpha^{(\lambda)}(\theta)$ is continuous in $\lambda$ 
for all $\theta$.  This, in view of \eqref{e:l:C(g)-properties-1}  and the Lebesgue dominated convergence theorem 
implies the continuity of $\lambda \mapsto C(g_\alpha^{(\lambda)})$ in $\lambda \in [0,1]$. 

Part {\em (iii)} is an immediate consequence of the Monotone Convergence Theorem and the observation that $\alpha \mapsto g_\alpha (\theta)$ are 
non-negative, monotone non-decreasing, and left-continuous in $\alpha$, for all $\theta\in \SX$.

To prove {\em (iv)}, observe first that for all $0\le \alpha<1$ and any fixed $\beta_0\in (\alpha, 1)$, we have that 
$$
 0\le g_{\alpha}(\theta) \le g_\beta(\theta) \le g_{\beta_0}(\theta),\ \ \mbox{ for all }\ \alpha<\beta<\beta_0.
$$
As shown in part {\em (i)}, we have $g_{\beta_0}\in L_+^1(b\cdot p_\Theta)$.  Also, by the fact that 
$q_{\alpha+}(\theta) = \lim_{\beta\downarrow\alpha } q_\beta(\theta)$ and the above dominance by an integrable function $g_{\beta_0}$, the
Lebesgue dominated convergence theorem implies the desired claim in {\em (iv)}.
   \end{proof}

\subsubsection{Proofs of Section \ref{sec:Breiman-models}}\label{sec:proofs_breiman_model}

  \begin{proof}[Proof of Proposition \ref{p:total-Breiman}.]
  The homogeneity of $\tau$ and Breiman's Lemma \ref{l:Breiman} entail
  $$
 \frac{\P[\tau(Z)>t]}{\P[\xi>t]}  = \frac{\P[\xi \tau(\eta) >t ]}{\P[\xi >t]} = \E[\tau(\eta)^\alpha],\ \ \mbox{ as }
 \to \infty,
 $$
 which proves \eqref{e:p:total-Breiman-i}.

 On the other hand, by the regular variation of $\xi$, for all $v>0$
 $$
 \frac{\P[\xi >t/v]}{\P[\xi >t ]} \to v^\alpha, \ \ \mbox{ as }t\to\infty.
 $$
 Thus, in view of \eqref{e:p:total-Breiman-i}, for all $w \in \{\tau>0\}$,
  $$
  p_t(w):= \frac{\P[ \xi > t/\tau(w)]}{\P[\tau(Z)> t]} \to p_\infty(w):= 
  \frac{\tau(w)^\alpha}{\E[\tau(w)^\alpha]},\ \ t\to\infty.
  $$
  Observe that both $p_t(w)$ and $p_\infty(w)$ are probability densities relative to the probability distribution
  $P_\eta(dw)$ of the random vector $\eta$ over the domain $\{w\,:\, \tau(w)>0\}$. Thus, as in 
  the proof of Lemma 3.11 in \cite{scheffler:stoev:2017}, the Scheff\'e theorem entails
  \begin{equation}\label{e:p:total-Breiman-1}
   \int_{\tau(w)>0} |p_t(w) -p_\infty(w)| P_\eta(dw) \to 0,\ \ t\to\infty.
  \end{equation}
  On the other hand, by the
  independence of $\xi$ and $\eta$, and the fact that $Z/\tau(Z) = \eta/\tau(\eta)$, we can write
\begin{align*}
  \P[ Z/\tau(Z) \in B | \{\tau(Z)>t\}] & = \frac{1}{\P[\tau(Z)>t]} \P[ \eta/\tau(\eta) \in B, \xi >t/\tau(\eta)]\\
  & = \int_{\{\tau(w)>0\}} 1_{B}(w/\tau(w)) p_t(w) P_\eta(dw),
 \end{align*}
 for all $B\in {\cal B}(S_\tau)$.  This fact and Relation \eqref{e:p:total-Breiman-1} yields the desired convergence 
 in total variation in \eqref{e:p:total-Breiman-ii}.
  \end{proof}

  \begin{proposition}[Statement of Proposition \ref{p:Pareto-Breiman} on the Pareto-Breiman models] \label{p:Pareto-Breiman-restatement} Let $Y$ and $X$ be as in \eqref{e:Y_X_Breiman} with $\xi$ as in \eqref{e:l:Breiman-i} with
$\alpha=1$. Let also $\tau(y,x):= y_+ + \tau_X(x)$, be a non-negative, continuous $1$-homogeneous such that
$0<\E[V]<\infty$ and $0<\E[\tau_X(W)]<\infty$. Then:\\

{\em (i)} $Y$ and $X$ are jointly regularly varying in the sense of Assumption \ref{a:X-Y-joint} and
\begin{equation}\label{e:p:Pareto-Breiman-i-statement}
\frac{(Y,X)}{\tau(Y,X)} \, \vert\, \{\tau(Y,X)>t\} \stackrel{TV}{\longrightarrow} \Theta_Z =:(U , (1-U) \Theta),\ \ \mbox{ as }t\to\infty.
\end{equation}
We have moreover that for all $B\in {\cal B}(\SX)$,
\begin{equation}\label{e:p:Pareto-Breiman-i-Theta-statement}
\P[ \Theta \in B\, |\, U<1] =\frac{1}{\E[ 1_{\{\tau_X(W)>0\}}\cdot (V + \tau_X(W))]}\E\Big[1_B\Big(\frac{W}{\tau_X(W)}\Big) \cdot 1_{\{\tau_X(W)>0\}} (V+\tau_X(W)) \Big].
\end{equation}

{\em (ii)} For all non-negative, measurable homogeneous $h:\R^{1+d}\to\R_+$, such that $\E[ h(V,W) ]<\infty$, we have
\begin{equation}\label{e:p:Pareto-Breiman-ii-statement}
 t\P[ h(Y,X) >t] \to c_\xi \E[h(V,W)],\ \ \mbox{ as }t\to\infty.
\end{equation}

{\em (iii)} If $h$ is as in part {\em (ii)} and such that $\{h>0\}\subset \{\tau >0\}$, then 
\begin{equation}\label{e:p:Pareto-Breiman-iii-statement}
\E [ h(U,(1-U)\Theta) ]  = \frac{1}{\E[\tau(V,W)]} \E[h(V,W)].
\end{equation}
\end{proposition}
\begin{proof}[Proof of Proposition \ref{p:Pareto-Breiman}]

{\em Part (i):} Proposition \ref{p:total-Breiman} readily implies the regular variation of $Z:=(Y,X)$ and the convergence in \eqref{e:p:Pareto-Breiman-i-statement} in 
the strong total variation sense.  It remains to argue that $Y$ and $X$ are jointly regularly varying.  To this end, it is enough to show that $\P[U=0]<1$ 
and $\P[U=1]<1$.

We shall apply \eqref{e:p:total-Breiman-ii-sigma(A)} with the set
$
  A = \{ \theta_Z=(u,\widetilde \theta_Z)\, :\, u=0\}\subset S_{\tau},
$
where $\widetilde \theta_Z = (\theta_Z(i))_{i=2}^{d+1}$.  In this case, 
$\eta:=(V,W)$ and $\tau(\eta) = \tau(V,W) = V + \tau_X(W)$, and thus we obtain
\begin{align*}
\P[U=0] = \sigma(A) &= \frac{1}{\E[ V + \tau_X(W)]} \E\Big[1_A\Big( \frac{(V,W)}{\tau(V,W)} \Big) \cdot  (V + \tau_X(W)) \Big] \\
& = \frac{1}{\E[ V + \tau_X(W)]} \E[1_{\{V = 0\}} (V + \tau_X(W)) ] \le  \frac{ \E[ \tau_X(W)]}{\E[ V]  + \E[\tau_X(W)]} < 1,
\end{align*}
where the last inequality follows from the assumption that $\E[V]>0$ and $\E[\tau_X(W)]>0$.  This shows that $\P[U=0]<1$.

Similarly, letting $A = \{(u,\widetilde \theta_Z)\, :\, u=1\} = \{(u,\widetilde \theta_Z)\, :\, \tau_X(\widetilde \theta_Z) = 0\}$, by
\eqref{e:p:total-Breiman-ii-sigma(A)}, we obtain
\begin{equation}\label{e:p:Pareto-Breiman-0}
\P[U=1] = \sigma(A) = \frac{1}{\E[ V + \tau_X(W)]} \E\Big[1_{\{\tau_X(W)=0\}}\cdot (V + \tau_X(W))\Big] \le \frac{\E[V]}{\E[V] + \E[\tau_X(W)]} <1.
\end{equation}
This completes the proof \eqref{e:a:X-Y-joint}, which establishes the joint regular variation of $Y$ and $X$.

To prove \eqref{e:p:Pareto-Breiman-i-Theta}, let now $B\in {\cal B}(\SX)$ and define
$$
A= \{ (u,(1-u)\theta)\, :\, u<1,\ \theta\in B\}.
$$
By \eqref{e:p:total-Breiman-ii-sigma(A)}, we have
\begin{align} \label{e:p:Pareto-Breiman-1}
\sigma(A)& = \P[\Theta\in B,\ U<1]  = \frac{1}{\E[V + \tau_X(W)]}\E\Big[ 1_A\Big( \frac{(V,W)}{\tau(V,W)} \Big) \cdot  (V + \tau_X(W)) \Big]\nonumber\\
&=\frac{1}{\E[V + \tau_X(W)]}\E\Big[ 1_B(W/\tau_X(W)) 1_{\{\tau_X(W)>0\}} \cdot  (V + \tau_X(W)) \Big],
\end{align}
where in the last relation we used the fact that $(v,w)/\tau(v,w) \in A$ if and only if $\tau_X(w)>0$ and 
$$
 (w/\tau(v,w))\cdot (1-v/\tau(v,w))^{-1} = \frac{w}{v+\tau_X(w)}\cdot \frac{v + \tau_X(w)}{\tau_X(w)} = \frac{w}{\tau_X(w)} \in B.
$$
By \eqref{e:p:Pareto-Breiman-1} with $B := \SX$ or in fact in view of \eqref{e:p:Pareto-Breiman-0}, we 
obtain $$
\P[U<1] = \frac{1}{\E[V+\tau_X(W)]}\E[ 1_{\{\tau_X(W)>0\}} \cdot (V+\tau_X(W) ].
$$
Taking the ratio of \eqref{e:p:Pareto-Breiman-1} with the last expression
yields \eqref{e:p:Pareto-Breiman-i-Theta}.\\

{\em Part (ii):}  By observing that $\tau(Z)= \tau(\xi \cdot( V,W)) = \xi\cdot \tau(V,W)$, we see that \eqref{e:p:Pareto-Breiman-ii-statement}
is an immediate consequence of Relations \eqref{e:l:Breiman-i} and \eqref{e:p:total-Breiman-i}.  \\

{\em Part (iii):} Let now $h$ be a non-negative homogeneous function such that $\{h>0\} \subset \{\tau>0\}$. 
If $h(z) := 1_A(z/\tau(z)) \cdot \tau(z)$, for some Borel $A\subset S_\tau$, then Relation \eqref{e:p:Pareto-Breiman-iii-statement} is precisely \eqref{e:p:total-Breiman-ii-sigma(A)},
since by definition $\sigma(A) = \E[1_A(U,(1-U)\theta)] = \E[h(U,(1-U)\Theta)]$.  Thus, Relation \eqref{e:p:Pareto-Breiman-iii-statement} immediately extends to simple homogeneous functions
$$
h(z) = \tau(z) \cdot \sum_{i=1}^n a_i 1_{A_i}(z/\tau(z)),\ \ \ (A_i\in {\cal B}(S_\tau))
$$
where $a_i\ge 0$ and by convention $\tau(z) 1_{A_i}(z/\tau(z))=0$ whenever $\tau(z) = 0$.  The claim of part (iii) follows by appealing to the Monotone 
Convergence Theorem, which applies in particular to non-negative functions $h$ such that $\E[h(V,W)] = \infty$. 
\end{proof}

\begin{corollary}[Statement of Corollary \ref{c:Pareto-Breiman}]\label{c:Pareto-Breiman-statement} Let $Y$ and $X$ be as in 
Proposition \ref{p:Pareto-Breiman}.  Suppose that $g= g^{(\rm opt)}$ is as in Theorem \ref{thm:final}.  That is,
$g\in L_+^1(b\, \cdot\,  p_\Theta)$, with $b(\theta)$ and $p_\Theta$ as in \eqref{e:b(theta)} and \eqref{e:p_U,Theta}, 
$C(g) = \E[(1-U)g(\Theta)] = \E[U]$, and 
$$
\Lambda(g) = \sup_{f\in L_+^1(b\, \cdot\,  p_\Theta),\ C(f) = \E[U]} \Lambda(f).
$$

Define the homogeneous extension of $g$:
$$
 h_g(x):= \tau_X(x) g(x/\tau_X(x)),\ \  x\in \R^d,
$$
where by convention $h_g(x) = 0$ whenever $\tau_X(x)=0$.  Then:\\

{\em (i)} We have
\begin{equation}\label{e:c:Pareto-Breiman-i-statement}
\lim_{t\to\infty} t\P[Y>t] =  \lim_{t\to\infty}  t\P[h_g(X)] = c_\xi \E[\tau(V,W)] C(g).
\end{equation}

{\em (ii)} For every non-negative, homogeneous and Borel measurable function $h:\R^d\to \R_+$ such that $\{h>0\} \subset\{\tau_X>0\}$ and
$\P[Y>t] \sim \P[ h(X)>t],\ t\to\infty$, we have $h\vert_{\SX}\in L_+^1(b\, \cdot\,  p_\Theta)$ and
\begin{equation}\label{e:c:Pareto-Breiman-ii-statement}
\Lambda(h) = \lambda(Y,h(X)) = \lim_{t\to\infty} \P[Y>t | h(X) >t] \le \Lambda(g) = \lambda(Y,h_g(X)). 
\end{equation}
where the above limits exist and $\Lambda$ is as in \eqref{e:Lambda-functional}. 
\end{corollary}
\begin{proof}[Proof of Corollary \ref{c:Pareto-Breiman}] The argument is similar to the proof of Lemma \ref{lem:nonameidea} but instead of using multivariate 
regular variation, we shall appeal to Breiman's lemma.  Observe that the non-negative homogeneous functions $h_Y(y,x):= y_+,\ h_X(y,x):=  h_g(X)$ and $h_{Y,X}(y,x):= y_+\wedge h_g(x)$ 
are by construction support-dominated by $\tau$. That is, for all $h\in \{h_Y,h_X,h_{Y,X}\}$ we have $\{h>0\}\subset \{\tau>0\}$, where $\tau(y,x)= y_++\tau_X(x)$.  
Therefore, Relations \eqref{e:p:Pareto-Breiman-ii} and \eqref{e:p:Pareto-Breiman-iii} of Proposition \ref{p:Pareto-Breiman} applied to $h = h_Y$ and $h=h_X$ 
imply \eqref{e:c:Pareto-Breiman-i-statement}.  Taking $h=h_{Y,X}$, we also obtain
\begin{align*}
\lim_{t\to\infty} t\P[ Y>t, h_g(X)>t] &= \lim_{t\to\infty} t\P[h_{Y,X}(Y,X)>t] = c_\xi \E[ V\wedge h_g(W) ]\\
& = c_\xi \E[ \tau(V,W) ] \E[ U\wedge(1-U)g(\Theta)].
\end{align*}
By taking the ratio of the last relation and the already established \eqref{e:c:Pareto-Breiman-i-statement}, we obtain
$$
 \lim_{t\to\infty} \P[ Y>t | h_g(X)>t ] = \Lambda(g),
$$
which, in view of  Lemma \ref{l:xi-eta-tdep} applied to $\xi:=Y$ and $\eta:=h_g(X)$, shows that the tail-dependence coefficient exists and 
$\lambda(Y,h_g(X)) = \Lambda(g)$.

Let now $h$ be an arbitrary homogeneous function as in part {\em (ii)}.  Since $\{h>0\}\subset \{\tau_X>0\}$ we have that
$h(x) = \tau_X(x) h(x/\tau_X(x))$, where by convention $h(x)=0$ if $\tau_X(x)=0$. That is, $h$ can be represented as the homogeneous extension 
of its restriction $f:= h\vert_{\SX}$ onto $\SX=\{\tau_X=1\}$.

Suppose for a moment that $f:=h\vert_\SX \in L_+^1(b\, \cdot\,  p_\Theta)$. Then, by the assumed tail-equivalence 
$\P[Y>t] \sim \P[h(X)>t],\ t\to\infty$, as in the first part of the proof, we obtain
\begin{equation}\label{e:c:Pareto-Breiman-1}
\lim_{t\to\infty} t\P[Y>t] = \lim_{t\to\infty} t\P[h(X)>t] = c_\xi \E[ \tau(V,W) ] C(f),
\end{equation}
and hence $C(g) = C(f)$ in view of \eqref{e:c:Pareto-Breiman-i-statement}.  Similarly, as argued above with $h_g$ replaced by $h=h_f$, we obtain
$$
\Lambda(f) = \lambda(Y,h(X)) = \lim_{t\to\infty} \P[Y>t | h(X)>t].
$$
Since, both $f$ and $g$ satisfy the constraint $C(f)=C(g) = \E[U]$, Theorem \ref{thm:final} readily implies $\Lambda(f)\le \Lambda(g)$,
which proves \eqref{e:c:Pareto-Breiman-ii-statement}.

Therefore, to complete the proof it remains to show that for every $\tau$-support dominated and asymptotically calibrated 
predictor $h(X)$, we have $f = h\vert_\SX \in L_+^1(b\, \cdot\,  p_\Theta)$.  We will do so next by the method of truncation.  

Define $h^{(M)}(x) := \tau_X(x) h(x/\tau_X(x)) 1_{\{ h(x/\tau_X(x)) \le M \}}$ where by convention $h^{(M)}(x) :=0$ whenever $\tau_X(x)=0$.  
Observe that since $\{h>0\}\subset \{\tau_X>0\}$, we have $h^{(M)}(x) \uparrow h(x),$ as $M\to\infty$.  Also, we have that $h^{(M)} (W) \le \tau_X(W) \cdot M <\infty$ 
and hence $\E[ h^{(M)} (W)] \le M \E[\tau_X(W)] <\infty$.  Thus, by Breiman's Lemma \ref{l:Breiman} 
$$
\lim_{t\to\infty} t\P[ h^{(M)}(X) > t] =\lim_{t\to\infty} t\P[ \xi h^{(M)}(W) > t]  = c_\xi \E[ h^{(M)}(W) ].
$$
On the other hand, since $0\le h^{(M)}(x)\le h(x)$ and in view of \eqref{e:c:Pareto-Breiman-1}, we obtain
$$
c_\xi \E[ h^{(M)}(W) ] = \lim_{t\to\infty} t \P[ h^{(M)}(X)>t] \le \lim_{t\to\infty}  t \P[h(X)>t] = \lim_{t\to\infty} t\P[Y>t] <\infty.
$$
Since the latter upper bound does not depend on $M$ and $c_\xi>0$, by the Monotone Convergence Theorem, we have that
$\E[ h^{(M)}(W)] \uparrow \E[ h(W)]  \le \E[ \tau(V,W) ] C(g) <\infty$, as $M\uparrow\infty$.  Relation \eqref{e:p:Pareto-Breiman-iii} applied
to the homogeneous function $(y,x)\mapsto  h(x)$ yields
$$
\E[h(W)]  =\E[ \tau (V,W)] \E[ h((1-U) \Theta) ] = \E[ \tau (V,W)]\E [(1-U) f(\Theta) ]<\infty.
$$
This shows that $f\equiv h\vert_\SX \in L_+^1(b\cdot p_\Theta)$, which completes the proof.
\end{proof}

\subsection{Proofs of Section \ref{sec:inference}}\label{sec:inference_proofs}

In this section we provide proofs and auxiliary results needed for Theorem \ref{thm:univ_consitency}. 
We start by providing a contiguity argument, which will be used to formally justify
the Peaks-over-Threshold methodology.  This argument is of independent interest since it allows one to 
effectively {\em plug-in} a sample from the limit angular distribution $\sigma$ in place of the 
triangular array coming from the exceedance distributions for a sequence of growing thresholds.

\subsubsection{More on contiguity and the d\`ecoupage de L\'evy}\label{subsec:contiguity_argument}

We begin by recalling established results and notation on Hellinger and total variation distances of 
measures. Let $P$ and $Q$ be two probability measures, defined on a common measure space $(S,{\cal S})$.
Recall that the {\em squared} Hellinger distance of $P$ and $Q$ is defined as:
 $$
 H^2(P,Q):= \frac{1}{2} \int_S \Big(\sqrt{\frac{dP}{d\nu}} - \sqrt{\frac{dQ}{d\nu}}\Big)^2 d\nu,
 $$
 where $\nu$ is any other measure dominating $P$ and $Q$, i.e., $\nu = P + Q$.

 We shall denote by $P^{\otimes k}$ the product measure on the cartesian product space $(S^k,\sigma({\cal S}^k))$.
 Note that $P^{\otimes k}$ is the probability distribution of a sample of $k$ independent elements
 taking values in $S$ and having probability distribution $P$. The following inequalities are well known 
 \citep[see e.g.][]{vaart:1998,lecam:yang:2000,pollard:2002}
 \begin{equation}\label{e:Hellinger-product-inequality}
 H^2( P^{\otimes k}, Q^{\otimes k}) \le k H^2(P,Q)
 \end{equation}
 and 
 \begin{equation}\label{e:Hellinger-TV-inequality}
  H^2(P,Q) \le TV(P,Q) \le \sqrt{2} H(P,Q),
 \end{equation}
 where $ TV(P,Q)$ stands for the total variation distance:
 \begin{equation}\label{e:TV-def}
 TV(P,Q):= \sup_{A\in {\cal F}} |P(A)- Q(A)|.
 \end{equation}

Let now $P_t$ be the conditional distribution of $Z/\tau(Z)\, |\, \tau(Z)>t$ and let
 $P_\infty$ denote its limit.  

 \begin{assumption} \label{a:Hellinger} The probability measures $P_\infty$ and $\{P_t,\ t=1,2,\ldots\}$
 are such that 
 $$
 H(P_t,P_\infty) \to 0,\ \ \mbox{ as }t\to\infty. 
 $$  
 \end{assumption}

  The next result is similar in spirit to Le Cam's First Lemma (see also Remark \ref{r:LeCam_first_lemma} below).
  
 \begin{proposition}\label{p:contiguity} Suppose $P_\infty$ and $\{P_t\}$ satisfy Assumption \ref{a:Hellinger}. 
 Let also $\Theta_i,\ i=1,\ldots $ be iid from $P_\infty$ and 
 $\Theta_i(t),\ i=1,2,\ldots$ be iid from $P_t$.  Consider:
 
 (1) A statistic  $T_k = T_k(\Theta_1,\cdots,\Theta_k)$ such that 
 $$
  T_k\stackrel{\P}{\to} 0,\ \ \mbox{ as }k\to\infty
 $$

 (2) A random integer sequence $\{K(t)\}$, independent from $\{\Theta_i(t)\}$ such that 
 $$
  K(t)\stackrel{\P}{\to}\infty\ \ \mbox{ and }\ \ K(t) H^{2}(P_t,P_\infty)= o_{\P}(1),\ \ \mbox{ as }t\to\infty.
 $$
 
 Then, (1) and (2) imply that 
 $$
 \wt T_{K(t)} := T_{K(t)} (\Theta_1(t),\cdots,\Theta_{K(t)} (t))\stackrel{\P}{\to} 0,\ \ \mbox{ as }t\to\infty.
 $$
 \end{proposition}
 
 \begin{remark} \label{r:LeCam_first_lemma} In view of \eqref{e:Hellinger-TV-inequality} and \eqref{e:TV-def}, Assumption \ref{a:Hellinger} readily implies that $P_t$ and 
 $Q_t:= P_\infty$ are contiguous \citep[see e.g. Definition 6.3 in ][]{vaart:1998}. 
 In principle, Proposition \ref{p:contiguity} can be also derived as a consequence of Le Cam's First Lemma 
 \citep[see e.g.\ Lemma 6.4 in ][]{vaart:1998}.  Nevertheless, here we elected to provide our own argument 
 since it makes the connection between the rate of convergence in the Hellinger distance and the random sample size more explicit.
 \end{remark}
 
 \begin{proof}[Proof of Proposition \ref{p:contiguity}]
 By assumption, for all $\epsilon>0$, we can find two deterministic integer sequences 
 $k_t^{(0)}\le k_t^{(1)}$ such that $k_t^{(0)}\to \infty$, as $t\to\infty$, 
 and for all $t\ge t_\epsilon$
 \begin{equation}\label{e:p:contiguity-1}
   k_t^{(1)} \le \frac{\epsilon^2}{2 H^2(P_t,P_\infty)},\ \ \ \mbox{ and }\ \ \ 
  \P[  k_t^{(0)} \le K(t) \le k_t^{(1)}] \ge 1-\epsilon.
 \end{equation}

 Thus, for all $t\ge t_\epsilon$, we have that, 
 \begin{align}\label{e:p:contiguity-2}
  \P[ |\wt T_{K(t)} (t)| >\epsilon ] &= \Big( \sum_{k< k_t^{(0)}} + \sum_{k\in [k_t^{(0)},k_t^{(1)}]} + 
  \sum_{k> k_t^{(1)}} \Big) \P[ |\wt T_{k} (t)| >\epsilon,\ K(t) = k ] \nonumber \\
  &\le \P[ K(t) \not \in [k_t^{(0)},k_t^{(1)}]] + \sum_{k=k_t^{(0)}}^{k_t^{(1)}} \P[ |\wt T_k|>\epsilon] \P[ K(t)=k]\nonumber \\
 & \le 2\epsilon + \sum_{k=k_t^{(0)}}^{k_t^{(1)}} \P[ |\wt T_k|>\epsilon] \P[ K(t)=k],
 \end{align}
 where the last two relations follow from \eqref{e:p:contiguity-1} and the independence of 
 $K(t)$ and the $\Theta_i(t)$'s.
 
 Now, by appealing to the definition of total variation distance in \eqref{e:TV-def} and using the
 inequalities \eqref{e:Hellinger-product-inequality} and \eqref{e:Hellinger-TV-inequality},
 for the $k$th term in \eqref{e:p:contiguity-2}, we obtain
 \begin{align}
 \P[ |\wt T_k|>\epsilon] & \le \P[ |T_k|>\epsilon]) + TV(P_t^{\otimes k},P_\infty^{\otimes k}) \nonumber\\
 & \le \P[ |T_k|>\epsilon]) + \sqrt{2} H(P_t^{\otimes k},P_\infty^{\otimes k}) \nonumber \\
 & \le \P[ |T_k|>\epsilon] + \Big( 2 k H^2(P_t,P_\infty) \Big)^{1/2}.
 \end{align}
 By the assumption on the sequences in \eqref{e:p:contiguity-1}, we further obtain that
 $ 2 k H^2(P_t,P_\infty)\le \epsilon^2$, for all $k\in [k_t^{(0)},k_t^{(1)}]$, and hence 
 $$
  \P[ |\wt T_k|>\epsilon] \le  \P[ |T_k|>\epsilon] + \epsilon, 
 $$
 which, in view of \eqref{e:p:contiguity-2} yields
 $$
 \P[ |\wt T_{K(t)} (t)| >\epsilon ] \le 3\epsilon + 
  \sum_{k=k_t^{(0)}}^{k_t^{(1)}} \P[ |T_k|>\epsilon] \P[ K(t)=k] \le 3 \epsilon + \sup_{k\ge k_t^{(0)}}\P[ |T_k|>\epsilon].
 $$
 However, since by assumption $T_k\stackrel{\P}{\to} 0$, as $k\to\infty$ and since $k_t^{(0)}\to\infty$,
 we have that $\sup_{k\ge k_t^{(0)}}\P[ |T_k|>\epsilon]\to 0$, as $t\to\infty$.  Hence, by possibly
 increasing the value of $t_\epsilon$, we have that 
 $$
 \P[ |\wt T_{K(t)} (t)| >\epsilon ] \le 4\epsilon, \mbox{ for all }t\ge t_\epsilon.
 $$
 This, since $\epsilon>0$ was arbitrary implies that $\wt T_{K(t)} \stackrel{\P}{\to} 0$, as $t\to\infty$.
 \end{proof}

 The following remark shows why the assumption of independence of $K(t)$ and the $\Theta_i$'s
 in Proposition \ref{p:contiguity} is not restrictive in the setting of peaks over threshold. 
 
\begin{remark}[decoupage de L\'evy, see e.g., \cite{resnick:1987}] \label{rem:decoupage-de-Levy}
 Let $Z_i,\ i=1,2,\ldots,n$ be iid and define the set of
 exceedances ${\cal I}(t) :=\{ i\in [n]\, :\, \tau(Z_i)>t \}$.  Then, letting 
 $\Theta_i:=Z_i/\tau(Z_i),\ i\in {\cal I}(t)$, we have the following equality in distribution 
 in the sense of random point measures
 \begin{equation}\label{e:Pi-Pi*}
  \Pi:= \{\Theta_i,\ i\in {\cal I}(t)\} \stackrel{d}{=} \Pi^* :=\{ \Theta^*_1,\cdots,\Theta^*_{K(t)}\},
 \end{equation}
 where $K(t):=|{\cal I}(t)|$ and $\Theta_1^*,\Theta_2^*,\ldots$ are iid and independent from $K(t)$ 
 such that $\Theta^*_i \stackrel{d}{=} Z_i/\tau(Z_i) | \tau(Z_i)>t$.
 In simple terms, the values of the $\Theta_i$'s are independent of the number of exceedances $K(t)$.
\end{remark}

\begin{proof}[Proof of equation \eqref{e:Pi-Pi*}]
 To prove \eqref{e:Pi-Pi*}, it is enough to show that the Laplace functionals of $\Pi$ and $\Pi^*$ coincide. That is, 
 we will show that for all non-negative measurable functions $f:\R^d\to \R_+$, we have 
 $\E[ e^{-\Pi(f)}] = \E[ e^{-\Pi^*(f)}]$ in standard Resnick notation \citep[][]{resnick:1987}.  We have
 \begin{align}\label{e:r:decoupage_de_Levy}
 \E[ e^{-\Pi(f)}] & = \E\Big[ \exp\Big\{ -\sum_{i\in {\cal I}(t)} f(Z_i/\tau(Z_i))\Big\} \Big] 
 = \E    \Big[ \exp\Big\{ -\sum_{i=1}^n g(Z_i) \Big\} \Big],
 \end{align}
 where $g(z) = f(z/\tau(z)),\ \tau(z)>t$ and $g(z) = 0$, if $\tau(z)\le t$.  The independence of the $Z_i$'s then entails
 \begin{align*}
\E[ e^{-\Pi(f)}] &= \Big(\E [ e^{-g(Z_1)} ]\Big)^n = \Big( \E[ e^{-f(\Theta)}] \P[ \tau(Z_1)>t] + (1-\P[\tau(Z)>t])\Big)^n \\
& = \sum_{k=0}^n {n \choose k} p^k (1-p)^{n-k} \Big(\E[ e^{-f(\Theta)}])^k,
 \end{align*}
 where $\Theta := Z_1/\tau(Z_1)\, |\, \tau(Z_1)>t$, and $p:= \P[ \tau(Z_1)>t]$.

 On the other hand, using the independence of $K(t)$ and the $\Theta_i^*$'s, we obtain
 $$
\E[ e^{-\Pi^*(f)}] = \sum_{k=0}^n \Big(\E[ e^{-f(\Theta^*_1)}]\Big)^k \P[ K(t) = k] = {n \choose k} p^k (1-p)^{n-k} \Big(\E[ e^{-f(\Theta^*_1)}]\Big)^k,
 $$
 which equals the right-hand side of \eqref{e:r:decoupage_de_Levy} completing the proof of \eqref{e:Pi-Pi*}.
\end{proof}

 \subsubsection{On the construction of uniformly consistent estimators}\label{sec:unif_consistency_for_g}

 In this section we demonstrate that a point-wise consistent estimator of a continuous function
 can be modified to obtain a uniformly consistent estimator over compacta.   Although surprising, at a first
 glance, this fact is akin to the {\em method of  sieves} and is a direct consequence of the
 {\em principle of the diagonal}. The caveat is that this modification may ultimately converge 
 at a potentially slow rate.  We were unable to find a reference for this simple fact and therefore we give 
 it here for the sake of completeness.

 \begin{proposition}\label{p:unif_cont} Let $g_n: K \to \R$ be a sequence of random measurable 
 functions defined on a compact set $K \subset \R^d$.   Suppose that  for all $\theta\in K$, we 
 have $g_n(\theta)\stackrel{\P}{\to} g(\theta)$, as $n\to\infty$,  where $g:K\to \R$ is continuous.  
 
 Then, for some $m_n\to\infty$, there exist continuous functions $\wt g_n: \R^d \to \R$ defined 
 over $\R^d$, which are measurable functions of $g_n(\theta), \theta\in K$ (but not of $g(\cdot)$), 
 such that
 \begin{equation}\label{e:p:unif_cont}
  \sup_{\theta \in K} | \wt g_n(\theta)  - g(\theta) | \stackrel{\P}{\longrightarrow} 0,\ \ \mbox{ as }n\to\infty.
 \end{equation}
 \end{proposition}
 
 We begin with the following 
 \begin{lemma}\label{l:diagonal} Adopt the assumptions of Proposition \ref{p:unif_cont}. Let ${\cal I} = \{\theta_i\}_{i\in \N} 
 \subset K$ be an arbitrary countable subset of $K$. Then, there is a sequence $m_n\to\infty$ such that
 \begin{equation}\label{e:l:diagonal}
   \max_{1\le j\le m_n} | g_n(\theta_j) - g(\theta_j)| \stackrel{\P}{\longrightarrow} 0,\ \ \mbox{ as }n\to\infty.
 \end{equation}
 \end{lemma}
 \begin{proof} By the point-wise convergence in probability $g_n(\theta) \stackrel{\P}{\to}g(\theta),\ n\to\infty$,
 for all $\theta_j\in {\cal I}$, and $k\in \N$, there exist $N(k,j)<\infty$ such that 
 \begin{equation}\label{e:l:diagonal-0}
 \P\Big[ |g_n(\theta_j) - g(\theta_j)| \ge \frac{1}{k} \Big] \le \frac{1}{2^j k},\ \ \mbox{ for all }n\ge N(k,j),
 \end{equation}
 where without loss of generality, we suppose that $N(k,j)$ is strictly increasing in $j$ and $k$.  
 Then, for all $m\in\N$, we have 
 \begin{equation}\label{e:l:diagonal-1}
 \P \Big[ \max_{1\le j\le m} | g_n(\theta_j) - g(\theta_j)| \ge \frac{1}{k} \Big]\le 
\sum_{j=1}^m  \P \Big[ | g_n(\theta_j) - g(\theta_j)| \ge \frac{1}{k} \Big]\le 
 \sum_{j=1}^{m}\frac{1}{2^j k} \le \frac{1}{k},\ 
  \end{equation}
  for all $n\ge N^*(k,m):= \max_{1\le j \le m} N(k,j)$.  For all $n\ge N^*(1,1)$ define
  $$
  m_n:= \max\{ m\ge 1\, :\,  n \ge N^*(m,m)\}.
  $$
  Observe that $m_n<\infty$  since $\{ N^*(m,m) \}_{m\in\N}$ is a strictly increasing integer sequence and
  hence $N^*(m,m)\uparrow \infty$, as $m\to\infty$.  Moreover, we have
  $m_n\to \infty$ because $N^*(m,m)<\infty$, for all $m$.    This, in view of \eqref{e:l:diagonal-1}
  with $k:= m:= m_n$, implies that
  \begin{equation}\label{e:l:diagonal-2}
  \P \Big[  \max_{1\le j\le m_n} | g_n(\theta_j) - g(\theta_j)| \ge \frac{1}{m_n} \Big] \le \frac{1}{m_n},
  \end{equation}
  which entails \eqref{e:l:diagonal} since $m_n\to\infty$.
  \end{proof}

  \begin{remark}\label{rem:m(n)-explicit}  The construction of the sequence $\{m_n\}$ in Lemma \ref{l:diagonal}
  is rather implicit.  In view of \eqref{e:l:diagonal-0}, only the rate in the pointwise convergence 
  $g_n(\theta)\stackrel{\P}{\to} g(\theta)$, as $n\to\infty$, plays a role.   In this remark, we will give a conservative but
  explicit bound on $m_n$.  
  
  Recall the Ky Fan metric between for two random variables
  $$
  \rho_{KF}(\xi,\eta) := \inf\{ \epsilon>0\, :\, \P[|\xi-\eta|\ge \epsilon]\le 1/\epsilon]\}
  $$
  and suppose that
  $$
  D(n):= \sup_{\theta\in K} \rho_{KF} (g_n(\theta),g(\theta)) \to 0,\ \ \mbox{ as }n\to\infty.
  $$
  That is, $\rho(n)$ is the uniform rate of pointwise consistency in the Ky Fan metric.  Without loss of generality suppose that $\rho(n)$ is monotone non-decreasing in $n$.
  Then, Relation \eqref{e:l:diagonal-1} holds, provided $\rho(N(j,k))\le {1/2^j k}$ and in turn $\rho(n)\le 1/(2^{m_n} m_n)$ implies \eqref{e:l:diagonal-2}.
  This shows that a crude upper bound for $m_n$ can be chosen as the logarithm of the inverse of the regularly varying function $x\mapsto 1/(x\log(x))$.  
  Concretely, if $\rho(n) = {\cal O}(n^{-\gamma}),\ n\to\infty$ for some $\gamma >0$, then any sequence $m_n\to \infty$ such that
  $m_n= o(\log(n))$ will satisfy the conditions of Lemma \ref{l:diagonal}.
\end{remark}

 \begin{proof}[Proof of Proposition \ref{p:unif_cont}] Let ${\cal I}$ be a countable and dense subset of
 $K$ and let $m_n$ be the sequence obtained from Lemma \ref{l:diagonal} ensuring \eqref{e:l:diagonal} holds.

 For each $\theta\in K$, let $\zeta_n(\theta)$ be the nearest neighbor to $\theta$ among $\{\theta_i,\ i\in [m_n]\}$ in the Euclidean norm,
 for example, and define the function
 $$
  f_n(\theta):= g_n(\zeta(\theta)),\ \ \theta\in K.
 $$

First, we will argue that 
 \begin{align}\label{p:unif_cont-0}
   \|f_n - g\|_{\infty, K}\stackrel{\P}{\to} 0,\ \ \mbox{ as }n\to\infty.
 \end{align}

 Note that by the triangle inequality, we have 
 \begin{align}\label{p:unif_cont-1}
  \sup_{\theta\in K} |f_n(\theta)  - g(\theta)| &\le  \sup_{\theta\in K}  |g_n(\zeta_n(\theta)) - g(\zeta_n(\theta))| +   
  \sup_{\theta\in K} |g(\zeta_n(\theta)) - g(\theta)| \nonumber \\
  & =  \max_{1\le j\le m_n} |g_n(\theta_j) - g(\theta_j)| + \omega_{\delta_n}(g),
 \end{align}
 where 
 $$
 \omega_\delta (g) := \sup_{ \|\theta'-\theta''\|<\delta,\ \theta',\theta''\in K} |g(\theta')-g(\theta'')|
 $$ 
 stands for the  modulus of continuity of $g$ and where
 $$
  \delta_n := \sup_{\theta \in K} \min_{j\in [m_n]} \|\theta - \theta_j\|.
 $$

 By Relation \eqref{e:l:diagonal}, the first term in the right-hand side of \eqref{p:unif_cont-1} vanishes in probability, as $n\to\infty$.
 On the other hand, the fact that ${\cal I}$ is dense in the compact set $K$ implies that $\delta_n\to 0$ and hence $\omega_{\delta_n}(g)\to 0,\ n\to\infty$,
 by the uniform continuity of $g$ on the compact $K$.  This completes the proof of \eqref{p:unif_cont-0}.

 Next, to complete the proof, we will construct a {\em continuous} (which can in fact be arbitrarily smooth) modification $\wt g_n$ 
 of $f_n$ that satisfies \eqref{e:p:unif_cont}.  This will be done by simply convolving $f_n$ with a suitable compactly supported kernel, 
 whose support shrinks as $n\to\infty$.  Indeed, let $\varphi : \R^d \to [0,\infty)$ be a non-negative 
 continuous function, such that $\varphi(\theta) = 0$, for all $\|\theta\| \ge 1$ and such that $\int_{\R^d} \varphi(x) dx  = 1$.  Let
 $\epsilon_n\downarrow 0$ be arbitrary and let $\varphi_n(x):= \epsilon_n^{-1}\varphi(x/\epsilon_n).$  Define 
 $$
  \wt g_n(\theta):= \int_{K} f_n(u) \varphi_n(\theta-u) du
 $$
 and note that $\wt g_n:\R^d\to \R$ is continuous.  By the triangle inequality, for all $\theta\in K$, we obtain
  \begin{align}\label{e:p:unif_cont-2}
      |\wt g_n(\theta) - g(\theta)| &\le \int_{K} | f_n(u) - g(u) | \varphi_n(\theta-u) du + \int_{K} |g(u) - g(\theta)| \varphi_n(\theta-u) du
      \nonumber\\
      &\le \|f_n - g\|_{\infty,K} \int_{\R^d} \varphi_n(\theta-u) du + \omega_{\epsilon_n}(g) \int_{\R^d} \varphi_n(\theta-u) du 
      \nonumber\\
      &= \|f_n - g\|_{\infty,K}  + \omega_{\epsilon_n}(g),
  \end{align}
  where the last inequality follows from the facts that $\varphi_n(u-\theta) = 0$ unless $\|u-\theta\|\le \epsilon_n$ 
  and $\int_{\R^d}\varphi_n(\theta-u)du=1$. 
  In view of \eqref{p:unif_cont-0} and the fact that $\epsilon_n\downarrow 0$, the upper bound 
  in \eqref{e:p:unif_cont-2} vanishes in probability, as $n\to\infty$. Since the
  right-hand side does not depend on $\theta\in K$, the desired convergence in \eqref{e:p:unif_cont} follows.
  \end{proof}

 \begin{remark} The rate of the uniform convergence in Proposition \ref{p:unif_cont} may be arbitrarily 
 slow since the sequence $m_n$, which controls the size of the set ${\cal I}_n:=\{\theta_1,\cdots,\theta_{m_n}\}$,
 may diverge to infinity rather slowly, as $n\to\infty$.  This means that, depending on the rate of the point-wise convergence of 
 $g_n(\theta)$ to $g(\theta)$, one may not be able to have a sufficiently fine mesh of values in 
 ${\cal I}_n$ to ensure a fast rate of decay of $\delta_n$ and hence a good rate on 
 the oscillation $\omega_{\delta_n}(g)$.  Providing an optimal upper bound in \eqref{p:unif_cont-1} involves a 
 trade-off between the point-wise rate of convergence and the regularity of $g$.  
 \end{remark}

 \begin{remark} The construction of $\wt g_n$ in Proposition \ref{p:unif_cont} 
 depends on the choice of $m_n$, which in turn depends on $g$ but {\em only through} the
 rate of the marginal convergence in probability $g_n(\theta)\to^{\P} g(\theta)$ 
 (recall \eqref{e:l:diagonal-1}). 
 Thus, under minor regularity assumptions on these rates, $m_n$ can be chosen in a fashion
 completely agnostic to the knowledge of $g$.  In many cases, simple upper bounds on 
 the marginal rate of the convergence $g_n(\theta) \to^{\P} g(\theta),\ n\to\infty$ suffice for
 a conservative selection of $m_n$.  
 \end{remark}
 

\section{Examples of asymptotically optimal homogeneous predictors}\label{sec:explicit_examples}

In this section, we will illustrate the optimal homogeneous predictors of Theorem \ref{thm:final}
for two classes of spectral measures -- purely discrete and absolutely continuous.  
The spectrally discrete case can be understood in a simpler way and yet, it is interesting 
since a curious {\em perfect asymptotic precision} phenomenon emerges in this context.  On the 
other hand, the spectrally continuous case is illustrated with a simple and yet elegant 
Dirichlet-type model, which may be of independent interest.

\subsection{The spectrally discrete case: A perfect precision phenomenon}\label{sec:spec-discr_mod}

Consider the linear independent factor model:
\begin{equation}\label{e:linear-model}
Y = \sum_{i=1}^p b_i \xi_i\ \ \ \mbox{ and }\ \ \  X = \sum_{i=1}^p a_i \xi_i,
\end{equation}
where the $\xi_i$'s are independent standard $1$-Pareto random variables, i.e.,\ $\P[\xi_i> x] = 1/x,\ x\ge 1$.  

We suppose that $b_i\ge 0$ and $a_i\in \R^d,\ i=1,\cdots,p,$ where not all $b_i$'s and not all $a_i$'s are zero.
Letting $\tau(y,x):= y_+ + \|x\|$ for some (any) fixed norm $\|\cdot\|$ on $\R^d$, it follows by Proposition 7.3 in \cite{resnick2007heav} 
that $Z:= (Y,X) \in RV_1(\R_+\times \R^d,\{a_n^{(Z)}:=n\},c_Z,\tau,\sigma)$, where with $c_i:= (b_i,a_i)$, we have
$$
c_Z=\sum_{i=1}^p \tau(c_i) = \sum_{i=1}^p b_i+\|a_i\|.
$$
In this case, the angular measure $\sigma$ is {\em discrete} and takes the form:
\begin{equation}\label{e:p:linear-model}
\sigma(d\theta) = \frac{1}{c_Z} \sum_{i=1}^p \tau(c_i) \delta_{c_i/\tau(c_i)}(d\theta).
\end{equation}

\noindent We begin with a counterpart to Proposition \ref{p:lambda-via-U-Theta} and for the sake of 
completeness, we provide a proof.

\begin{proposition} \label{p:spec-discrete-lambda} 
Let $Z = (Y,X) \in  RV_1(\R_+\times \R^d,\{a_n^{(Z)}:=n\},\tau,\sigma)$, where 
the angular measure $\sigma$ is as in \eqref{e:p:linear-model}. Let also $h:\R^d\to [0,\infty)$ be a non-negative 
continuous $1$-homogeneous function. Then:\\

(i)  We have that
$$
\lim_{t\to\infty} t\P[ Y>t] = \sum_{i=1}^p b_i,\ \ \ \lim_{t\to\infty} t\P[h(X)>t] = \sum_{i=1}^p h(a_i),
$$
and consequently $h(X)$ is an asymptotically calibrated extremal predictor of $Y$ if and only if
\begin{equation}\label{e:c-discrete}
 \sigma_Y= \sum_{i=1}^p b_i = \sum_{i=1}^p h(a_i)
\end{equation}
(recall Definition \ref{def:asymp-calibration}).\\

(ii) If $h(X)$ is an asymptotically calibrated extremal predictor of $Y$, as in \eqref{e:c-discrete}, then
\begin{equation}\label{e:lambda-Y-hX-discrete}
 \lambda(Y,h(X)) = \frac{1}{\sigma_Y} \sum_{i=1}^p b_i \wedge h(a_i).
\end{equation}
\end{proposition}
\begin{proof} Introduce the following three non-negative continuous $1$-homogeneous functions on $\R_+\times \R^d$:
$$
h_Y(y,x):= y,\ \ h_X(y,x):= h(x),\ \ \mbox{ and }\ \ h_{X,Y}(y,x):=y\wedge h(x).
$$
Using that $Z:= (Y,X) \in RV_1(\R_+\times \R^d,\{a_n^{(Z)}:=n\},\tau,\sigma)$, applying Proposition \ref{p:rv-via-h} to each of 
them, we obtain, as $t\to\infty$:
$$
 t\P[ Y>t ] \to \sigma(h_Y),\ \ t \P[ h(X)> t] \to \sigma(h_{X}),\ \ \mbox{ and } \ \ t \P[ Y\wedge h(X)> t]\to \sigma(h_{X,Y}).
$$
Notice that in view of \eqref{e:p:linear-model}, and the homogeneity of the functions
$$
\sigma(h_Y) = \sum_{i=1}^p h_Y(c_i) = \sum_{i=1}^p b_i,\ \ \sigma(h_X) = \sum_{i=1}^p h(a_i),\ \mbox{ and }\ \sigma(h_{X,Y}) = \sum_{i=1}^p
b_i \wedge h(a_i).
$$
Thus, $h(X)$ is an asymptotically calibrated extremal predictor of $Y$  (cf Definition \ref{def:asymp-calibration}) if and only if 
$\sigma(h_Y) =\sigma(h_X)$, which is equivalent to
$$
\sigma_Y:= \sum_{i=1}^p b_i = \sum_{i=1}^p h(a_i),
$$
which proves \eqref{e:c-discrete}.

Now, if \eqref{e:c-discrete} holds, by Lemma \ref{lem:nonameidea}, for the asymptotic precision of the predictor $h(X)$ of $Y$, we obtain:
$$
\lambda(Y,h(X)) = \lim_{t\to\infty} t\P[Y>t| h(X)>t] = \frac{1}{\sigma_Y} \sum_{i=1}^p b_i \wedge h(a_i),
$$
which completes the proof. 
\end{proof}


The following result as an immediate but rather curious consequence of the formula in 
\eqref{e:lambda-Y-hX-discrete}. 

\begin{corollary} \label{c:spec-discrete} Adopt the setting of Proposition \ref{p:spec-discrete-lambda},
where the dimension of the predictor $X$ is $d\ge 2$.  
Suppose that:
\begin{enumerate}
 \item $a_i\not =0,\ i=1,\cdots,r$, for $r\le p$ and  $a_i = 0,\ i=r+1,\cdots,p$, if $r<p$; 
 \item $a_i/\|a_i\| \not =a_j/\|a_j\|$, for all $1\le i\not=j\le r$.
\end{enumerate}

\noindent Then, we have that:

(i) The optimal homogeneous extremal precision is:
\begin{equation}\label{e:c:spec-discrete-i}
\lambda^{\rm (opt)}_{\cal G}(Y,X) =  \frac{\sum_{i=1}^r b_i}{\sum_{i=1}^p b_i}.
\end{equation}

(ii) If $\lambda^{\rm (opt)}_{\cal G}(Y,X)>0$, i.e., $\sum_{i=1}^r b_i>0$, then 
$\lambda^{\rm (opt)}_{\cal G}(Y,X) = \lambda(Y,h^{\rm (opt)}(X))$, where an optimal {\em Oracle} predictor 
$h^{\rm (opt)}:\R^d\to \R_+$ is a continuous, non-negative $1$-homogeneous function such that
\begin{equation}\label{e:h-opt-perfect}
 h^{\rm (opt)}(a_i):= (1 + \delta_r) \cdot b_i,\ i=1,\cdots,r, \mbox{ where } 1+\delta_r = \sum_{i=1}^p b_i/\sum_{i=1}^r b_i.
\end{equation}
\end{corollary}
\begin{proof} Suppose first that $b_i=0$, for all $i=1,\cdots,r$.  Then, for all $1$-homogeneous functions $h:\R^d\to\R_+$, by
\eqref{e:lambda-Y-hX-discrete}, we have $\lambda(Y,h(X)) = 0$, since either $b_i=0$ or $h(a_i) = 0$, for all $i=1,\cdots,p$. This proves
\eqref{e:c:spec-discrete-i} in this case.

Now, if $\sum_{i=1}^r b_i>0$, by assumption, since $a_i/\|a_i\| \not = a_{i'}/\|a_{i'}\|$ for all $1\le i\not=i' \le r$, we have that one can 
define a continuous non-negative homogeneous function as in \eqref{e:h-opt-perfect}, for every choice of the $b_i$'s. 

Moreover, the choice of $\delta_r>0$ implies that $h^{\rm (opt)}(X)$ is an asymptotically calibrated extremal predictor 
of $Y$ (cf Proposition \ref{p:spec-discrete-lambda}). By \eqref{e:lambda-Y-hX-discrete}, since $b_i \wedge (1+\delta_r)b_i = b_i$, 
we also have
$$
\lambda(Y,h^{\rm (opt)}(X))  = \frac{1}{\sigma_Y} \sum_{i=1}^r b_i \wedge (1+\delta_r) b_i = \frac{1}{\sigma_Y} \sum_{i=1}^r b_i,
$$
where $\sigma_Y:=\sum_{i=1}^p b_i$. 

On the other hand, for every other asymptotically calibrated homogeneous predictor $h(X)$, by \eqref{e:lambda-Y-hX-discrete} we also 
have
$$
\lambda(Y, h(X)) = \frac{1}{\sigma_Y} \sum_{i=1}^p b_i \wedge h(a_i) =  \frac{1}{\sigma_Y} \sum_{i=1}^r b_i \wedge h(a_i) \le \frac{1}{\sigma_Y} \sum_{i=1}^r b_i = \lambda(Y,h^{\rm (opt)}(X)).
$$
This shows that $h^{\rm (opt)}$ in \eqref{e:h-opt-perfect} is indeed an optimal extremal predictor, which completes the proof of 
\eqref{e:c:spec-discrete-i} and the corollary.
\end{proof}


\begin{remark} The linear factor models in \eqref{e:linear-model} are just one instance of models
with discrete spectra. Others such as max-linear or generalized Breiman-type models can have such spectra. 
Proposition \ref{p:spec-discrete-lambda} and Corollary \ref{c:spec-discrete} apply to all cases where 
\eqref{e:p:linear-model} holds. 
\end{remark}

\begin{remark} One limitation in Corollary \ref{c:spec-discrete} is that the
$a_i$'s are assumed to be non-proportional.  Theorem \ref{thm:final} provides the solution 
optimization problem in full generality. 
\end{remark}

\begin{remark} In practice, we do not have access to the $b_i$'s and $a_i$'s so the results in 
Corollary \ref{c:spec-discrete} concern the asymptotic precision of 
the {\em oracle} predictors. Nevertheless, the spectral measure in \eqref{e:p:linear-model} can be estimated 
consistently from data. Section \ref{sec:inference} addresses the general inference methodology.
\end{remark}

\begin{remark}[The perfect precision phenomenon]\label{r:perfect_precision} 
Corollary \ref{c:spec-discrete} reveals a curious phenomenon of {\em perfect asymptotic precision}.  
Indeed, if in \eqref{e:linear-model} we have $a_i\not = 0$, for all $1\le i\le p$, 
and no two $a_i$ and $a_{i'}$ are proportional, then by \eqref{e:c:spec-discrete-i}:
$$
\lambda^{\rm (opt)}_{\cal G}(Y,X) = \lambda^{\rm (opt)} (Y,X) =1.
$$
That is, one can predict $Y$ via $X$ with an asymptotically perfect precision!  

Our estimators developed in Section \ref{sec:inference} do indeed enjoy near perfect precision in such models
(cf Figure \ref{fig:boxplot_discrete}). This surprising phenomenon can be explained in terms of the so-called 
{\em single large jump heuristic}, which means that  in \eqref{e:linear-model} the sum $X$ is extreme in norm if 
and only if {\em one and only one} of the factors $\xi_i,\ i=1,\cdots,p$ is extreme.  More precisely,
$$
\P[ \| \sum_{i=1}^p a_i \xi_i \| >t ] \sim \sum_{i=1}^p \P[ \|a_i\| \xi_i > t],\ \ \mbox{ as }t\to\infty. 
$$
Therefore, conditionally on $\{\|X\|>t\}$, as $t\to\infty$, with probability one, the direction $X/\|X\|$ is 
asymptotically associated with the direction $a_i$ of the unique factor $\xi_i$ responsible for the extremes of $Y$.
This leads to asymptotically perfect identification of $Y$.
\end{remark}


\subsection{The spectrally continuous case}

In this section, we illustrate Theorem \ref{thm:final} in the so-called {\em spectrally continuous}
case where the conditional distribution $p(du |\theta)$ has a density.  Inference
methodology for the optimal predictors  in this rather general regime will be developed in Section \ref{sec:inference}.

We consider here {\em non-negative} random vectors $Z=(Y,X)$ that are $\tau$-regularly varying with
$$
\tau(y,x) := y_+ + \sum_{i=1}^d (x_i)_+.
$$
In this case the unit sphere $S_\tau$ becomes unit simplex:
$$
S_\tau = \Delta_d:= \Big\{ w = (w_i)_{i=0}^d\, :\, \sum_{i=0}^d w_i= 1,\ w_i\ge 0,\ i=0,\cdots,d\Big\}.
$$
A natural family of probability distributions on $\Delta_d$ are the Dirichlet models.
Recall that a random vector $W = (W_0,W_1,\cdots,W_d)$
has the Dirichlet distribution with parameter vector 
$\beta = (\beta)_{i=0}^d \in (0,\infty)^{d+1}$
if $W$ takes values in $\Delta_d$ and the joint density of $W_1,\cdots,W_d$ is given by:
$$
f_{W_1,\cdots,W_{d}}(x_1,\cdots,x_{d}) =
\frac{\Gamma(\sum_{i=0}^d\beta_i)}{\prod_{i=0}^d \Gamma(\beta_i)}\cdot \Big(1-\sum_{i=1}^d x_i\Big)^{\beta_0-1}\times
\prod_{i=1}^{d} x_i^{\beta_i-1},\ \ x_i \ge 0,\ \sum_{i=1}^d x_i\le 1,
$$
In this case, we shall write $W\sim {\rm Dirichlet}(\beta)$.

Thus, in this context, it is natural to consider $\tau$-regularly varying vectors $Z$ with Dirichlet 
angular distribution \citep[see e.g. ][]{BoldiDavison2007,scheffler:stoev:2017,SabourinNaveau2014,CorradiniStrokorb2024}.

\begin{definition} A $\tau$-regularly varying random vector $Z=(Y,X)\in RV(\R_+^{d+1},\{a_n\},\tau,\sigma_Z)$
is said to be spectrally Dirichlet 
if
$$
 \frac{Z}{\tau(Z)} \, | \, \{\tau(Z) > t\} \stackrel{d}{\to} \Theta_Z \sim {\rm Dirichlet}(\beta),
$$
for some $\beta = (\beta_i)_{i=0}^{d} \in (0,\infty)^{d+1}$.
\end{definition}

The optimal homogeneous predictors in the class of spectrally Dirichlet model are particularly elegant.

\begin{proposition} Let $Z=(Y,X)=(Y,X_1,\cdots,X_d)$ be spectrally Dirichlet with parameter 
$\beta\in (0,\infty)^{d+1}$. Then:

{\em (i)} $h^{\rm (opt)}(X):= c\|X\|_1,$ is an asymptotically calibrated  (cf \eqref{e:g-calibration-lemma}) 
and optimal  homogeneous predictor for $Y$, where $c:= {\beta_0}/{(\sum_{i=1}^d \beta_i)}$.
  
{\em (ii)} The optimal asymptotic precision is given by 
\begin{equation}\label{e:p:Spec-Dirichlet}
\lambda^{\rm (opt)}_{{\cal G}} (Y,X) = \lambda(Y,h^{\rm (opt)}(X)) = 
\E\Big[ \frac{U}{\mu_U} \wedge \frac{(1-U)}{(1-\mu_U)}\Big] 
\end{equation}
where $U\sim {\rm Beta}(\beta_0, \sum_{i=1}^d \beta_i)$ and $\mu_U = \E[U] = \beta_0/(\sum_{i=0}^d \beta_i)$.
\end{proposition}
\begin{proof} The spectrally Dirichlet models are a special case of generalized Breiman models treated in Section \ref{sec:Breiman-models}.
In this case, the key idea is to observe that 
$$
\Theta_Z := (U, (1-U)\Theta) \sim {\rm Dirichlet}(\beta_0,\beta_1,\cdots,\beta_d).
$$
The neutrality property of the Dirichlet distribution \citep[see, e.g.][]{ConnorMosimann1969},
implies that $\Theta \in {\rm Dirichlet}(\beta_1,\cdots,\beta_d)$ and
$U$ are {\em independent}.  Therefore, $p_{U|\Theta}(u) = p_U(u)$ and the quantile 
$q_\alpha(\theta)={\rm const}$ in Theorem \ref{thm:final} does not depend on $\theta$, which implies claim (i).

Claim (ii) is a standard calculation, treated in more detail in Proposition \ref{p:Pareto-Dirichlet}, below.
\end{proof}

\medskip
\noindent {\bf Spectrally-Dirichlet models.}  Recall that a non-negative 
random vector 
$$
(V,W) = (V,W_1,\cdots,W_d),\ d\ge 1
$$ 
has the multivariate Dirichlet distribution with parameter vector 
$\beta = (\beta_0,\beta_1,\cdots, \beta_d)\in (0,\infty)^{d+1}$
if: (i) $V + \sum_{i=1}^d W_i = 1$ and
(ii) the joint probability density of $W_1,\cdots,W_d$ is given by:
$$
f_{W_1,\cdots,W_{d}}(x_1,\cdots,x_{d}) = \frac{\Gamma(\sum_{i=0}^d\beta_i)}{\prod_{i=0}^d \Gamma(\beta_i)}\cdot \Big(1-\sum_{i=1}^d x_i\Big)^{\beta_0-1}\times
\prod_{i=1}^{d} x_i^{\beta_i-1},\ \ x_i \ge 0,\ \sum_{i=1}^d x_i\le 1,
$$
In this case, we shall write $(V,W)\sim {\rm Dirichlet}(\beta)$.

Since the Dirichlet distribution is a probability distribution on the positive unit
simplex, it is natural to consider polar coordinates with respect to the $L^1$-norm:
$$
\tau(y,x) :=  |y| + \|x\|_1,
$$
and define a version of the Pareto-Breiman models in Section \ref{sec:Breiman-models}.

\begin{definition} The special case of Pareto-Breiman models $Z:=(Y,X):= \xi \cdot (V,W)$ in \eqref{e:Y_X_Breiman},
where $(V,W)\sim {\rm Dirichlet}(\beta)$, will be referred to as Pareto-Dirichlet$(\beta)$, for $\beta\in (0,\infty)^{d+1}$.  
\end{definition}

 The Pareto-Dirichlet models and their spectral mixtures have been considered
 due to the tractable analytical form of its spectral distribution 
 \citep[see e.g. ][]{BoldiDavison2007,scheffler:stoev:2017,SabourinNaveau2014,CorradiniStrokorb2024}.
 The following result shows that the optimal homogeneous predictor takes
 a particularly simple form for the Pareto-Dirichlet model.

\begin{proposition} \label{p:Pareto-Dirichlet} Let 
$(Y,X) = (Y,X_1,\cdots,X_d)\sim \mbox{{\rm Pareto-Dirichlet}}(\beta)$. Then:\\

{\em (i)} $h^{\rm (opt)}(X):= c\|X\|_1,$ is an asymptotically calibrated  (cf \eqref{e:g-calibration-lemma}) 
and optimal  homogeneous predictor for $Y$, where $c:= {\beta_0}/{(\sum_{i=1}^d \beta_i)}$.\\
  
{\em (ii)} The optimal asymptotic precision is given by 
\begin{equation}\label{e:p:Pareto-Dirichlet}
\lambda^{\rm (opt)}_{{\cal G}} (Y,X) = \lambda(Y,h^{\rm (opt)}(X)) = 
\E\Big[ \frac{U}{\mu_U} \wedge \frac{(1-U)}{(1-\mu_U)}\Big] 
\end{equation}
where $U\sim {\rm Beta}(\beta_0, \sum_{i=1}^d \beta_i)$ and $\mu_U = \E[U] = \beta_0/(\sum_{i=0}^d \beta_i)$.\\

{\em (iii)} Moreover, we have
$$
\lambda^{\rm (opt)}_{{\cal G}} (Y,X) = \lambda_p(Y, h^{\rm (opt)}(X)), \ \ \  \mbox{  for all }\ p \in (p_0,1),
$$
where $\lambda_p$ as in \eqref{e:lambda-p} and $p_0 = (1\vee c)/(c+1) = \mu_U\vee (1-\mu_U)$.
\end{proposition}
 \begin{proof} Note that since $\tau(V,W) = V + \|W\|_1 = 1$, in this case $\xi = \tau(Y,X) $ and 
 $(V,W) = (Y,X)/\tau(Y,X)$, are independent. Hence, \eqref{e:p:Pareto-Breiman-i} holds trivially, with 
 $$
 (U,(1-U)\Theta) := (Y,X)/\tau(Y,X) = (V,W)  \sim {\rm Dirichlet}(\beta_0,\wt \beta),
 $$ 
 where $\wt\beta = (\beta_1,\cdots,\beta_d)$. 

 Moreover, by the {\em neutrality property} of the Dirichlet distribution, 
 we have that $U = (1-\|W\|_1)$ and $\Theta:= W/\|W\|_1$ are independent 
 and such that
$$
(U, (1-U)) \sim {\rm Dirichlet}(\beta_0, \|\wt \beta\|_1)\ \ \mbox{ and } \ \ 
\Theta:= W/\|W\|_1 \sim {\rm Dirichlet}(\wt \beta),
$$
Therefore, $f_{U|\Theta}(u|\theta) = f_U(u)$, and the conditional quantile function is 
constant: $q_\beta(\theta) = {\rm const}$ in $\theta$.  This, in view of Theorem \ref{thm:final} (see also 
Corollary \ref{c:Pareto-Breiman}) implies that the function $h^{\rm (opt)} (x)= c \|x\|_1$ yields an
optimal homogeneous predictor, for a suitable calibrating constant $c$.
To calibrate the predictor, we must have
$$
\E[ U ]  = \E [h^{\rm (opt)} (W)] = c \E[ (1-U) \|\Theta\|_1 ] = c\E [ (1-U) ].
$$
Since $U\sim {\rm Beta}(\beta_0,\|\wt\beta\|_1)$, we have that 
$\E[ U ] ={\beta_0}/{(\beta_0 +\|\wt \beta\|_1)},$  and hence $c =\E[U]/\E[1-U] = \beta_0/\|\wt \beta\|_1$.
Finally, by \eqref{e:c:Pareto-Breiman-ii} and \eqref{e:Lambda-functional},  for the {\em optimal extremal precision}, we get
\begin{align}\label{e:lambda-spec-Dirichlet}
\lambda(Y,h^{\rm (opt)} (X))  &= \frac{1}{\E[U]} \E[ U\wedge (c \|W\|_1) ] = 
 \E\Big[  \frac{U}{\E[U]} \wedge \frac{(1-U)}{\E[1-U]}\Big]
\end{align}
proving  \eqref{e:p:Pareto-Dirichlet}.

To prove part (iii), observe first that for all random variables $0<\eta< C_\eta$, independent of the standard $1$-Pareto 
$\xi$, we have that $\P[ \xi\eta >t |\eta ] = \eta/t,$ whenever $t>\eta$.   Thus,
\begin{align*}
\P[ \xi \eta >t]= \E[ \P[ \xi > t/\eta | \eta] ] = \E[ \frac{\eta}{t} ] = \frac{\E[\eta]}{t}\ \ \mbox{ when $t> C_\eta$.}
\end{align*}


That is, the tail of $\xi\cdot \eta$ is precisely Pareto, beyond quantile $p_0:= 1-\E[\eta]/C_\eta$.  This implies that
$$
\P[ Y>t ]=\P[ \xi\cdot U>t] = \frac{\E[U]}{t},\ \ \P[ h^{\rm (opt)}(X) >t ] = \P[ \xi \cdot c(1-U) >t] = c \frac{\E[(1-U)]}{t},
$$
and 
\begin{equation}\label{e:lambda-spec-Dirichlet-1}
\P[ Y\wedge h^{\rm (opt)}(X) > t] = \frac{\E[ U \wedge c(1-U)]}{t},
\end{equation}
for all $t\ge \max\{1, c\}$. The calibration of $h^{\rm (opt)}(X)$ entails $\E[U]= c \E[(1-U)]$, and hence for all
$t\ge \max\{1,c\}$, with
$$
 p := 1-  \frac{\E[U]}{t} = 1- c \frac{\E[(1-U)]}{t} \ge p_0 :=  1- \frac{\E[U]}{\max\{1,c\}} = \max\{1-\E[U],\E[U]\},
$$
we obtain $t = F_Y^\leftarrow(p) = F_{h^{\rm (opt)}(X)}^{\leftarrow}(p)$.  This, in view of \eqref{e:lambda-spec-Dirichlet-1},
entails that for all $p\ge p_0$:
$$
\lambda_p( Y, h^{\rm (opt)}(X)^{\leftarrow}) = \frac{1}{\P[Y>t]} \P[ Y\wedge h^{\rm (opt)}(X) >t ] 
= \frac{\E[ U \wedge c(1-U)]}{\E[U]},
$$
which equals $\lambda_{\cal G}^{\rm (opt)}(Y,X)$, completing the proof of part {\em (iii)}.
\end{proof}

Figure \ref{fig:Spectrally Dirichlet model precision} (left)
illustrates this optimal extremal precision as a function of 
$a_0=\beta_0$ and $a_1=\|\wt \beta\|_1$, where the expectation in 
\eqref{e:p:Pareto-Dirichlet} is expressed
via incomplete Beta functions in terms of $\beta_0$ and $\|\beta\|_1$.  

The right panel of Figure \ref{fig:Spectrally Dirichlet model precision} illustrates
the finite-sample performance of the optimal (oracle) predictor and a general
estimator for the optimal homogeneous predictor implemented in Section \ref{sec:algorithm}.

\begin{figure}
    \centering
    \includegraphics[width=0.45\linewidth]{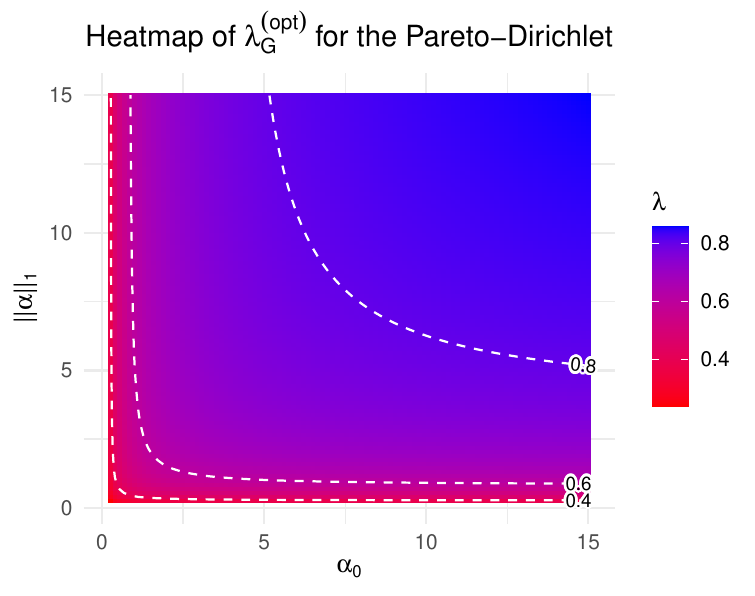}
    \includegraphics[width=0.45\linewidth]{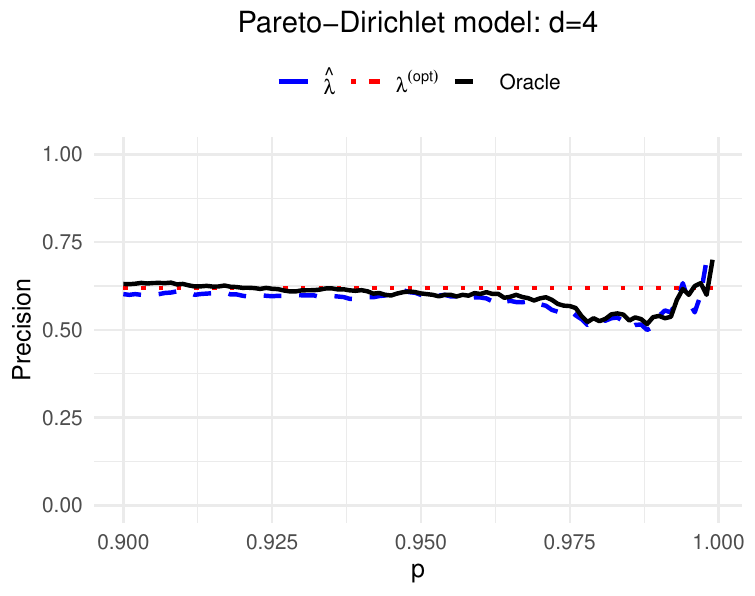}
    \caption{
    \label{fig:Spectrally Dirichlet model precision}
    {\em Left panel:} Optimal extremal precision $\lambda_{\cal G}^{\rm (opt)}$ for
    predicting one component of the Pareto-Dirichlet model via a homogeneous function of the rest \eqref{e:lambda-spec-Dirichlet}.
    {\em Right panel:} The empirical tail-dependence coefficients 
    $\hat\lambda_p$ for: (i) the asymptotically optimal 
    Oracle predictor $h^{(\rm opt )} (X) \propto \|X\|$ 
    (black solid line) (ii) a non-parametric estimator of the optimal predictor 
    $\wh Y := \wh h(X)$ based on a training and testing samples of size
    $n_{\rm train} = n_{\rm test} = 10^4$ (see Section \ref{sec:algorithm}).
    }
\end{figure}

\subsection{On the sub-optimality of homogeneous predictors} \label{sec:ex:suboptimality} 

In some cases, the optimal predictors are in fact homogeneous functions of the covariates (cf
Example \ref{ex:gumbel}).  We will demonstrate next, however, that the class of homogeneous functions
does not always contain neither optimal nor asymptotically optimal predictors.  
The spirit of this example is to show that the extremes of 
$X$ do not necessarily predict well the extremes of $Y$.  

\begin{example}\label{ex:suboptimality} 
Let $\epsilon$ and $Z$ be independent standard $1$-Pareto random variables and
$$
Y := c Z + (1-c) \epsilon,
$$
for some $c\in (0,1]$. Consider the covariate vector:
$$
X = (X_1,X_2)^\top := (e^{-YZ} + Z, Z)^\top
$$
taking values in $\R_+^2$.

Clearly, with 
$h^{\rm (opt)} (x) := -\log(x_1-x_2)/x_2,$ $x = (x_1,x_2) \in \R_+^2$, we have
$Y = h^{\rm (opt)}(X)$, almost surely.  Thus, the optimal extremal and non-extremal precisions 
are equal to $1$. On the other hand, the intuition that the extremes of $X$
(defined in terms of some homogeneous radial function $h$) will predict the extremes of $Y$ fails!

Indeed, for every continuous, $1$-homogeneous $h:\R_+^2\to \R_+$ such that $h(x)>0$, for all $x\not=0$,
it can be shown that
$$
\lambda(Y, h(X)) = \lambda(Y,Z) = c,
$$
which can be made {\em arbitrarily small}.
\end{example}

\section{Practical Implementation}

 \label{sec:algorithm}

 Many of the existing implementations of quantile regression can accommodate {\em importance weights} and thus
 achieve estimators for the quantiles of the tilted distribution in \eqref{e:q-alpha}.  Here, we opted 
 for using the versatile {\em quantile regression forest} of \cite{meinshausen2006quantile} and its
 efficient implementation in the package {\tt ranger} of \cite{wright:ziegler:2017}.
 Specifically, this can be achieved with the {\tt R} commands:
 \begin{align*}
 &\mbox{\tt  weights <- 1 - u}\\
 &\mbox{{\tt data<-data.frame(u=u, theta)}}\\
 &\mbox{\tt rf<-ranger(u$\sim$.,data=data, case.weights=weights, quantreg=TRUE)}
\end{align*}
Here ${\tt u}$ and ${\tt \theta}$ are $(n_r\times 1)$ and $ (n_r\times d)$ arrays, whose $i$th
rows correspond to $U_i$ and the vector $\Theta_i$, respectively, $i=1,\cdots,n_r$. 
The {\tt ranger} function returns a {\em random forest} {\tt S3} object, which can be used to 
generate predictions of $\wh q_\alpha(\theta_{\rm new})$ using the commands
\begin{align*}
 &\mbox{{\tt theta\_new = x\_new/L1\_norm(x\_new)}}\\
 &\mbox{{\tt pred <- predict(rf, data=theta\_new, type="quantiles", quantiles=alpha)}}\\
 &\mbox{{\tt q\_alpha <- pred\$predictions}}
\end{align*}

The following Algorithm \ref{algo:g_opt_hat} summarizes the inference.  The {\tt R code} implementing the 
estimator using the {\tt ranger} package of \cite{wright:ziegler:2017} is given in \cite{optXpred:2026}. An 
{\tt R Shiny App} demonstrating the estimator is deployed on \cite{optXpred_shiny:2026}, where 
further covariate selection heuristics are implemented.

 \begin{algorithm}[t!]
\caption{Optimal homogeneous predictor inference}\label{algo:g_opt_hat}
\begin{algorithmic}[1]
\Require 
\State {\em Training data:} Arrays $Y$ and $X$ of 
sizes $(n\times 1)$ and $(n\times d),\ d\ge 1$; 
\State {\em Threshold probability parameter:} $p_r\approx 1$;
\State {\em Binary search tolerance:} ${\epsilon}>0$.
\For{ $i=1,\cdots,n$ } 
\State  $R_i \gets Y_i + \sum_{j=1}^d X_{i,j}$ 
\EndFor
\State $r\gets {\rm quantile}(R,p_r)$ and $I \gets which(R>=r)$
\For{$i \in I$} \State 
$U_i \gets Y_i/R_i$ and $\Theta_i \gets X_i/\|X_i\|$
\EndFor \\
\State {\bf Procedure:} {\em (Implement conditional quantile)}
 \State For each $a\in (0,1)$, the function 
  ${\tt quantile}(a)$ returns an approximation of $q_\alpha(\theta)$ in \eqref{e:q-alpha},
  given the sample $(U_i,\Theta_i)$ for $i\in I$.
\State {\bf Compute constraint:}
 $$
 c\gets {\rm mean}(U)
 $$
 \State {\bf Procedure:} {\em (Calibrate $\alpha\in (0,1)$ using binary search)} 
    \State Start with $a_0\approx 0, a_0>0$ and $a_1 \gets 1-a_0\approx 1$
     \State Compute $c_j \gets  {\rm mean}( g(a_j,\Theta)),\ j=0,1$, where
    $$
    g(a,\theta) = \frac{{\tt quantile}(a)}{1-{\tt quantile}(a)}.
    $$
    \State Ensure $c_0 < c <c_1$, otherwise degrease starting value $a_0>0$. 
    \While{$a_1-a_0 > \epsilon$}
    \State Let $a_{\rm new}\gets (a_0 + a_1)/2$ and $c_{\rm new} \gets {\rm mean}(g(a_{\rm new},\Theta))$.
    \State If $c_{\rm new}< c$, then $c_0\gets c_{\rm new}$ and $a_0\gets a_{\rm new}$.
     Otherwise, $c_1\gets c_{\rm new}$ and $a_1 \gets a_{\rm new}$.
    \EndWhile
    \State Return calibrated $\alpha\gets a_{\rm new}$
\State {\bf Procedure:} (Optimal homogeneous predictor)
\State Input $X_{\rm new}$ -- an array of size $(1\times d)$ featuring a {\em new} vector of covariates.
\State Compute $\Theta_{\rm new}\gets X_{\rm new}/\|X_{\rm new}\|$
\State \Return
$
\wh Y_{\rm new} := h_{\rm opt}(X_{\rm new}) \gets  \|X_{\rm new}\| g(\alpha,\Theta_{\rm new}).
$
\end{algorithmic}
\end{algorithm}

Note that the sequence $(\alpha_n)_{n\in \mathbb{N}}$ provided by Algorithm \ref{algo:g_opt_hat} fulfill the assumption of theorem \ref{thm:univ_consitency} by application of lemma \ref{lem:conv_alpha}.

\begin{lemma}\label{lem:conv_alpha}
If for any $\theta \in S^d$ the quantile estimator $\hat q_{\alpha_n,n}$ is uniformly consistent on compacts of $(0,1)$, then the sequence $(\alpha_n)_{n\in \mathbb{N}}\subset (0,1)$ defined by the algorithm \ref{algo:g_opt_hat} satisfies
\[
\alpha_n \underset{n\rightarrow\infty}{\overset{\mathbb{P}}{\longrightarrow}}\alpha_{opt}.
\]
\end{lemma}

\begin{proof}[Proof of Lemma \ref{lem:conv_alpha}]
For $n \geq 2$ we can write, 
\begin{align*}
\left\{|\alpha_{opt}-\alpha_n|\leq \frac{1}{2^{n-1}}\right\}=&\left\{|\alpha_{opt}-\alpha_{n-1}|\leq \frac{1}{2^{n-2}}\right\}\bigcap \left\{\exists \ l_0>0 | \forall l>l_0,\frac{q(\alpha_{opt})-\hat q_l(\alpha_n)}{q(\alpha_{opt})-q(\alpha_n)}>0\right\}\\
=&\left\{|\alpha_{opt}-\alpha_{2}|\leq 1\right\}\bigcap\left(\bigcap_{k=2}^n \left\{\exists \ l_0>0 | \forall l>l_0,\frac{q(\alpha_{opt})-\hat q_l(\alpha_k)}{q(\alpha_{opt})-q(\alpha_k)}>0\right\}\right)\\
=&\bigcap_{k=2}^n \left\{\exists \ l_0>0 | \forall l>l_0,\frac{q(\alpha_{opt})-\hat q_l(\alpha_k)}{q(\alpha_{opt})-q(\alpha_k)}>0\right\}.\\.
\end{align*}
In the previous equalities, since the number of intersections is finite, we can always use the always integer $l_0$ in each sets. Hence, taking the complementary in the previous equality, we get 
\[
\mathbb{P}\left(|\alpha_{opt}-\alpha_n|\geq \frac{1}{2^{n-1}}\right)\leq\sum_{k=2}^n\mathbb{P}\left(\exists \ l_0>0 | \forall l>l_0,\frac{q(\alpha_{opt})-\hat q_l(\alpha_k)}{q(\alpha_{opt})-q(\alpha_k)}<0\right).
\]
Noticing that, for all $k$, 
\[
\left\{\exists \ l_0>0 | \forall l>l_0,\frac{q(\alpha_{opt})-\hat q_l(\alpha_k)}{q(\alpha_{opt})-q(\alpha_k)}<0\right\}\subset \left\{\exists \ l_0>0 | \forall l>l_0,|\hat q_l(\alpha_k)-q(\alpha_k)| >|q(\alpha_k)-q(\alpha_{opt})|\right\},
\]
since $\hat q_l$ converge uniformly in probability on $(0,1)$, there exists $l_n\in \mathbb{N}$ such that, for all $l >l_n$
\[
\mathbb{P}\left(|\hat q_l(\alpha_k)-q(\alpha_k)| >|q(\alpha_k)-q(\alpha_{opt})|\right)\leq \frac{1}{n(n-2)},\quad \forall \ 2\leq k \leq n.
\]
This leads to 
\[
\mathbb{P}\left(|\alpha_{opt}-\alpha_n|\geq \frac{1}{2^{n-1}}\right)\leq \frac{1}{n}.
\]
\end{proof}

\section{An application to Solar flare prediction} \label{sec:solar-flares}

In this section, we apply the general optimal homogeneous prediction methodology described in 
the previous section to an important open problem from space physics
\citep[see, e.g.][ and the references therein]{bobra2015sola,verma:stoev:chen:2024}.  

In the following, the responses $\{Y(t),\ t=1,2,\cdots\}$ come from a curated X-ray flux time series derived from 
a collection up to 15 different GOES satellite measurements (Figure \ref{fig:flux}). The covariates 
$X_1(t),\cdots,X_d(t),\ t=1,2,\ldots$ are time series of $d=22$ statistics referred to as {\tt SHARP} 
time series. They are derived from solar disk images and capture fundamental physical characteristics of the active 
regions (sun spots) which are believed to drive the formation of solar flares  \citep[][]{bobra2014theh,Bobra2021}.  

The objective is to predict {\em extreme} flux events $\{Y(t+h)>y_0\}$ at some future time period $t+h$, 
based on $\{X_i(t),\ i=1,\cdots,d\}$. The events when the X-ray flux exceeds high thresholds are known as Solar 
flares and they are classified in several severity classes \cite{NOAA_GOES_Xray}.  We will focus on predicting the most extreme 
M- and X-class flares, which correspond to X-ray fluxes exceeding $10^{-5}\ W/m^2$ and $10^{-4}\ W/m^2$, respectively (Figure \ref{fig:flux}). \\

\begin{figure}
    \centering
    \includegraphics[width=0.45\linewidth]{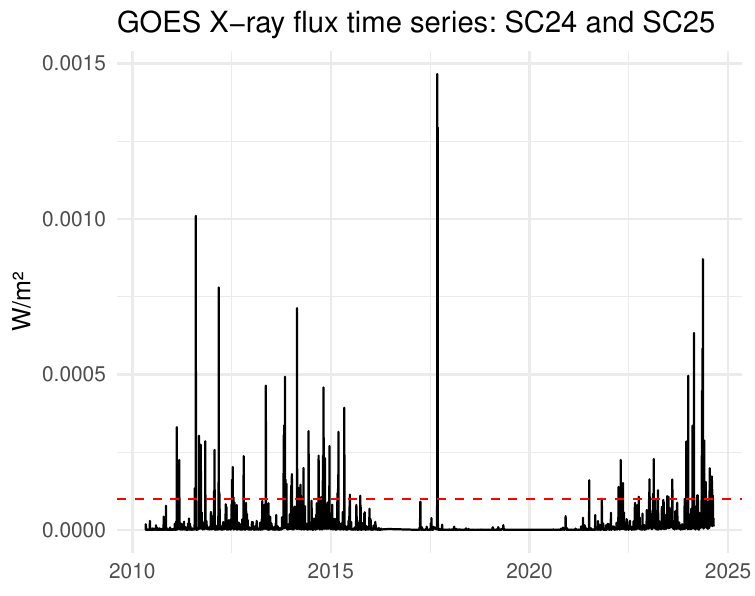}
    \caption{Time series of maximum daily X-ray flux (in Watts per square meter).  The data is based on the highest quality measurements obtained from merging observations from 
    the GOES satellites \citep[][]{NOAA_GOES_Xray}. It spans two solar cycles (SC24 \& SC25). The horizontal line marks the X-class flare threshold ($10^{-4}\ W/m^2$).
    }
    \label{fig:flux}
\end{figure}

{\bf Standardization.} For both the response and the predictors, we consider 
non-overlapping block maxima over a window $w$ corresponding to $24$ hours:
$$
Y^{(w)} (t):= \max_{ w(t-1)+1 < s\le wt} Y(s)\ \ \mbox{ and }\ \ X_{j}^{(w)}(t):=
\max_{w(t-1)+1 < s\le wt} X_j(s).
$$
We then construct estimates $\wh F_{Y^{(w)}}$ and $\wh F_{X_j^{(w)}}$ 
of $Y^{(w)}(t)$ and $X_j^{(w)}(t)$ based on a {\em training sample} 
$t\in {\cal D}_{\rm train}$.  Tail extrapolation is used to ensure that 
the CDFs have unbounded supports and define the training sample with standardized 
approximately $1$-Pareto marginals:
$$
\wt Y(t) := \frac{1}{1-\wh F_{Y^{(w)}}(Y^{(w)}(t))} \ \ \mbox{ and }\ \ 
\wt X_j(t) := \frac{1}{1-\wh F_{X_j^{(w)}}(X_j^{(w)}(t))},
$$
for all $t\in {\cal D}_{\rm train}$, and $j=1,\cdots,d$.

Then, under the assumption that the so-standardized vectors
$(\wt Y(t+h), \wt X_1(t),\cdots, \wt X_d(t)), t, t+h\in {\cal D}_{\rm train}$ are 
realizations from a jointly regularly varying vector $(Y,X)$, we apply the inference
methodology of Section \ref{sec:inference}.  Specifically, we select the threshold
$r$ as a high quantile of the sample of norms, and apply the quantile random forest
to estimate the optimal homogeneous predictor $\wh h$.\\

{\bf Prediction.} Given a test sample ${\cal D}_{\rm test}$, we compute
$$
\wh Y(t+h) := \wh h( \wt X(t)),\ t, t+h\in {\cal D}_{\rm test},
$$
where $\wt X(t) = (\wt X_j(t),\ j=1,\cdots,d)$ are standardized as:
$$
\wt X(t): =\frac{1}{\wh F_{X_j^{(w)}} (X_j^{(w)}(t))},\ t\in {\cal D}_{\rm test}.
$$
Recall that $ \wh F_{X_j^{(w)}} $ is an estimate of the CDF of $X_j^{(w)}$ obtained 
from the training sample.

We predict a certain type of flare whenever $\wh Y(t+h)$ exceeds a threshold. Specifically,
we let $\hat p := \wh F_{Y^{(w)}}(u)$ and produce the binary predictors:
$$
I_t(h):= I\{\wh Y(t+h) > 1/(1-\wh p) \}.
$$
This choice, under the assumption of stationarity, ensures the approximate 
calibration $\E[ I_t(h)] \approx \P[ Y^{(w)}(t+h)> u]$.  More precisely, 
for the  M- and X-class flares, we use the thresholds $u=10^{-5}W/m^2$ and 
$u = 10^{-4} W/m^2$, respectively.  The X-class flares are among the most extreme and rare
events, which can lead to major magnetic storms. While less severe, the M-class flares can 
also cause major disruption of Earth's ionosphere and lead to satellite and electric outages.\\

\begin{figure}
    \centering
    \includegraphics[width=0.45\linewidth]{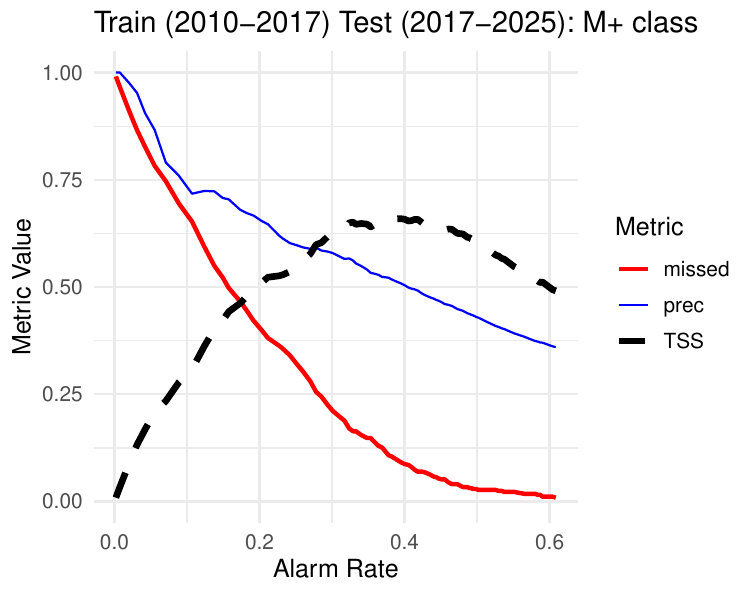}
    \includegraphics[width=0.45\linewidth]{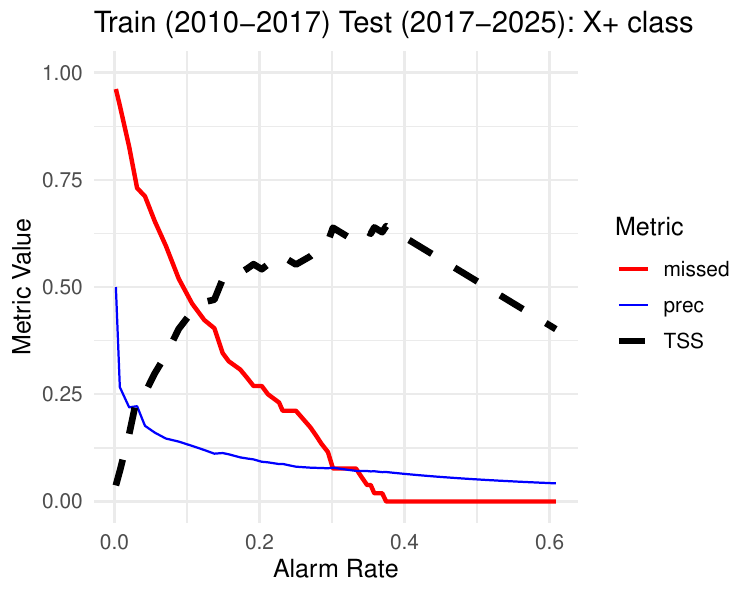}
    \caption{Evaluation metrics for the prediction of M-class (left panel) and 
    X-class (right panel) solar flare events.  The homogeneous predictor was 
    based on 24-hour block-maxima of the SHARP time series covariates
    {\tt TOTUSJH}, {\tt AREA\_ACR},  {\tt ABSNJZH}, {\tt SAVNCPP},
    {\tt TOTPOT}, and a 24-hour lagged flux. The response is the 24-hour 
    block-maximum flux time series.  We used the period of 2010--2017 
    with $n_{\rm train} =1,941$ observations for training 
    and tested on the period 2017-2025 with $n_{\rm test} = 2,018$ observations. 
    The radius threshold was chosen as the $0.97$-th quantile of the empirical distribution and
    the quantile random forest was built using $2000$ trees, i.e., the {\tt ranger} 
    function was called with parameter {\tt num.trees=2000}. }
    \label{fig:solar-flare-prediction} 
  
\end{figure}

{\bf Results.} Figure \ref{fig:solar-flare-prediction} shows a simple, nearly
tuning-free application of the optimal homogeneous prediction
methodology to Solar flare prediction.  Specifically, the homogeneous predictor 
was trained over the period (2010-2017) and applied to forecasting M- and X-class
flares over the period (2017-2025).  Confusion matrices for the binary forecasts 
were used to obtain various metrics displayed on the plots.  Specifically,
letting TP, FP, TN, FN denote true-positive, false-positive, true-negative, and 
false-negative counts of the predictions, we obtain the standard metrics:
$$
{\rm prec} = \frac{TP}{TP+FN}, \ \ \text{TSS} = 
\frac{\mathrm{TP}}{\mathrm{TP} + \mathrm{FN}} - 
\frac{\mathrm{FP}}{\mathrm{FP} + \mathrm{TN}}, \ \mbox{ and }\ 
{\rm missed} = \frac{FN}{TP+FN}.
$$
These metrics depend on the chosen quantile $\tau_p$ for the variable 
$\wh Y=h(\wt X)$, which ultimately determines the alarm rate:
$$
 {\rm alarm} = \frac{{\rm TP} + {\rm FP }}{{ n_{\rm test}}},
$$
where $n_{\rm test}$ is the size of the test sample.

As the alarm rate decays, generally, the precision increases but so does the 
proportion of missed events. Observe that the popular TSS statistic can be 
non-informative or at best misleading.  Indeed, very high values of TSS may 
be associated with quite poor precision and hence quite high false-alarm rate.  
On the other hand, well-calibrated predictors with alarm rate equal to the 
event rate may have quite low TSS values. Nevertheless, the peak of the TSS
values seem to be aligned with an elbow in the missed rate, so it happens to
be a good heuristic in this case.

In principle, however, in addition to TSS, we advocate for examining 
the precision and missed metrics, as a function of the alarm rate. The
practitioner may choose to settle for an acceptable alarm rate that 
achieves good precision and low missed rate that may or may not correspond to 
high TSS.  Optimizing a flare prediction method for achieving high TSS without 
considering calibration or alarm rates could be misleading.
This said, the TSS values achieved by the generic optimal homogeneous predictor 
methodology are surprisingly high.  While the best M-class flare forecasting machines
claim to achieve TSS rates as high as $0.7$ (or sometimes $0.8$ in carefully 
designed studies), our method reaches $0.66$ with 
{\em hardly any tuning} at all, should one look at a suitable 
alarm rate of about $30\%$.  At the same time the existing
X-class flare forecasting methods barely reach TSS as high as $0.5$ in 
operational settings (i.e., without off-line post-processing of the data, which 
can often induce hard-to-control leakage of information from the testing 
data into the training data).  
The maximum TSS values for our method of predicting X-class flares is $0.64$, 
comfortably surpassing the coveted state-of-the-art threshold, should one 
contend with alarm rate of $40\%$.  While we do not claim to have achieved
close to optimal forecasts, this brief analysis demonstrates the great potential of
the optimal homogeneous prediction methodology in the solar flare forecasting 
challenge.

\end{document}